\documentclass[11pt,reqno]{amsart}
\usepackage[utf8]{inputenc}
\usepackage[T1]{fontenc}
\usepackage[linktocpage=true,colorlinks=true,hyperindex,pagebackref=false, citecolor=cyan,pdfpagelabels,linkcolor=magenta,]{hyperref}
\usepackage{articlemacro}
\usepackage{color,soul}
\usepackage{geometry}

\usepackage{tikz}
\usetikzlibrary{arrows,positioning,decorations.markings}
\tikzset{>=angle 90}

\usepackage{tikz-cd}
\usepackage{stmaryrd}

\usepackage{hyperref}

\usepackage{url}

\makeatletter
\def\@tocline#1#2#3#4#5#6#7{
    \begingroup 
    \ifnum #1>\c@tocdepth 
  \else
    \@ifempty{#4}{}{}

    \parindent\z@ \leftskip#3\relax \advance\leftskip\@tempdima\relax
    #5\hskip-\@tempdima
      \ifcase #1
       \or\or \hskip 2em \or \hskip 1em \else \hskip 3em \fi%
      #6\nobreak\relax
    \dotfill\hbox to\@pnumwidth{\@tocpagenum{#7}}\par
    \nobreak
    \fi
    \endgroup
 }
 \def\l@section{\@tocline{1}{0pt}{1pc}{}{}}

\renewcommand{\tocsection}[3]{%
  \indentlabel{\@ifnotempty{#2}{\makebox[1.3em][l]{%
    \ignorespaces#1 \bfseries{#2}.\hfill}}}\bfseries{#3}
    \vspace{0pt}}

\renewcommand{\tocsubsection}[3]{%
  \indentlabel{\@ifnotempty{#2}{\hspace*{-0.5em}\makebox[2.1em][l]{%
    \ignorespaces#1#2.\hfill}}}#3
    \vspace{0pt}}

\makeatother 

\makeatletter 
\newcommand\appendix@section[1]{%
  \refstepcounter{section}%
  \orig@section*{Appendix \@Alph\c@section. #1}%
}
\let\orig@section\section
\g@addto@macro\appendix{\let\section\appendix@section}
\makeatother

\newcommand{\PPerf}{\textup{Perf}}
\DeclareMathOperator{\MMod}{Mod}
\DeclareMathOperator{\QQCoh}{QCoh}
\newcommand{\cn}{\textup{cn}}
\DeclareMathOperator{\Equiv}{Equiv}
\DeclareMathOperator{\cofib}{cofib}
\newcommand{\cl}{\textup{cl}}
\DeclareMathOperator{\fib}{fib}
\DeclareMathOperator{\Poly}{Poly}
\DeclareMathOperator{\N}{N}
\newcommand{\SCRMod}{\textup{AR-Mod}}
\DeclareMathOperator{\CAlg}{CAlg}

\newcommand{\dSt}{\textup{dSt}}
\newcommand{\Shv}{\textup{Shv}}

\DeclareMathOperator{\dSch}{dSch}
\newcommand{\hSp}{\textup{hSp}}
\DeclareMathOperator{\Ch}{Ch}

\DeclareMathOperator{\Ani}{\textup{Ani}}
\newcommand{\AniAlg}[1]{\textup{AR}_{#1}}
\newcommand{\AniM}{\SCRMod^{\cn}}
\newcommand{\AniMod}[1]{\textup{Ani($#1$-Mod)}}

\usepackage{relsize}
\usepackage[bbgreekl]{mathbbol}
\usepackage{amsfonts}
\DeclareSymbolFontAlphabet{\mathbb}{AMSb} 
\DeclareSymbolFontAlphabet{\mathbbl}{bbold}

\title{Notes on derived algebraic geometry}
\author{Can Yaylali}

\begin{document}

\begin{abstract}
	These are notes on derived algebraic geometry in the context of animated rings. More precisely, we recall the proof of To\"en-Vaqui\'e that the derived stack of perfect complexes is locally geometric in the language of $\infty$-categories. Along the way, we recall the necessary notions in derived commutative algebra and derived algebraic geometry. We also analyze the deformation theory and quasi-coherent modules over derived stacks.
\end{abstract}
\maketitle

\tableofcontents

\section*{Introduction}
Derived algebraic geometry is the study of algebraic geometry using homotopy theoretical methods. It was used by Arinkin-Gaitsgory \cite{ArinkinGaits} to formulate a geometric Langlands conjecture. T\"oen-Vaqui\'e used derived algebraic geometry to study moduli of dg-categories  in \cite{TVaq}. In \cite{AG} Antieau-Gepner studied derived versions of Brauer groups and Azumaya algebras. More recently, Annala constructed in \cite{Annala} a bivariant derived algebraic cobordism for quasi-smooth projective derived schemes. In \cite{yaylali}, the author defined derived versions of $F$-zips and used these to analyze families of proper smooth morphisms, which was previously not possible in this generality with classical $F$-zips.
\par  
A general theory using certain model categories (for example simplicial commutative rings, DG-rings and connective $E_{\infty}$-rings) was developed by To\"en-Vezzosi in \cite{TV2}. The idea is to study sheaves of $\infty$-groupoids on these model categories. They also look more closely into the 3 settings listed above. \par
	Before working with model categories, one could ask how to extend algebraic geometry to sheaves of rings that take values in $n$-groupoids, for $n\geq 2$. The idea of Simpson in \cite{Simpson}, was to extend the definition of an atlas. Namely, he first starts with $0$-geometric stacks, which are algebraic spaces. The $0$-geometric morphisms are morphisms that are representable by algebraic spaces. Then an $n$-geometric stack, is an \'etale sheaf of $(n+1)$-groupoids on rings together with a smooth $(n-1)$-geometric morphism, by a $(n-1)$-geometric stack. It turns out, that if we start sheaves with values in $\infty$-groupoids, then an $n$-geometric sheaf is automatically $(n+1)$-truncated. The definition of To\"en-Vezzosi of $n$-geometric stacks replaces the source category by a ``nice'' model category with a suitable topology (where ``nice'' is in their sense, for example simplicial commutative rings with the \'etale topology).\par
The case of simplicial commutative rings is also treated by Lurie in his PhD-thesis \cite{DAG} in the language of $\infty$-categories. The benefit of working with $\infty$-categories is  the fact, that we do not need to specify a model structure on animated rings (the $\infty$-categories of simplicial commutative rings\footnote{The term ``simplicial commutative ring'' in the $\infty$-categorical language was replaced by \v{C}esnavi\v{c}ius-Scholze in \cite{CS} by the term ``animated ring''. This makes in particular clear, that one works with objects up to homotopy.}). Besides this setting, there is a lot of work by Lurie on spectral algebraic geometry, i.e. algebraic geometry using $E_{\infty}$-rings (see \cite{SAG}). Even though in positive characteristic there is no equivalence between animated rings and $E_{\infty}$-rings, morally, geometry in both settings should behave the same. \par 
For example, in both settings one can look at the sheaf of perfect complexes, i.e. we can look at the functor $A\mapsto (\MMod_{A}^{\perf})^{\simeq}$, that maps an animated ring (resp. an $E_{\infty}$-ring) $A$ to the largest $\infty$-groupoid in perfect $A$-modules. In \cite{TVaq} To\"en-Vaqui\'e show that this stack is locally geometric in the simplicial commutative ring setting (in the language of model categories) and in \cite{AG} Antieau-Gepner show local geometricity in the spectral setting. Here one should remark, that the definition of $n$-geometric stacks in \cite{TV2}, \cite{AG} and \cite{DAG} differ. Most notably, in \cite{TV2} separated Artin stacks are $0$-geometric, in \cite{AG} a $0$-geometric stacks are disjoint unions of affines and in \cite{DAG} the $0$-geometric stacks must have an \'etale atlas.\par
In this article, we focus on derived algebraic geometry in the context of animated rings. The most notable difference is, that for animated rings in positive characteristic, we have a Frobenius endomorphism. This does not hold for $E_{\infty}$-rings, as we would need to find a homotopy coherent Frobenius, whereas for animated rings, we can use the classical Frobenius to induce one on animated rings. Throughout, we will omit that we work in the context of animated rings.\par
Another benefit of animated rings and so, derived algebraic geometry is the naturally occurring deformation theory. There is a natural notion of square-zero extensions and thus derivations. Classically, derivations are represented by the module of K\"ahler-differentials. In the context of derived algebraic geometry, derivations are represented by the cotangent complex.
\par 
The goal of these notes is to recall important aspects of \cite{TV2} and \cite{SAG} in terms of animated rings and prove that the stack of perfect complexes is locally geometric following the proofs in \cite{TVaq} and \cite{AG}. 

\subsubsection*{Derived commutative algebra}
The $\infty$-category of animated rings $\AniAlg{\ZZ}$ is given by polynomial algebras, where we freely adjoin sifted colimits of those. Looking at over categories for any animated ring $A$, we can define the $\infty$-category of animated $A$-algebra $\AniAlg{A}\coloneqq (\AniAlg{\ZZ})_{A/}$. The benefit of this definition is that a lot of questions about functors from $\AniAlg{\ZZ}$ to $\infty$-categories with sifted colimits can be reduced to polynomial algebras. Animated rings should be thought as spectral commutative rings (i.e. $E_\infty$-rings) with extra structure. Especially, after forgetting this extra structure, we can also define modules over animated rings (as modules over the underlying $E_{\infty}$-ring). One important example of such a module is the cotangent complex. This module arises naturally if we want to define an analogue of the module of differentials as the module that represents the space of derivations.\par 
By left Kan-extension of the classical square-zero extension, we can associate to any connective $A$-module, the square-zero extension $A\oplus M$ of $A$ by $M$. In this way, we can define $A$-linear derivations by $M$, as factorizations $A\rightarrow A\oplus M\rightarrow A$ of the identity of $A$. As in the classical case, there is a connective $A$-module $L_{A}$ representing the space of derivations, called \textit{cotangent complex}. If $A$ is discrete, then this is the classical cotangent complex of Illusie. We can also extend this definition to a relative version $L_{B/A}$ for animated ring homomorphisms $A\rightarrow B$.\par 
The underlying $E_{\infty}$-ring of an animated ring is a commutative algebra object in spectra. Thus, we can define homotopy groups of animated rings and automatically see (with the theory developed in \cite{HA}) that we can associate to every animated ring $A$ a $\NN_{0}$-graded ring $\pi_{*}A$. Using this, we reduce definitions like smoothness of a morphism $A\rightarrow B\in\AniAlg{\ZZ}$ to smoothness of the ordinary rings $\pi_0A\rightarrow \pi_0B$ together with compatibility of the graded ring structure, i.e. $\pi_*A\otimes_{\pi_0A} \pi_0B\cong \pi_*B$. \par 
A benefit of naturally occurring deformation theory is, that we can work with the cotangent complex to relate its properties to geometric properties of the underlying animated rings. 
\begin{prop}[\protect{\ref{smooth cotangent}}]
	Let $f\colon A\rightarrow B$ is a morphism of animated rings that is finitely presented in $\pi_{0}$, then $L_{B/A}$ is finite projective if and only if $f$ is smooth. 
\end{prop}
 This cannot be achieved using the module of K\"ahler-differentials, as $\Omega_{B/A}^{1}=0$ whenever $f$ is a closed immersion.\par 
One can also show that \'etale maps of animated rings $A\rightarrow B$ are equivalently given by \'etale maps $\pi_{0}A\rightarrow \pi_{0}B$ of (classical) rings. So, one can think of animated rings as nilpotent thickenings into the homotopy direction of (classical) rings.

\subsubsection*{Derived algebraic geometry}
Defining \textit{derived stacks} is rather straightforward now, we set them as presheaves (of spaces) on $\AniAlg{\ZZ}$ which satisfy \'etale descent. One important example are \textit{affine derived schemes}, which we set as representable presheaves on $\AniAlg{\ZZ}$. We can also define relative versions, where we replace $\ZZ$ with an animated ring. We will see that they naturally satisfy fpqc-descent. For affine derived schemes it is easy to define properties by using their underlying animated rings. To do the same for derived stacks, we will need the notion of \textit{$n$-geometric} morphisms. This notion is defined inductively, where we say a morphism $f\colon F\rightarrow G$ of derived stacks is
\begin{enumerate}
	\item[•] $(-1)$-geometric, if the base change with an affine derived scheme is representable by a an affine derived schemes. A $(-1)$-geometric morphism is smooth if it is so after base change to any affine.
	\item[•] The morphism $f$ is $n$-geometric, if for any affine derived scheme $\Spec(A)$ with morphism $\Spec(A)\rightarrow G$ the base change $F\times_G \Spec(A)$ has a smooth $(n-1)$-geometric effective epimorphism by $\coprod \Spec(T_i)$, where an $n$-geometric morphism is smooth if after affine base change the induced maps of the atlas to the base is $(-1)$-geometric and smooth.
\end{enumerate}
We can also define \textit{locally geometric} derived stacks, as those derived stacks $X$, that can be written as the filtered colimit of open substacks.\par
For a good class\footnote{With ``good class'' we mean stable under base change, composition, equivalences and smooth local on source and target.} of properties \textbf{P} of affine derived schemes, e.g. smooth, flat,\dots\footnote{Note that the property \'etale is not smooth local on the source ! We have to be careful if we want to define \'etale morphisms of $n$-geometric stacks.}, we can now say that a morphism of derived stacks has property $\pbf\in\Pbf$ if it is $n$-geometric for some $n$ and after base change with an affine derived scheme the atlas over the affine base has property $\pbf$. Since geometricity is defined by using smooth atlases, we are mostly interested in this property. 
\par
We would like to relate smoothness to deformation theory, as in the affine case. But before that, we need to define quasi-coherent modules on derived stacks. As modules over animated rings satisfy fpqc descent, we can define the $\infty$-category quasi-coherent modules $\QQCoh(X)$ on an $n$-geometric derived stack $X$ via gluing along smooth atlases, i.e. via right Kan-extension of $A\mapsto \MMod_{A}$. Even though not clear, we will see that using the works of Lurie, one has that for classical schemes, the quasi-coherent modules are equivalently given by complexes in the derived category with quasi-coherent cohomology.
\begin{prop}[\protect{\ref{derived cat is kan ext}}]
	Let $X$ be a scheme. Then we have an equivalence of $\infty$-categories $\Dcal_{\textup{qc}}(X)\simeq \QQCoh(X)$, where $\Dcal_{\textup{qc}}(X)$ denoted the derived $\infty$-category of $\Ocal_{X}$-modules with quasi-coherent cohomologies.
\end{prop}
\par 
Now we can define the cotangent complex as in the affine case. So, we define it as the quasi-coherent module that represents derivations. We can also use the atlas of $n$-geometric derived stacks to see, that any $n$-geometric morphisms of derived stacks has a cotangent complex. In particular, as in the affine case, we can relate geometric properties of derived stacks to properties of the cotangent complex.

\begin{thm}[\ref{cotangent implies smooth}]
	Let $f\colon X\rightarrow Y$ be an $n$-geometric morphism of derived stacks. Then $f$ is smooth if and only if $f_{|\Ring}$ is locally of finite presentation and the cotangent complex of $f$ is perfect and has Tor-amplitude in $[-n-1,0]$.
\end{thm}

The proof of this theorem uses a lot of the homotopical algebraic geometric methods combining ideas of the homotopy theoretical world with ideas of the algebraic geometric world.
We can extend these results even further, to see that any $n$-geometric stack $X$ is nilcomplete, meaning that if we know the values on truncated animated rings, we know the value on all animated rings, i.e. $X\simeq \lim_{n} X\circ\tau_{\leq n}.$ This also shows that $n$-geometric derived stacks are automatically hypercomplete.

\subsubsection*{Geometricity of perfect complexes}

Let us state the main theorem of this article.

\begin{theorem}[\ref{main thm}]
	The derived stack
	\begin{align*}
	\PPerf\colon \AniAlg{R} &\rightarrow \SS\\
	A&\mapsto (\MMod_A^{\textup{perf}})^{\simeq}
	\end{align*}
	is locally geometric and locally of finite presentation.
\end{theorem}

The proof is separated into different steps. First, write this stack as a filtered colimit of substacks $\PPerf^{[a,b]}$, where we fix the Tor-amplitude of the perfect complexes by $a\leq b\in \ZZ$. Per definition, it is enough to show that these are $(b-a+1)$-geometric stacks locally of finite presentation and open in $\PPerf$. \par The proof is by induction over $n\coloneqq b-a$. For $n=0$, this is an interesting and well known stack $\PPerf^{[0,0]}\simeq \Proj\simeq\coprod_n \BGL_n$.\par The idea for higher $n$ is straightforward, knowing that any perfect module $M$ of Tor-amplitude in $[a,b]$ can be written as the fiber of $Q\rightarrow M$, where $Q$ is a perfect module of Tor-amplitude $[a+1,b]$ and $M$ is perfect of Tor-amplitude in $[a+1,a+1]$. Inductively, this reduces to the question, if the stack classifying morphisms between perfect modules is geometric and smooth. But this this follows from explicit calculations of the cotangent complex.\par 
For the openness, we have to consider the vanishing locus of perfect complexes. But since in classical algebraic geometry this is open, we can extend this to the derived world.

\subsection*{Structure of this article}
We start by summarizing \cite{HA} (see Section \ref{HA}). We try to show that spectral rings and modules behave in some sense like expected.\par 
The next step is the introduction of derived commutative algebra, i.e. the theory  of algebras over animated rings (see Section \ref{sec:derived commutative algebra}). We first define animated rings $\AniAlg{\ZZ}$, show that there is a relation to $E_\infty$-rings and use this relation to define modules over animated rings. As a consequence, we can define the cotangent complex and define what properties of morphisms in $\AniAlg{\ZZ}$ are. We end this section by showing that the cotangent complex is highly related to smoothness of morphisms.\par 
After talking about derived commutative algebra, we introduce the theory derived algebraic geometry (see Section \ref{sec:der.alg.geo}). Mainly, we introduce the notion of derived stacks, geometricity of morphisms and derived schemes. We also talk about truncation of those and how it relates to classical algebraic geometry. Further, we again cover smoothness of such morphisms and how it relates to the cotangent complex.\par 
We finish the summary on derived algebraic geometry by showing that the stack of perfect modules is locally geometric and locally of finite presentation (see Section \ref{sec:perf}).\par

\subsection*{Assumptions}
All rings are commutative with one. \par
We work with the Zermelo-Frenkel axioms of set theory with the axiom of choice and assume the existence of inaccessible regular cardinals.\par
Throughout this paper we fix some uncountable inaccessible regular cardinal $\kappa$ and the collection $U(\kappa)$ of all sets having cardinality $<\kappa$, which is a Grothendieck universe (and as a Grothendieck universe is uniquely determined by $\kappa$) and hence satisfies the usual axioms of set theory (see \cite{universe}). When we talk about small, we mean $\Ucal(\kappa)$-small. In the following, we will use some theorems, which assume smallness of the respective ($\infty$-)categories. When needed, without further mentioning it, we assume that the corresponding ($\infty$-)categories are contained in $\Ucal(\kappa)$.\par
If we work with families of objects that is indexed by some object, we will assume, if not further mentioned, that the indexing object is a $\Ucal(\kappa)$-small set.

\subsection*{Notation}
We work in the setting of $(\infty,1)$-categories (see\cite{HTT}). By abuse of notation for any $1$-category $C$, we will always denote its nerve again with $C$, unless otherwise specified.\par
A \textit{subcategory} $\Ccal'$ of an $\infty$-category $\Ccal$ is a simplicial subset $\Ccal'\subseteq\Ccal$ such that the inclusion is an inner fibration. In particular, any subcategory of an $\infty$-category is itself a $\infty$-category and we will not mention this fact.
\begin{enumerate}
	\item[$\bullet$] $\Delta$ denotes the simplex category (see \cite[000A]{kerodon}), i.e. the category of finite non-empty linearly ordered sets, $\Delta_{+}$ the category of (possibly empty) finite linearly ordered sets. We denote with $\Delta_{s}$ those finite linearly ordered sets whose morphisms are strictly increasing functions and with $\Delta_{s,+}$ those (possibly empty) finite linearly ordered sets whose morphisms are strictly increasing functions.
	\item[$\bullet$] With an $\infty$-category, we always mean an $(\infty,1)$-category.
	\item[$\bullet$] $\SS$ denotes the $\infty$-category of small spaces (also called $\infty$-groupoids or anima).
	\item[$\bullet$] $\ICat$ denotes the $\infty$-category of small $\infty$-categories.
	\item[$\bullet$] $\Sp$ denotes the $\infty$-category of spectra.
	\item[$\bullet$] For an $E_{\infty}$-ring $A$, we denote the $\infty$-category of $A$-modules in spectra, i.e. $\MMod_{A}(\Sp)$ in the notation of \cite{HA}, with $\MMod_{A}$.
	\item[$\bullet$] For any ordered set $(S,\leq)$, we denote its corresponding $\infty$-category again with $S$, where the corresponding $\infty$-category of an ordered set is given by the nerve of the of $(S,\leq)$ seen as a $1$-category (the objects are given by the elements of $S$ and $\Hom_S(a,b)=\ast$ if and only if $a\leq b$ and otherwise empty). 
	\item[$\bullet$] For any set $S$ the $\infty$-category $S^{\disc}$ will denote the nerve of the set $S$ seen as a discrete $1$-category (the objects are given by the elements of $S$ and $\Hom_S(a,a)=\ast$ for any $a\in S$ and otherwise empty).
	\item[$\bullet$] For any morphism $f\colon X\rightarrow Y$  in an $\infty$-category $\Ccal$ with finite limits, we denote the functor from $\Delta_{+}$ to $\Ccal$ that is given by the \v{C}ech nerve of $f$ (see \cite[\S 6.1.2]{HTT}) if it exists by $\Cv(Y/X)_{\bullet}$.
	\item[$\bullet$] Let $\Ccal$ be an $\infty$-category with final object $\ast$. For morphisms $f\colon \ast\rightarrow X$ and $g\colon \ast\rightarrow X$, we denote the homotopy pullback $\ast\times_{f,X,g}\ast$ if it exists with $\Omega_{f,g}X$. If $\Ccal$ has an initial object $0$, then we denote the pullback $0\times_{X}0$ with $\Omega X$.
	\item[$\bullet$] Let $f\colon X\rightarrow Y$ be a morphism in $\SS$ and let $y\in Y$. We write $\fib_{y}(X\rightarrow Y)$ or $\fib_{y}(f)$ for the pullback $X\times_{Y}\ast$, where $\ast$ is the final object in $\SS$ (up to homotopy) and the morphism $\ast\rightarrow Y$ is induced by the element $y$, which by abuse of notation, we also denote with $y$.
	\item[$\bullet$] For a morphism $f\colon M\rightarrow N$ in $\MMod_{A}$, where $A$ is some $E_{\infty}$-ring, we set $\fib(f)=\fib(M\rightarrow N)$ (resp. $\cofib(f)=\cofib(M\rightarrow N)$) as the pullback (resp. pushout) of $f$ with the essentially unique zero morphism $0\rightarrow N$ (resp. $M\rightarrow 0$). 
	\item[$\bullet$]  When we say that a square diagram in an $\infty$-category $\Ccal$ of the form 
	$$
	\begin{tikzcd}
		W\arrow[r,""]\arrow[d,""]& X\arrow[d,""]\\
		Y\arrow[r,""] &Z
	\end{tikzcd}
	$$
	 is commutative, we always mean, that we can find a morphism $\Delta^{1}\times\Delta^{1}\rightarrow \Ccal$ of $\infty$-categories that extends the above diagram.

\end{enumerate}

\subsection*{Acknowledgement}
These note are part of my PhD thesis \cite{thesis} (another extract, dealing with derived $F$-zips, is \cite{yaylali}). I want to thank my advisor Torsten Wedhorn with whom I had many discussions concerning problems in derived algebraic geometry. He suggested that I should write these notes separately, so they could be beneficial for all people interested in this topic. I would also like to thank Benjamin Antieu, Adeel Khan and Jonathan Weinberger for their help concerning derived algebraic geometry respectively higher topos theory and Timo Richarz for helpful discussions. Lastly, I want to thank Simone Steilberg and Thibaud van den Hove.


\section{Overview: Higher algebra}
\label{HA}
In this section, we want to summarize some important aspects of spectra, $E_\infty$-rings and modules over those presented in \cite{HA}.\par 

Our main interests are animated rings and modules over those. Animated rings are presented in Section \ref{sec:derived commutative algebra}, so we want to focus on modules. These will be defined over a monoidal $\infty$-category, in this case the $\infty$-category of spectra $\Sp\coloneqq\Sp(\SS)$. So, one should think of $\Sp$ as an $\infty$-category equipped with a tensor product and spectral modules as modules for this tensor product over some commutative ring, called $E_\infty$-rings.\par
Let us start by recalling stable $\infty$-categories. An $\infty$-category is stable if it has a zero object, it admits fibers and cofibers and every cofiber sequence is a fiber sequence (see \cite[Def. 1.1.1.9]{HA}). This can be seen as an $\infty$-analogue of an abelian category. One very important feature of a stable $\infty$-category is that its homotopy category is automatically triangulated (see \cite[Thm. 1.1.2.14]{HA}). Stable $\infty$-categories have other nice stability properties, e.g. a square is a pullback if and only if it is a pushout and there exists finite limits and colimits (see \cite[Prop. 1.1.3.4]{HA}), but listing everything concerning stable $\infty$-categories would be to involved so we refer to \cite[\S1]{HA}.\par
Now let us come to the definition of the spectrum $\Sp(\Ccal)$ of a pointed $\infty$-category $\Ccal$ with finite limits. One definition is obtained by setting $\Sp(\Ccal)$ as the $\infty$-category of excisive, reduced functors from\footnote{Here $\SS^{\textup{fin}}_{\ast}$ denotes the smallest subcategory of $\SS$ that contains the final object, is stable under finite colimits and consist of pointed objects, where pointed means objects $x\in \SS$ with a morphism $\ast\rightarrow x$, (see \cite[Not. 1.4.2.5]{HA}).} $\SS^{\textup{fin}}_{\ast}$ to $\Ccal$ (see \cite[Def. 1.2.4.8]{HA}). Alternatively, one obtains $\Sp(\Ccal)$ as the homotopy limit of the tower $\dots\xrightarrow{\Omega}\Ccal\xrightarrow{\Omega}\Ccal$ (see \cite[Rem. 1.4.2.25]{HA}). Both viewpoints are useful. An important property of $\Sp(\Ccal)$ is that it is a stable $\infty$-category and if $\Ccal$ is presentable then so is $\Sp(\Ccal)$ (see \cite[Cor. 1.4.2.17, Prop. 1.4.4.4]{HA}). From now on, we assume $\Ccal$ to be presentable. The first definition as functors allows us to define a functor 
$$
\Omega^{\infty}\colon \Sp(\Ccal)\rightarrow \Ccal
$$
by evaluation on the zero sphere. In fact, $\Ccal$ is stable if and only if $\Omega^{\infty}$ is an equivalence. Another property of $\Omega^{\infty}$ is that it admits a left adjoint $\Sigma^{\infty}$ (see \cite[Prop 1.4.4.4]{HA}).\par 
Let us set $\Ccal=\SS_*$, the $\infty$-category of pointed spaces and let us set $\Sp\coloneqq\Sp(\SS_*)$. The second definition we gave allows us to identify the homotopy category $\hSp$ with the classical stable homotopy category\footnote{Symmetric spectra are certain sequences of Kan complexes $X_0,X_1,\dots$ with maps $\Sigma X_{n-1}\rightarrow X_{n}$. This category is equipped with a model structure (called the stable model structure) and is closed monoidal. One can also equip certain sequences of pointed topological spaces $X_0,X_1,\dots$ with maps $\Sigma X_{n-1}\rightarrow X_{n}$ with a model structure and endow it with a closed monoidal structure using the smash product. Both constructions are in fact Quillen equivalent (see \cite{SymSpec} for further information about symmetric spectra).} (see \cite[Rem. 1.4.3.2]{HA}). In fact, one can show that $\Sp$ is the $\infty$-category associated to the model category of symmetric spectra (see \cite[Ex. 4.1.8.6]{HA}). This allows us to define a monoidal structure on $\Sp$ using the monoidal structure on the underlying model category, given by the smash product. One important aspect of this monoidal structure is that its unit element is the sphere spectrum\footnote{We use the convention of \cite{HA}, where the sphere spectrum is the image of the final object $\ast\in\SS$ by $\Sigma^{\infty}$ (see \cite[\S 1.4.4]{HA} for more details).} and the tensor product preserves colimits in each variable (see \cite[Cor. 4.8.2.19]{HA}). This construction shows that for a spectrum object $X\in \Sp$ we have $$\Hom_{\Sp}(S,X)\simeq \Hom_{\SS}(\Delta^0,X)\simeq \Omega^{\infty}(X),$$ where $S$ denotes the sphere spectrum. This is nothing new if we think about abelian groups for example, since $\ZZ$-module homomorphisms from $\ZZ$ to any abelian group are uniquely characterized by the elements of the group and since here $\Hom_{\Sp}(S,X)$ is a Kan-complex, we see that it is equivalent to the underlying space of the spectrum $X$. An important side remark is that the heart of spectra is naturally identified with (the nerve of) the category of abelian groups.\par 
Using $\Omega^{\infty}$, stability of $\Sp$ and homotopy groups of Kan-complexes, we can define an accessible $t$-structure on $\Sp$ (see \cite[Prop 1.4.3.6]{HA}). In particular, we can define the homotopy groups of spectrum objects $X\in \Sp$, via $\pi_nX\simeq\pi_0\Omega^{\infty}(X[-n])$ (see proof of \cite[Prop 1.4.3.6]{HA}) and if $n\geq 2$, then these are given by $\pi_n\Omega^{\infty}(X)$ (see \cite[Rem. 1.4.3.8]{HA}).\\ \par 
Before we can define modules, we start with $E_\infty$-rings. We will not go into detail, since we will work with an analogue, namely animated rings. One should think of animated rings as $E_\infty$-rings with an extra structure. This extra structure vanishes if we are in characteristic zero but gives us no relation except a conservative functor from animated rings to $E_\infty$-rings in positive or mixed characteristic (see Proposition \ref{fun E to SCR}).\par 
As stated above, one should think of $\Sp$ as an $\infty$-categorical analogue of abelian groups. To define commutative rings in this $\infty$-category one could use the theory of $\infty$-operads and describe $E_\infty$-rings in terms of sections of the operad induced by the monoidal structure of $\Sp$ (see \cite[\S2]{HA} for more information about $\infty$-operads and \cite[\S3, 4]{HA} for the construction of rings using this approach). We will not describe how this is achieved but instead use a rectification argument, i.e. we set $E_\infty$-rings as the $\infty$-category associated to the commutative algebra objects in the underlying model category of $\Sp$. Both approaches are equivalent (see \cite[Thm. 4.5.4.7]{HA}) so we can think of $E_\infty$-rings as certain commutative rings in the model category associated to $\Sp$. Using the Eilenberg-Mac Lane spectrum one can see that for example ordinary commutative rings are discrete $E_\infty$-rings. This construction is also vague, since it requires the $\infty$-categorical localization of cofibrant commutative algebra objects in the underlying model category. But contrary to $\infty$-operads, we think it gives a more classical feeling of commutative rings.\\ \par 
Now let us conclude this section with modules over $E_{\infty}$-rings. Again \cite{HA} deals with modules using $\infty$-operads (see \cite[\S 3, 4]{HA}). Analogous to the $E_\infty$-ring case, we can apply a rectification statement to define modules using the monoidal structure on the underlying model category, again both constructions are equivalent (see \cite[Thm. 4.3.3.17]{HA}). For an $E_\infty$-ring $A$, we will denote the $\infty$-category of spectral $A$-modules with $\MMod_A$. By forgetting the module structure, we get a functor $\MMod_A\rightarrow \Sp$ (induced by the construction using $\infty$-operads, see \cite[Def. 3.3.3.8]{HA}). As in the classical case for abelian groups and modules, the forgetful functor is conservative (this follows from \cite[Cor. 4.2.3.2]{HA}). Further, if we restrict ourselves to connective\footnote{Here an object $c$ in an $\infty$-category $\Ccal$ with a conservative functor $f\colon\Ccal\rightarrow \Dcal$ into a stable $\infty$-category with $t$-structure is connective, if $f(c)$ is connective, i.e. $\pi_{i}f(c)\cong 0$ for $i<0$.} modules, then even the composition $\MMod_A^{\cn}\rightarrow\Sp\xrightarrow{\Omega^{\infty}}\SS$ is conservative (see \cite[Rem. 7.1.1.8]{HA}) This is an analogue to the fact that a morphism of classical modules is an isomorphism if and only if it is a bijection on the underlying sets. Another fact that is well known in the classical case is that, we have the following equivalence in $\SS$
$$
\Hom_{\MMod_A}(A,M)\simeq \Hom_{\Sp}(S,M)\simeq \Omega^\infty(M)
$$
(see \cite[Cor. 4.2.4.7]{HA}).\par 
As in the classical case, we can also define $E_\infty$-algebras and modules over algebras by setting $\Einfty_A\coloneqq \CAlg(\MMod_A)$, i.e. we endow $A$-modules with a monoidal structure and define $A$-algebras as commutative object in $\MMod_A$ (see \cite[7.1.3.8]{HA}). Alternatively, we could look at the over category $\Einfty_{A/}$ but both constructions are in fact equivalent (see \cite[Cor. 3.4.1.7]{HA}). We also have a forgetful functor from $\Einfty_A$ to $\Sp$, which under the identification $\Einfty_A\simeq \CAlg(\MMod_A)$ factors through the forget functor $\Einfty_A\rightarrow \MMod_A$ which is conservative (again follows from \cite[Cor. 4.2.3.2]{HA}). Further, the identification of $\pi_nR$ with $\pi_0\Hom_{\Sp}(S[n],R)$ for an $E_\infty$-ring $R$, where $\pi_nR$ is defined on the underlying spectrum of $R$, allows us to endow $\pi_*R\coloneqq \bigoplus_{n\in \ZZ}\pi_n R$ with a graded commutative ring structure (see \cite[7 7.1.1, Rem. 7.1.1.6]{HA}). We can also endow $\pi_*M\coloneqq\bigoplus_{n\in\ZZ}\pi_nM$ for any $R$-module $M$, with a graded $\pi_*R$-module structure (see \cite[\S 7.1.1]{HA}).\par 
The $\infty$-category of $A$-modules is also stable (see \cite[7.1.1.5]{HA}) and has an accessible $t$-structure induced by the accessible $t$-structure on $\Sp$. This $t$-structure allows us to identify the heart of $\MMod_A$ with the (nerve of the) ordinary category of $\pi_0A$-modules (see \cite[Prop. 7.1.1.13]{HA}) (note that by the above $\pi_0A$ is an ordinary commutative ring). This is analogous to $E_\infty$-algebra case, where for a connective $E_\infty$-ring $A$ the discrete $A$-algebras are precisely the ordinary commutative $\pi_0A$-algebras (see \cite[Prop. 7.1.3.18]{HA} (note that there is a typo in the statement)). \par 
A key difference to connective $E_\infty$-rings is that over ordinary (discrete) commutative rings $R$, the $R$-module spectra are not discrete $R$-modules but instead we have $\MMod_R\simeq \Dcal(R)$ as symmetric monoidal $\infty$-categories, where $\Dcal(R)$, denotes the derived $\infty$-category of $R$-modules\footnote{The derived $\infty$-category of a Grothendieck abelian category $\Acal$ is the $\infty$-category associated to the model category of chain complexes $\Ch(\Acal)$ (see \cite[Prop.  1.3.5.15.]{HA} and \cite[Prop. 1.3.5.3]{HA} for the model structure on chain complexes). The homotopy category $\textup{h}\Dcal(\Acal)$ is equivalent to the ordinary derived category $D(\Acal)$ of $\Acal$.} (see \cite[Thm. 7.1.2.13]{HA}). Let us also remark, that under this equivalence the homotopy groups of module spectra are isomorphic to the homology groups of the associated complex and since this is an equivalence of symmetric monoidal $\infty$-categories this isomorphism also respects the module structure (see Remark \ref{pi_n = H_n}).


\section{Derived commutative algebra}
\label{sec:derived commutative algebra}
In the following, $R$ will be a ring.\par 
In this section, we want to give a quick summary about animated rings, present the cotangent complex and analyse smooth morphisms between animated rings. Mainly, we show that these notions arise from our classical point of view and behave like one can expect.
\subsection{Animated rings}
\label{sec:simplicial commutative algebras}

In this section, we summarise important aspects of animated rings, for this we will follow \cite[\S25]{SAG}.
\par
By $\Poly_R$ we denote the category of polynomial $R$-algebras in finitely many variables. Then the category of $R$-algebras is naturally equivalent to the category of functors from $\Poly_R^{\op}$ to $\sets$ which preserve finite products\footnote{For a functor $F\colon \Poly_{R}^{\op}\rightarrow \Sets$ that preserves finite products we can set $F(R[X])$ as the underlying ring of $F$, where the multiplication is induced by $R[T]\rightarrow R[T_{1}]\otimes_{R}R[T_{2}],\ T\mapsto T_{1}T_{2}$ and the addition by $T\mapsto T_{1}+T_{2}$. Conversely, for any $R$-algebra $A$, we can construct a contraviarant functor from $\Poly_{R}$ to $\Sets$ via $A\mapsto\Hom_{\Alg{R}}(-,A)$. These constructions are inverse to each other.}. Applying this construction to the $\infty$-categorical case, we obtain $\AniAlg{R}$ the $\infty$-category of animated $R$-algebras, i.e.
$$
\AniAlg{R} \coloneqq \Fun_{\pi}(\Poly_R^{\op},\SS),
$$
where the subscript $\pi$ denotes the full subcategory of $\Fun(\Poly_R^{\op},\SS)$, that preserve finite products.
Alternatively, this $\infty$-category is obtained by freely adjoining sifted colimits to $\Poly_R$ (this is the meaning of \cite[Prop. 5.5.8.15]{HTT} using that any element in $\AniAlg{R}$ can be obtained by a sifted colimit in $\Poly_{R}$ by \cite[Lem. 5.5.8.14, Cor. 5.5.8.17]{HTT}).\par 
For a cocomplete category $C$ that is generated under colimits by its full subcategory of compact projective objects $C^{\textup{sfp}}$, Cesnavicius-Scholze define the $\infty$-category \textit{animation of $C$} in \cite[\S 5.1]{CS} denoted by $\Ani(C)$. The $\infty$-category $\Ani(C)$ is the $\infty$-category freely generated under sifted colimits by $C^{\textup{sfp}}$. In particular, with this definition, we see that $\Ani(\Alg{R})\simeq \AniAlg{R}$. This process can also be applied to $R$-modules, which we will look at later in Section \ref{sec:mod}, and to abelian groups, where $\Ani(\Ab)$ recovers the $\infty$-category of simplicial abelian groups (see \cite[\S 5.1]{CS} for more details). The animation of $\Sets$ recovers the $\infty$-category of $\infty$-groupoids, i.e. $\Ani(\Sets)\simeq \SS.$\par

We have another description for animated $R$-algebras. Let $\Abf$ be the category of product preserving functors from $\Poly_{R}^{\op}$ to simplicial sets\footnote{As for product preserving functors from $\Poly^{\op}_{R}$ to $\Sets$ it is not hard to see that a product preserving functor $F\colon\Poly^{\op}_{R}\rightarrow \SS$ defines a simplicial commutative ring, via $F\mapsto F(R[X])$ (the face and degeneracy maps have to respect the ring structure by functoriality). In particular, in this way we can identify $\Abf$ with the category of simplicial commutative $R$-algebras.}. We obtain a model structure on $\Abf$ by the Quillen model structure on simplicial sets (see \cite[5.5.9.1]{HTT}) - this is often called the model category of \textit{simplicial commutative $R$-algebras}. This model category is known to be a combinatorial, proper, simplicial model category (for more details on these properties, we refer to \cite[Ch. II \S4, \S6]{Quillen}). The $\infty$-category associated to this model category (i.e. $\N^{\textup{hc}}(\Abf^{\circ})$) is equivalent to $\AniAlg{R}$, where $\Abf^\circ$ denotes the full subcategory consisting of fibrant/cofibrant objects (see \cite[Cor. 5.5.9.3]{HTT}). 

\begin{defi}
	For a ring $R$, we define the $\infty$-category of \textit{animated $R$-algebras}, denoted by $\AniAlg{R}$ as the $\infty$-category $\Fun_{\pi}(\Poly_{R}^{\op},\SS)$. For an animated ring $A$, we define $\AniAlg{A}\coloneqq (\AniAlg{\ZZ})_{A/}$ as the $\infty$-category of \textit{animated $A$-algebras}.\par 
	Further, if $R=\ZZ$, we call an animated $R$-algebra an \textit{animated ring}.
\end{defi}

\begin{rem}
	 Note that for a ring $R$, we have $\AniAlg{R}\simeq(\AniAlg{\ZZ})_{R/}$ by \cite[Prop. 25.1.4.2]{SAG}.
\end{rem}

\begin{rem}
\label{Ani presentable}
	Since $\AniAlg{A}$ is the over category of $\AniAlg{\ZZ}$ which is the $\infty$-category of a combinatorial model category (which is explained in the beginning), we see with \cite[Prop. 5.5.3.11, Prop. A.3.7.6]{HTT} that $\AniAlg{A}$ is a presentable $\infty$-category.
\end{rem}

The following theorem allows us to connect $\AniAlg{R}$ with $\Einftycn_R$. The idea is simple,
as $\AniAlg{R}$ is generated by $\Poly_{R}$ under sifted colimits, any sifted colimit preserving functor is up to homotopy determined by its restriction to $\Poly_{R}$. Since $\Poly_{R}$ lies fully in $\Einftycn_{R}$ (which has sifted colimits), we therefore get a functor $\theta\colon\AniAlg{R}\rightarrow \Einftycn_{R}$ corresponding to the inclusion $\Poly_{R}\hookrightarrow \Einftycn_{R}$. We can also use this philosophy to analyze functors by restricting them to $\Poly_{R}$ if they preserve sifted colimits.

\begin{prop}
	Let $j\colon \Poly_R\rightarrow \AniAlg{R}$ denote the Yoneda-embedding. Then we have an equivalence of $\infty$-categories
	$$
	\Fun_{\textup{sift}}(\AniAlg{R},\Einftycn_R)\rightarrow \Fun(\Poly_R,\Einftycn_R),
	$$
	where the subscript \textup{sift} denotes the full subcategory of sifted colimit preserving functors.
\end{prop}
\begin{proof}
	This follows from \cite[Prop. 5.5.8.15, Cor. 5.5.8.17]{HTT} and note that $\AniAlg{R}$ has small colimits since it is presentable by Remark \ref{Ani presentable}.
\end{proof}

\begin{prop}
	\label{fun E to SCR}
	The functor $\theta\colon \AniAlg{R}\rightarrow \Einftycn_R$ described above is conservative, has a left adjoint $\theta^L$ and a right adjoint $\theta^R$.
\end{prop}
\begin{proof}
	This is \cite[25.1.2.2]{SAG}.
\end{proof}

\begin{rem}
	Let us take a closer look at the left adjoint of $\theta$. We know that per definition $\ZZ[X]$ is a compact and projective object of $\AniAlg{\ZZ}$, so in particular the functor $\Hom_{\AniAlg{\ZZ}}(\ZZ[X],-)$ commutes with sifted colimits (see \cite[Prop. 5.5.8.25]{HTT}) and since we can write any animated ring as a sifted colimit of polynomial rings, we see that for any animated ring $A$, we have $\Hom_{\AniAlg{\ZZ}}(\ZZ[X],A)\simeq \Omega^{\infty}\theta(A)$. We observe that for the free $E_{\infty}$-$\ZZ$-algebra in one variable $\ZZ\lbrace X\rbrace$ the same equivalence holds, i.e. $\Hom_{\Einftycn_{\ZZ}}(\ZZ\lbrace X\rbrace,\theta(A))\simeq \Omega^{\infty}\theta(A)$. Using the adjunction, we therefore see that $\theta^{L}(\ZZ\lbrace X\rbrace)\simeq \ZZ[X]$.
\end{rem}

The functor $\theta$ allows us to view any $A\in\AniAlg{R}$ as a ring object in $\Sp$. Thus we can associate fundamental groups to this object and also module objects in $\Sp$.
\begin{defi}
	Let $A\in \AniAlg{R}$. For any $i\in\ZZ$, we set $\pi_i(A)\coloneqq \pi_i(\theta(A))$ and we set $\MMod_A\coloneqq \MMod_{\theta(A)}$. We refer to elements of $\MMod_{A}$ as \textit{$A$-modules}.
\end{defi}

	Recall from Section \ref{HA}, that animated rings per definition have no negative homotopy groups and $\pi_{*}A$ is a graded ring.	

\begin{notation}
\label{notation truncation}
	We want to remark that we have the notion of truncation functors for animated rings (see \cite[\S 25.1.3]{SAG}), denoted by $\tau_{\leq n}$ for $n\in\NN_{0}$ and are induced by the truncations on the underlying $E_{\infty}$-rings. For an animated ring $A$ we denote $\tau_{\leq n}A$ with $A_{\leq n}$.\par 
	We denote with $(\AniAlg{R})_{\leq n}$ the full subcategory of $n$-truncated animated $R$-algebras. The elements of $(\AniAlg{R})_{\leq 0}\simeq \Alg{R}$ are called \textit{discrete}.
\end{notation}

\begin{rem}
	The inclusion of $n$-truncated animated $R$-algebras $(\AniAlg{R})_{\leq n}$ into $\AniAlg{R}$, for some $n\in\NN_{0}$, has a left adjoint denoted by $\tau_{\leq n}$ (see \cite[Rem. 25.1.3.4]{SAG}). Since per definition $\tau_{\leq 0} = \pi_{0}$, we see that passage to the underlying ring of an animated ring via $\pi_{0}$ preserves colimits.
\end{rem}

We can view any connective $E_\infty$-algebra over $R$ as a connective $R$-module (more precisely $\Einftycn_R\simeq\CAlg(\MMod^{\cn}_R)$). This induces a forgetful functor $\Einftycn_R\rightarrow \MMod^{\cn}_R$, which has a left adjoint (see \cite[Ex. 3.1.3.14]{HA}). Using the above left adjoint $\theta^L$, we can associate an animated ring to any connective $R$-module $M$.

\begin{defi}
	Let $A$ be an animated ring. Let $M\in\MMod_{A}^{\cn}$ be a connective $A$-module. We denote the image of $M$ under the left adjoint to the forgetful functor $\Einftycn_{\theta(A)}\rightarrow \MMod^{\cn}_A$ composed with $\theta^{L}$ by $\Sym_{A}(M)$ and call it the \textit{symmetric animated $A$-algebra of $M$}.\footnote{By \cite[Prop. 5.2.5.1]{HTT} the adjunction of $\theta$ and $\theta^{L}$ can be transfered to the adjunction of slice categories, i.e. $\theta$ induces an adjunction, which by abuse of notation we denote the same, between $\AniAlg{A}$ and $\Einftycn_{\theta(A)}$.}
\end{defi}
Further, in the following remark we want to explain that there are two possible ways to define homotopy groups on animated rings, rather naturally. But both notions are in fact equivalent (in the sense that the two notions produce isomorphic homotopy groups).

\begin{remark}
	The homotopy groups of an animated ring can be defined alternatively via the following. We have a natural functor from rings to abelian groups and then to sets by forgetting the ring structure. This induces a functor from $$F\colon \AniAlg{\ZZ}\rightarrow\Ani(\Ab)\rightarrow \Ani(\Sets)$$ (see \cite[\S 5.1.4]{CS}). The animation of abelian groups is the $\infty$-category of simplicial abelian groups and the animation of sets is $\SS$. Using this functor, we can also define the $n$-homotopy group of an animated ring $A$ via $\pi_{n}F(A)\in\SS$.	
	This construction of the homotopy groups agrees with the construction of the homotopy groups via passage to spectra.\par 
	The reason for this is the following commutative diagram
	$$
	\begin{tikzcd}
		\AniAlg{\ZZ}\arrow[dd,"F",bend right, shift right =0.9em, swap]\arrow[r,"\theta"]\arrow[d,""]& \Einfty^{\cn}\arrow[d,""]\\
		\Ani(\Ab)\arrow[r,""]\arrow[d,""]&\Dcal(\ZZ)^{\cn}\arrow[dl,"\Omega^{\infty}"]\\
		\SS
	\end{tikzcd}
	$$
	 (as all of these functors commute with sifted colimits\footnote{For $F$ and $\theta$ this follows from construction. That the forgetful functor from connective $E_{\infty}$-rings to modules commutes with sifted colimits follows from the fact, that the tensor product on spectra commutes with sifted colimits (see \cite[4.8.2.19, Cor. 3.2.3.2]{HA}), for $\Omega^{\infty}$ see \cite[Prop. 1.4.3.9]{HA}.}, we only need to check commutativity on polynomial $\ZZ$-algebras, which follows by construction).
\end{remark}

We want to conclude this section by explaining localizations of animated rings.

\begin{defrem}
	Let $\Ccal$ be a presentable $\infty$-category and let $S$ be a set of morphisms in $\Ccal$. Then, we say that an object $Z\in \Ccal$ is \textit{$S$-local} if for any morphism $f\colon X\rightarrow Y\in S$, we have that the morphism $\Hom_{\Ccal}(Y,Z)\rightarrow \Hom_{\Ccal}(X, Z)$ induced by $f$ is an equivalence. We say that a morphism $f\colon X\rightarrow Y$ in $\Ccal$ is an \textit{$S$-equivalence} if for any $S$-local object the morphism  $\Hom_{\Ccal}(Y,Z)\rightarrow \Hom_{\Ccal}(X, Z)$ induced by $f$ is an equivalence (see \cite[Def. 5.5.4.1]{HTT}, note that this definition does not need presentability of $\Ccal$ but as explained below presentability allows us to work with the full subcategory of $S$-local objects in a nice way).\par 
	The inclusion of the full subcategory $\Ccal[S^{-1}]$ of $S$-local objects in $\Ccal$ admits a left adjoint, which we call localization $\Ccal\rightarrow \Ccal[S^{-1}]$ (see \cite[Prop. 5.5.4.15]{HTT}). The idea is to "complete" $S$ by taking $\Sbar$ as the set of $S$-equivalences in $\Ccal$. Then $\Ccal[S^{-1}]$ is the localization of $\Ccal$ by $\Sbar$, which is strongly saturated (see \cite[\S 5.5.4]{HTT} for more details).\par
	 This is analogous to the classical localization, where even if we want to localize at one element of a ring, we have to automatically localize the multiplicative subset generated by the element.
\end{defrem}

\begin{propdef}[Localization]
	\label{localization}
	Let $A$ be an animated $R$-algebra and let $F\subseteq \pi_0(A)=\pi_0(\Hom_{\MMod_A}(A,A))$ be a subset. Then there exists $A[F^{-1}]\in \AniAlg{A}$, such that for all $B\in\AniAlg{A}$ the simplicial set $\Hom_{\AniAlg{A}}(A[F^{-1}],B)$ is nonempty if and only if the image of all $f\in F$ under $\pi_0(A)\rightarrow\pi_0(B)$ is invertible. Further if it is nonempty, then it is contractible.
\end{propdef}
\begin{proof}
	The proof is analogous to the proof of \cite[Prop. 1.2.9.1]{TV2}, which treats the special case where $F$ has only one element.\par 
	Let $\Sym\colon \MMod^{\cn}_A\rightarrow \AniAlg{A}$ be the left adjoint to the map $\AniAlg{A}\rightarrow\Einftycn_A\rightarrow\MMod^{\textup{cn}}_A$. Consider the set $S\coloneqq\lbrace \Sym(f)\colon \Sym(A)\rightarrow \Sym(A)\mid f\in F\rbrace$. We set $A[F^{-1}]$ as the image of $A$ under the localization map $\AniAlg{A}\rightarrow \AniAlg{A}[S^{-1}]$, where $\AniAlg{A}[S^{-1}]$ denotes the full subcategory of $\AniAlg{A}$ of $S$-local objects (note that $\AniAlg{A}$ is presentable by Remark \ref{Ani presentable}).\par 
	An object $B\in \AniAlg{A}$ is $S$-local if and only if the induced map
	$$
	f^*\colon \Hom_{\MMod_A}(A,B)\rightarrow \Hom_{\MMod_A}(A,B)
	$$
	is an equivalence for all $f\in F$. Equivalently, $f^{*}$ is an equivalence if and only if the multiplication by the image of $f$ on $\pi_iB$ is an equivalence for all $f\in F$. Therefore, any $A$-algebra $B$ is $S$-local if and only if the image of $f$ under $\pi_0A\rightarrow \pi_0B$ is invertible for all $f\in F$.\par
	Now assume that $\Hom_{\AniAlg{A}}(A[F^{-1}],B)$ is nonempty, then the morphism $\pi_{0}A\rightarrow \pi_{0}B$ factors through $\pi_{0}A[F^{-1}]$, so every $f\in F$ has invertible image in $\pi_{0}B$, since by definition $A[F^{-1}]$ is $S$-local.
	To see that if $\Hom_{\AniAlg{A}}(A[F^{-1}],B)$ is nonempty, then it is contractible, note that in this case $B$ is $S$-local and we have
	$$
		\Hom_{\AniAlg{A}}(A[F^{-1}],B)\simeq \Hom_{\AniAlg{A}}(A,B)\simeq \ast
	$$
	by adjunction.
\end{proof}

\begin{rem}
	Note that in the proof of Proposition \ref{localization} if we have a subset $F\subseteq \pi_{0}A$ and denote its generated multiplicative subset by $S$, then by \cite[Prop. 5.5.4.15]{HTT} an animated $A$-algebra $B$ is $\lbrace  \Sym(f)\colon \Sym(A)\rightarrow \Sym(A)\mid f\in F\rbrace$-local if and only if it is $\lbrace  \Sym(s)\colon \Sym(A)\rightarrow \Sym(A)\mid s\in S\rbrace$-local (note that $\AniAlg{A}$ is presentable by Remark \ref{Ani presentable}).
\end{rem}

\begin{notation}
	Let $A$ be an animated ring and $f\in \pi_{0}A$. Then we define the localization by an element as $A[f^{-1}]\coloneqq A[\lbrace f\rbrace^{-1}]$.
\end{notation}

\begin{rem}
Let $F$ be a subset of $\pi_{0}A$ and let $S$ be the multiplicative subset generated by $F$. By the universal property of the localization of rings, we know that $\pi_{0}A[F^{-1}]\cong S^{-1}\pi_{0}A$.\par 
Now assume that $F$ is given by a single element $f\in \pi_{0}A$. After the characterization of \'etale morphisms via the cotangent complex, we will see that the $\pi_{i}A[f^{-1}]\cong (\pi_{i}A)_{f}$ for all $i\geq 0$ (see Lemma \ref{homotopy of localization})	
\end{rem}


\begin{defi}
	Let $A\rightarrow B$ be a morphism of animated rings. Then $B$ is \textit{locally of finite presentation} over $A$ if it is compact as an animated $A$-algebra, i.e. the functor $\Hom_{\AniAlg{A}}(B,-)\colon \AniAlg{A}\rightarrow \SS$ commutes with filtered colimits.
\end{defi}

\begin{rem}
	Our notion of ''locally finite presentation'' is \textit{stronger} than the notion of ``finitely presented'' in the classical sense. What we mean is that if a morphism of animated rings $A\rightarrow B$ is locally of finite presentation, then the induced morphism of rings $\pi_{0}A\rightarrow \pi_{0}B$ is finitely presented. But the other way around is not true, as we will see that $A\rightarrow B$ is locally of finite presentation if and only if $\pi_{0}A\rightarrow \pi_{0}B$ is locally of finite presentation and its cotangent complex is perfect (see Proposition \ref{lfp perfect cotangent}). An example of a finitely presented morphism with non-perfect cotangent complex is the non-lci morphism $\FF_{p}\rightarrow \FF_{p}[X,Y]/(X^{2},XY,Y^{2})$ (the non-perfectness follows from \cite[(1.3)]{Avramov}).
\end{rem}

As open immersions are finitely presented in the classical world of algebraic geometry, we would expect a similar result in the derived world. This will get explicit later but first we would like to show that the fundamental example of an open immersion, the localization of a ring along an element, is locally of finite presentation. 

\begin{lem}
	\label{localization lfp}
	Let $A$ be an animated ring and $f\in\pi_0A$. Then $A[f^{-1}]$ is locally of finite presentation over $A$.
\end{lem}
\begin{proof}
	We have that $A$-algebra morphisms from the localization to any other $A$-algebra $B$ are either empty or contractible, depending whether $f$ is invertible in $\pi_0B$. For any filtered system $(B_i)_{i \in I}$ of $A$-algebras, we have $\pi_0\colim_{i \in I} B_i = \colim_{i \in I} \pi_0 B_i$. Therefore, we see that $\Hom_{\AniAlg{R}}(A[f^{-1}],\colim_{i \in I} B_{i})$ is empty if $f$ is not invertible in $\colim_{i \in I} \pi_{0}B_{i}$ and if $f$ is invertible in $\colim_{i \in I} \pi_{0}B_{i}$, then the space $\Hom_{\AniAlg{R}}(A[f^{-1}],\colim_{i \in I} B_{i})$ is contractible. Since $\pi_{0}A[f^{-1}]\cong \pi_{0}(A)_{f}$ is locally of finite presentation as a $\pi_{0}A$-algebra, we know that $f$ is invertible in $\colim \pi_{0}B_{i}$ if and only there is an $i'\in I$ such $f$ is invertible in $\pi_{0}B_{i'}$. So in particular, if there is such $i'$, we have $$\Hom_{\AniAlg{R}}(A[f^{-1}],\colim_{i \in I} B_{i})\simeq\ast\simeq \colim_{i\in I}\Hom_{\AniAlg{R}}(A[f^{-1}],B_{i})$$ and if there is no such $i'$, we have 
$$\Hom_{\AniAlg{R}}(A[f^{-1}],\colim_{i \in I} B_{i})=\emptyset= \colim_{i\in I}\Hom_{\AniAlg{R}}(A[f^{-1}],B_{i}).$$
\end{proof}

\subsection{Modules over animated rings}
\label{sec:mod}

Let us recall some useful notions about modules over animated ring. In the following $A$ will be an animated ring.

\begin{rem}
	\label{pi_n = H_n}
	Before we start, let us remark that under the symmetric monoidal equivalence of stable $\infty$-categories $\MMod_{R}\simeq \Dcal(R)$ explained in Section \ref{HA}, for a ring $R$, the homotopy groups are isomorphic to the corresponding homology groups, this isomorphism respects the module structure (see \cite[B.1]{SS}).\par 
\end{rem}

\begin{rem}
	We want to make clear that throughout, we will work in \textit{homological} notation. This is natural from the homotopy theory standpoint but differs from the algebraic geometry standpoint which uses \textit{cohomological} notation. In particular, we will define notions such as ``Tor-amplitude'' homologically.
\end{rem}

\begin{defrem}
	Let $P$ be a connective $A$-module, then $P$ is called \textit{projective} if for all connective $A$-modules $Q$, we have $\Ext^{1}(P,Q)\cong 0$, where $\Ext^{1}(P,Q)$ is defined as $$\pi_0\Hom_{\MMod_A}(P,Q[1])\cong \Hom_{h\MMod_{A}}(P,Q[1])\footnote{Recall that the homotopy category of a stable $\infty$-categories is an additive category, so the expression $\Ext^{1}(P,Q)\cong 0$ makes sense.}.$$ \par Equivalently, $P$ is projective if for all fiber sequences $$M'\rightarrow M\rightarrow M'',$$ where $M,M',M''$ are connective $A$-modules, the induced map $\Ext^0(P,M)\rightarrow \Ext^0(P,M'')$ is surjective (see \cite[Prop. 7.2.2.6]{HA}. This also shows equivalence with the definition given in \cite[Def. 7.2.2.4]{HA}).\par
	We denote by $\Proj(A)$ the full subcategory of projective $A$-modules in $\MMod_{A}$.
\end{defrem}

\begin{rem}
	From the definition of projective modules it follows that if an $A$-module $P$ is projective, then for any connective module $Q$ we have $\Ext^{i}(P,Q)\cong 0$ for all $i\geq 0$. In fact, this condition is equivalent to the same condition with $Q$ assumed to be discrete. This is also an equivalent definition of projective modules (see \cite[Prop. 7.2.2.6]{HA}).
\end{rem}

\begin{defi}
	A connective $A$-module $M$ is called \textit{flat}, if $\pi_{0}M$ is a flat $\pi_{0}A$-module and the natural morphism $\pi_{i}A\otimes_{\pi_{0}A}\pi_{0}M\rightarrow \pi_{i}M$ is an isomorphism.
\end{defi}

The compatibility with the higher homotopy groups is important if we want to define for example \textit{flat morphisms} of animated rings. The Tor-spectral sequence below then shows us that the homotopy groups are compatible with base change.\par
The following is a direct consequence of the definition of flatness.

\begin{lem}
\label{flat proj}
	Let $P$ be an $A$-module.
	\begin{enumerate}
		\item	If $P$ is projective it is flat.
		\item 	If $P$ is flat, then it is projective if and only if $\pi_0P$ is projective over $\pi_0A$.
	\end{enumerate}
\end{lem}
\begin{proof}
	See \cite[Lem. 7.2.2.14]{HA} and \cite[Prop. 7.2.2.18]{HA}.
\end{proof}

We also have a homotopy equivalence relating projective modules over $A$ and over $\pi_0A$.

\begin{prop}
	\label{projective lift}
	The base change with the natural map $A\rightarrow \pi_0A$ induces an equivalence between $h\Proj(A)$ and the $h\Proj(\pi_{0}A)$\footnote{Note, that since projective modules are flat, $h\Proj(\pi_{0}A)$ is just the usual category of (classical) projective $\pi_{0}A$-modules.}.
\end{prop}
\begin{proof}
	This follows from \cite[Cor. 7.2.2.19]{HA}.
\end{proof}

\begin{defi}
	We call an $A$-module $P$ \textit{finite projective}, if it is projective and $\pi_0P$ is finitely presented over $\pi_0A$.
\end{defi}

We can generalize this notion via the notion of perfectness.

\begin{defrem}
	An $A$-module $P$ is called \textit{perfect}, if it is a compact object of $\MMod_A$.
	Equivalently, $P$ is perfect if and only if there exists an $A$-module $P^{\vee}$ such that we have $\Hom_{\MMod_A}(P,-)\simeq \Omega^{\infty}(P^{\vee}\otimes_A -)$ (see \cite[Def. 7.2.4.1]{HA} and \cite[Prop. 7.2.4.2]{HA}).
\end{defrem}

\begin{remark}
	If $A$ is discrete, then we have $\MMod_A\simeq \Dcal(\pi_{0}A)$ as symmetric monoidal $\infty$-categories and a complex of $A$-modules is perfect in the our sense if and only if it is perfect in the classical sense (see \cite[07LT]{stacks-project}).
\end{remark}

\begin{rem}[\protect{\cite[Prop. 7.2.1.19]{HA}}]
\label{Tor-ss}
	Let us insert a quick remark about the Tor-spectral sequence associated to spectral modules. Let $A$ be an $E_{\infty}$-ring and $M,N\in\MMod_{A}$. Then $\pi_{\ast}M$ and $\pi_{\ast}N$ are graded $\pi_{\ast}A$-modules and we have a spectral sequence in graded abelian groups of the following form called the \textit{Tor-spectral sequence}
	$$
		E^{p,q}_{2}= \Tor_{p}^{\pi_{\ast}A}(\pi_{*}M,\pi_{*}N)_{q}\Rightarrow \pi_{p+q}(M\otimes_{A}N).
	$$
	Here convergence is in the sense that $\pi_{p+q}(M\otimes_{A}N)$ has a filtration $F^{\bullet}$ such that  
	$\gr^{p}_{F}\cong E_{\infty}^{p,q}$, there is a $k\leq 0$ such that $F^{n}\simeq 0$ for $n\leq k$ and $\colim_{n} F^{n}\simeq \pi_{p+q}(M\otimes_{A}N)$ (see \cite[Var. 7.2.1.15]{HA} for the construction of the Tor-group).
\end{rem}

We also have the notion of a Tor-amplitude for $A$-modules. We will use the homological notation, since it is in line with the definitions given in homotopy theory.

\begin{defi}[\protect{\cite[Def. 2.11]{AG}}]
	Let $M$ be an $A$-module. Then we say that $M$ has \textit{Tor-amplitude (concentrated) in $[a,b]$} for $a\leq b\in \ZZ$ if for all discrete $A$-modules $N$, we have $$\pi_i(M\otimes_A N)= 0$$ for all $i\not\in [a,b]$.\end{defi}

\begin{lem}
	Let $M$ be an $A$-module. Then $M$ has Tor-amplitude in $[a,b]$ if and only if the ordinary complex $M\otimes_A \pi_0A$ in $\Dcal(\pi_0A)$ has Tor-amplitude in $[a,b]$.
\end{lem}
\begin{proof}
	Let $F\colon \MMod_{\pi_{0}A}\rightarrow \MMod_{A}$ be the forgetful functor. This functor comes from the Cartesian fibration of \cite[Cor. 3.4.3.4]{HTT}. In particular, we see that for a discrete $\pi_{0}A$-module $N$ the underlying spectra of $N$ and $F(N)$ are equivalent. So, their homotopy groups are isomorphic, so $F(N)$ is discrete and up to equivalence determined by $\pi_{0}F(N)$ (see \cite[Prop. 7.1.1.13]{HA}). Since $\pi_{0}N$ determines $N$ up to equivalence and $\pi_{0}N\cong \pi_{0} F(N)$, we see that the diagram
	$$
	\begin{tikzcd}
		& N\Mod{\pi_{0}A}\arrow[ld,"i_{1}",hook, swap]\arrow[dr,"i_{2}",hook]&\\
		\MMod_{\pi_{0}A}\arrow[rr,"F"]&&\MMod_{A}
	\end{tikzcd}
	$$
commutes on the level of elements up to equivalence. Therefore, we see that for any $N\in N\Mod{\pi_{0}A}$, we have $$M\otimes_{A}i_{2}(N)\simeq M\otimes_{A}F(i_{1}(N))\simeq M\otimes_{A}\pi_{0}A\otimes_{\pi_{0}A}i_{1}(N)$$ concluding the proof.
\end{proof}

\begin{lem}
	\label{general props of Tor}
	Let $A$ be an animated $R$-algebra. Let $P$ and $Q$ be $A$-modules.
	\begin{enumerate}
		\item If $P$ is perfect, then $P$ has finite Tor-amplitude.
		\item If $B$ is an $A$-algebra and $P$ has Tor-amplitude in $[a,b]$, then the $B$-module $P\otimes_A B$ has Tor-amplitude in $[a,b]$.
		\item If $P$ has Tor-amplitude in $[a,b]$ and $Q$ has Tor-amplitude in $[c,d]$, then $P\otimes_A Q$ has Tor-amplitude in $[a+c,b+d]$.
		\item If $P,Q$ have Tor-amplitude in $[a,b]$, then for any morphism $f\colon P\rightarrow Q$ the fiber of $f$ has Tor-amplitude in $[a-1,b]$ and the cofiber of $f$ has Tor-amplitude in $[a,b+1]$.
		\item If $P$ is a perfect $A$-module with Tor-amplitude in $[0,b]$, with $0\leq b$, then $P$ is connective and $\pi_0P\simeq \pi_0(P\otimes_A \pi_0A)$.
		\item $P$ is perfect and has Tor-amplitude in $[a,a]$ if and only if $P$ is equivalent to $M[a]$ for some finite projective $A$-module.
		\item If $P$ is perfect and has Tor-amplitude in $[a,b]$, then there exists a morphisms 
		$$
		M[a]\rightarrow P
		$$
		such that $M$ is a finite projective $A$-module and the cofiber is perfect with Tor-amplitude in $[a+1,b]$.
	\end{enumerate}
\end{lem}
\begin{proof}
	Since modules over animated rings are defined as modules over their underlying $E_\infty$-ring spectrum, this is \cite[Prop. 2.13]{AG}.
\end{proof}

Let us conclude this section by specifically looking at connective modules over animated rings. These will be given by the animation of classical modules. This illustrates why we work with modules over spectra, as the animation of modules seems natural but produced only \textit{connective} objects.\par 
First let us consider the $\infty$-category $\MMod(\Sp)$ of tuples $(M,A)$, where $A$ is an $E_{\infty}$-R-algebra and $M$ is an $A$-module. This $\infty$-category comes naturally with a cartesian fibration $\MMod(\Sp)\rightarrow \Einfty_{R}$ (see \cite[\S 4.5]{HA} for more details).\par 
Now we can define the $\infty$-category $\SCRMod_R\coloneqq \MMod(\Sp)\times_{\CAlg_{R}(\Sp)}\AniAlg{R}$. Let us denote the full subcategory, consisting of objects $(M,A)\in\SCRMod_R$, where $M$ is connective by $\AniM_R$. The next proposition shows that $\AniM_R$ is the animation of the category of tuples $(A,M)$, where $A$ is an $R$-algebra and $M$ is an $A$-module.

\begin{prop}
	\label{Animod}
	Let $\Ccal\subseteq \AniM_R$ be the full subcategory consisting of objects $(M,A)$, where $A$ is a polynomial $R$-algebra and $M$ is a free $A$-module of finite rank. Let $\Ecal$ be an $\infty$-category, which admits sifted colimits. Let us denote by $\Fun_{\textup{sift}}(\AniM_R,\Ecal)$ the full subcategory of $\Fun(\AniM_R,\Ecal)$ spanned by those functors, which preserve sifted colimits. Then the restriction functor 
	$$
	\Fun_{\textup{sift}}(\AniM_R,\Ecal)\rightarrow \Fun(\Ccal,\Ecal)
	$$
	is an equivalence of $\infty$-categories.
\end{prop}
\begin{proof}
	This is \cite[Cor. 25.2.1.3]{SAG}, but since some references are broken, we recall the proof.\par 
	It suffices to show that $P_{\Sigma}(\Ccal)\simeq \AniM_R$, since then the proposition follows from \cite[Prop. 5.5.8.15]{HTT}, where $P_{\Sigma}(\Ccal)$ denotes those presheaves that preserve finite products.\par 
	The following is \cite[Prop. 25.2.1.2]{SAG} (here the references are broken). Note that $\Ccal$ consists of those objects in $\AniM_R$ which are coproducts of finitely many copies of $C\coloneqq (R[X],0)$ and $D=(R,R)$. We especially see that $C$ and $D$ corepresent the functors $\AniM_R\rightarrow \SS$ given by
	$$
	(A,M)\mapsto \Omega^{\infty}A^{\textup{sp}},\quad (A,M)\mapsto \Omega^{\infty}M
	$$
	respectively. Since both functors preserve sifted colimits, the objects $C$ and $D$ are compact, projective and $\Ccal$ consists of compact projective objects of $\AniM_R$. It follows with \cite[Prop. 5.5.8.22]{HTT}, that the inclusion $\Ccal\hookrightarrow \AniM_R$ extends (see \cite[Prop. 5.5.8.15]{HTT}) to a fully faithful functor $F\colon P_{\Sigma}(\Ccal)\rightarrow \AniM_R$, which commutes with sifted colimits. Since the inclusion preserves finite coproducts, we see that $F$ preserves small colimits (see \cite[Prop. 5.5.8.15]{HTT}) and therefore admits a right adjoint $G$, by the adjoint functor theorem (see \cite[Cor. 5.5.2.9]{HTT}). To prove that $F$ is an equivalence, it suffices to show that $G$ is conservative.\par 
	To see this, note that since $F$ is left adjoint and fully faithful, the unit map $\id\rightarrow GF$ is an equivalence.\par
	That $G$ is conservative is clear, since the conservative functor 
	$$
	\AniM_R\rightarrow \SS\times\SS,\quad (A,M)\mapsto (\Omega^{\infty}A^{\textup{sp}},\Omega^{\infty}M)
	$$
	factors through $G$.
\end{proof}

\begin{notation}
	Again, for an animated ring $A$, we denote $\SCRMod_A\coloneqq \SCRMod\times_{\Ani} \AniAlg{A}$ and $\AniM_A\coloneqq \AniM\times_{\Ani}\AniAlg{A}$.
\end{notation}

\begin{remark}
	The above proposition also shows that the $\infty$-category of simplicial commutative $R$-modules, which is equivalent to $\AniMod{R}$ (the animation of $R$-modules), is equivalent to the connective $R$-modules $\MMod_R^{\cn}$.
\end{remark}

Lastly, we will insert a lemma which will show the uniqueness of the cotangent complex for derived stacks (see Definition \ref{defi.cotangent.global}).

\begin{lem}
	\label{connective yoneda ff}
	Let $j\colon \MMod_A\rightarrow \Pcal(\MMod_A^{\cn,\op})$ be the Yoneda embedding followed by restriction. Then for all $n\geq 0$ the restriction of $j$ to the $(-n)$-connective objects is fully faithful.
\end{lem}
\begin{proof}
	We will not prove this here and refer to \cite[Prop. 1.2.11.3]{TV2}.
\end{proof}

\subsection{The cotangent complex}
\label{sec:cotangent affine}
We will define square zero extensions and the cotangent complex following \cite[\S25]{SAG}.\par 
Proposition \ref{Animod} allows us to define square zero extensions. Namely, if we look at the functor $\Ccal\rightarrow \Ring\simeq (\AniAlg{R})_{\leq 0}\hookrightarrow \AniAlg{R}$, where $\Ccal$ is as in Proposition \ref{Animod}, given by $(M,A)\mapsto A\oplus M$, we see that it induces a functor $\AniM_R\rightarrow \AniAlg{R}$ commuting with sifted colimits. 

\begin{defi}
	For an $A\in\AniAlg{R}$ and a connective $A$-module $M$, we define the \textit{square zero extension of $A$ by $M$} as the image of $(M,A)$ under the functor $\AniM_R\rightarrow \AniAlg{R}$ described above and denote the resulting animated $R$-algebra by $A\oplus M$.
\end{defi}

\begin{rem}
	Since the forgetful functor from $\AniAlg{R}\rightarrow \MMod_{R}$ preserves colimits, we see that the underlying module of $A\oplus M$, for some animated $R$-algebra $A$ and a connective $A$-module $M$, is equivalent to the direct sum in $\MMod_{R}$ of $A$ and $M$.
\end{rem}

\begin{remark}
	Let $A\in \AniAlg{R}$ and $M$ be a connective $A$-module. In $\MMod_A$ we have (up to homotopy) unique maps $0\rightarrow M\rightarrow 0$, these determine maps between animated rings $A\rightarrow A\oplus M\rightarrow A$. Thus, we can view $A\oplus M$ as an element of $(\AniAlg{A})_{/A}$.
\end{remark}

Since we have defined square zero extensions of an animated algebra by a connective module, we can now define the notion of a derivation.

\begin{defi}
	Let $A\in\AniAlg{R}$ and $M\in \MMod^\cn_A$. The \textit{space of $R$-linear derivations $\Der_R(A,M)$ of $A$ into $M$}  is defined as the mapping space $\Hom_{(\AniAlg{R})_{/A}}(A,A\oplus M)$.
\end{defi}

\begin{defi}
\label{defi triv square zero ext}
	Let $A\in\AniAlg{R}$, $M\in \MMod^\cn_A$ and $d\in \Der_R(A,M)$. Then we define $A\oplus_{d} M$ as the pullback of $d\colon A\rightarrow A\oplus M$ and the trivial derivation $s$, i.e. we have a pullback diagram of the form 
	$$
	\begin{tikzcd}
		A\oplus_{d}M\arrow[r,""]\arrow[d,""]& A\arrow[d,"d"]\\
		A\arrow[r,"s"]&A\oplus M.
	\end{tikzcd}
	$$
\end{defi}

Next, we want to define the absolute cotangent complex associated to an aniamted ring $A$. This should be thought of as an $\infty$-analogue of the module of differentials. So, we will characterize it by a universal property.
\begin{propdef}
	Let $A\in\AniAlg{R}$. There is a connective $A$-module $L_A$ and a derivation $\eta\in \Der_R(A,L_A)$ uniquely (up to equivalence) characterized by the property, that for every connective $A$-module $M$ the map
	$$
	\Hom_{\MMod_A}(L_A,M)\rightarrow \Der_R(A,M)
	$$ 
	induced by $\eta$ is an equivalence.\par 
	We call the $A$-module $L_{A/R}$ the \textit{cotangent complex of $A$ over $R$}.\par 
	If $R=\ZZ$, then we write $L_{A}$ and call it the \textit{absolute cotangent complex of $A$}.
\end{propdef}
\begin{proof}
	This is \cite[Prop. 25.3.1.5]{SAG}.
\end{proof}

\begin{rem}
\label{morph on derivations}
	Let $A\rightarrow B$ be a morphism in $\AniAlg{R}$. 
	Then for any connective B-module $M$, we will see that we have a map $\Der_{R}(B,M)\rightarrow \Der_{R}(A,M)$, where we see $M$ as an $A$-module via the forgetful functor. This follows from the following.\par 
	Note that by functoriality, we have a commutative diagram of the form 
	$$
	\begin{tikzcd}
		A\oplus M\arrow[r,""]\arrow[d,""]& A\arrow[d,""]\\
		B\oplus M\arrow[r,""]&B.
	\end{tikzcd}
	$$
	This induces a map $A\oplus M\rightarrow A\times_{B} (B\oplus M)$, which is an equivalence when passing to the underlying $R$-modules, since the underlying $R$-module of $B\oplus M$ is the direct sum of $B$ and $M$. Therefore for an $R$-derivation $d\colon B\rightarrow B\oplus M$, we get the following diagram in $\AniAlg{R}$ with pullback squares
	$$
	\begin{tikzcd}
		A\arrow[r,""]\arrow[d,""]& A\oplus M\arrow[d,""]\arrow[r,""]&A\arrow[d,""]\\
		B\arrow[r,""]&B\oplus M\arrow[r,""]&B,
	\end{tikzcd}
	$$
	where the composition of the horizontal arrows is the identity on $A$, respectively $B$. 
	Therefore, an $R$-derivation of $B$ induces an $R$-derivation of $A$.\par 
 Thus, per definition, we get a map $$\Hom_{B}(L_{B/R},M)\rightarrow\Hom_{A}(L_{A/R},M)\simeq \Hom_{B}(L_{A/R}\otimes_{A}B,M).$$ Where the second map is induced by the adjunction of the forgetful functor and the tensor product. Now Lemma \ref{connective yoneda ff} induces a map $L_{A/R}\otimes_{A}B\rightarrow L_{B/A}$. 
\end{rem}

\begin{defi}
	Let $A\rightarrow B$ be a morphism in $\AniAlg{R}$.	
	Then we define \textit{the (relative) cotangent complex of $B$ over $A$}, denoted by $L_{B/A}$, as the cofiber of the induced map $B\otimes_A L_{A/R}\rightarrow L_{B/R}$.
\end{defi}

\begin{remark}	
\label{rel cotangent rep}
Let $A\rightarrow B$ be a morphism in $\AniAlg{R}$.
	For every connective $B$-module $M$ the definition of $L_{B/A}$ as the cofiber of $L_{A/R}\otimes_{A} B\rightarrow L_{B/R}$ induces an equivalence $$\Hom_{\MMod_B}(L_{B/A},M)\simeq\fib_{d_{0}}(\Der_{R}(B,M)\rightarrow \Der_{R}(A,M)),$$ where $d_{0}$ is the trivial $R$-derivation of $A$ into $M$. Therefore, seeing $B\oplus M$ as an animated $A$-algebra, via the trivial derivation and the natural morphism $A\oplus M\rightarrow B\oplus M$, we have an equivalence 
	$$
	\Hom_{\MMod_B}(L_{B/A},M)\xrightarrow{\sim} \fib_{\id_B}(\Hom_{\AniAlg{A}}(B,B\oplus M)\rightarrow \Hom_{\AniAlg{A}}(B,B)).
	$$\par
	So the relative cotangent complex represents morphisms $B\rightarrow B\oplus M$, with an augmentation $B\oplus M\rightarrow B$, which are not only $R$-linear but in fact also $A$-linear, i.e. $\Hom_{\MMod_{B}}(L_{B/A},M)\simeq \Hom_{(\AniAlg{A})_{/B}}(B,B\oplus M)$.
\end{remark}

\begin{rem}
\label{cotangent bc}
	We claim that the definition of the cotangent complex as the module representing derivations shows that for any pushout diagram of the form
	$$
	\begin{tikzcd}
		A'\arrow[r,""]\arrow[d,""]& B'\arrow[d,""]\\
		A\arrow[r,""]&B
	\end{tikzcd}
	$$
	in $\AniAlg{R}$, we have $L_{B/A}\simeq L_{B'/A'}\otimes_{A'} B$.\par
	Indeed, let $M$ be a $B$-module. Using the arguments of Remark \ref{morph on derivations}, we get a morphism $$\Hom_{(\AniAlg{A})_{/B}}(B,B\oplus M)\rightarrow \Hom_{(\AniAlg{A'})_{/B'}}(B',B'\oplus M).$$  Let $d\in \Hom_{(\AniAlg{A'})_{/B'}}(B',B'\oplus M)$, which is given by a diagram 
	$$
	B'\rightarrow B'\oplus M\rightarrow B'
	$$
	such that the composition is homotopic to the identity on $B'$. By the universal property of the pushout, this induces an $A$-derivation of $B$ into $B\oplus M$. Both constructions are inverse to each other using universal properties, so the morphism $$\Hom_{(\AniAlg{A})_{/B}}(B,B\oplus M)\rightarrow \Hom_{(\AniAlg{A'})_{/B'}}(B',B'\oplus M)$$ is an equivalence and using Remark \ref{rel cotangent rep}, we are done.
\end{rem}

\begin{remark}
	We defined the cotangent complex of an animated $R$-algebra as the $\infty$-analogue of the K\"ahler-differentials. The same construction can be done for $E_\infty$-algebras (see \cite[7.3]{HA}). Thus, we could ask the question whether for a map of animated $R$-algebras $A\rightarrow B$ the relative cotangent complex $L_{B/A}$ is equivalent to the relative cotangent complex of the underlying map of $E_\infty$-algebras $L^\infty_{B/A}\coloneqq L_{\theta(B)/\theta(A)}$. In general, the answer is no, take for example $\ZZ\rightarrow \ZZ[X]$ (since this morphism in not formally smooth if we allow non-connective $E_{\infty}$-rings, this follows from \cite[Prop. 2.4.1.5]{TV2}). \par 
	 But we have an induced morphism $L^\infty_{A/B}\rightarrow L_{B/A}$ and passing to the homotopy groups, we see that the map $\pi_i L^\infty_{A/B}\rightarrow \pi_i L_{B/A}$ is an isomorphism if $i\leq 1$ and surjective for $i=2$ (see \cite[25.3.5.1]{SAG}, note that Lurie calls our cotangent complex \textit{the algebraic cotangent complex}).\par 
	In characteristic $0$ however we have that the map on homotopy groups is an isomorphism for all $i\in\ZZ$ (see \cite[25.3.5.3]{SAG}). This is not surprising, since in characteristic $0$ we have an equivalence of animated rings and connective $E_\infty$-rings. 
\end{remark}

\begin{remark}
\label{rem.sym.cotangent}
	Let $R$ be a ring and $A$ be an animated $R$-algebra. Let $B=\Sym_A(M)$ for some connective $A$-module $M$. We claim that $L_{B/A}$ is given by $M\otimes_A B$.\par
	Indeed, for any connective $B$-module $M'$, we have 
	\begin{align*}
		\Hom_{(\AniAlg{A})_{/B}}(\Sym_{A}(M),A\oplus M') \\ \simeq \fib_{\id_{B}}(\Hom_{\AniAlg{A}}&(\Sym_{A}(M),B\oplus M')\rightarrow \Hom_{\AniAlg{A}}(\Sym_{A}(M),B)).
	\end{align*}
	 Using the adjunction of the forget functor and $\Sym_{A}$, we get a map $\iota\colon M\rightarrow B$ corresponding to the identity on $B$. Further, we have
	\begin{align*}
		\fib_{\id_{B}}(\Hom_{\AniAlg{A}}(\Sym_{A}(M),B\oplus M')\rightarrow \Hom_{\AniAlg{A}}(\Sym_{A}(M),B))\\ \simeq \fib_{\iota}(\Hom_{\MMod_{A}}(M,B\oplus M')\rightarrow \Hom_{\MMod_{A}}(M,B)).
	\end{align*}
	Now the underlying module of $B\oplus M$ is the direct sum of $B$ with $M$ in $\MMod_{A}$. Therefore, this fiber is equivalent to $\Hom_{\MMod_{A}}(M,M')\simeq \Hom_{\MMod_{B}}(M\otimes_{A} B,M')$. Hence, by the universal property the cotangent complex $L_{B/A}$ is given by $M\otimes_A B$.
\end{remark}

\begin{prop}
	\label{affine cotangent complex}
	Let $A$ be an animated $R$-algebra and write $A$ as the sifted colimit of polynomial $R$-algebras $P^\bullet$. Then we have  $L_{A/R}\simeq \colim L_{P^\bullet/R}\otimes_{P^\bullet} A$.
\end{prop}
\begin{proof}
	This is \cite[Lec. 5 Thm. 2.3]{Khan}. For the convenience of the reader, we recall the proof.\par 
	Note that each $P^i$ is equivalent to $\Sym(R^{n_{i}})$ for some $n_{i}\in \NN$.
	Also by Remark \ref{rem.sym.cotangent}, we have $\colim L_{P^\bullet/R}\otimes_{P^\bullet} A \simeq \colim R^\bullet\otimes_{R} P^\bullet \otimes_{P^\bullet} A \simeq \colim A^\bullet$, where $A^\bullet\coloneqq R^\bullet\otimes_R A.$ Thus, we have to show that 
	$$
	\lim\Hom_{\MMod_A}(A^\bullet,M)\simeq \lim\fib_{\iota_\bullet}(\Hom_{\AniAlg{R}}(P^\bullet,A\oplus M)\rightarrow \Hom_{\AniAlg{R}}(P^\bullet,A))
	$$
	for all connective $A$-modules $M$, where $\iota_\bullet\colon P^\bullet\rightarrow A$ are the natural maps. But now we will see that both sides can be identified with $\lim \Omega^{\infty}(M)^\bullet$, so we conclude the equivalence.
	\par 
	Indeed, each $A^{n_{i}}$ is free this is clear for the left hand side and for the right hand side note that $P^\bullet$ is given by symmetric algebras of free algebras and conclude using the above.
\end{proof}

\begin{remark}
	Note that the proof of Proposition \ref{affine cotangent complex} shows that for a discrete ring $R$, the cotangent complex $L_{B/R}$ for some $B\in\AniAlg{R}$ agrees with the left Kan extension of $\Omega^{1}_{-/R}\colon\Poly_R\rightarrow \Dcal(R)$ along the inclusion $\Poly_{R}\hookrightarrow\AniAlg{R}$. So in particular, our definition agrees with the classical definition in the discrete case\footnote{Let  $L\Omega^{1}_{-/R}$ denote the left Kan extension of $\Omega^{1}_{-/R}\colon\Poly_R\rightarrow \Dcal(R)$ along the inclusion $\Poly_{R}\hookrightarrow\AniAlg{R}$. Then $L\Omega^{1}_{-/R}$ factors through $\AniM_R$, since $\Omega^{1}_{A/R}$ has an $A$-module structure for some polynomial $R$-algebra $A$. Thus, we see that $L\Omega^{1}_{B/R}$ has a natural $B$-module structure and on polynomial algebras agrees with $L_{B/R}$ (note that this a priori proves the equivalence in $\Dcal(R)$ and thus in $\Sp$ but the forget functor $\AniM_R\rightarrow \Sp$ is conservative by \cite[Cor. 4.2.3.2]{HA}). }.\par 
Setting $R=\ZZ$, we get an analogous description of $L_B.$
\end{remark}

We add a little lemma showing that the steps in the Postnikov towers are square-zero extensions. For $E_\infty$-algebras this can be found in \cite[Cor. 7.4.1.28]{HA}. We will not prove this lemma, since it would be to involved. But a detailed (model categorical) proof can be found in the notes of Porta and Vezzosi \cite{PV}.

\begin{lem}
	\label{postnikov square zero}
	Let $A$ be an animated $R$-algebra. There exists a unique derivation $$d\in\pi_0\Der(A_{\leq n-1},\pi_n(A)[n+1])$$ such that the projection
	$A_{\leq n-1}\oplus_{d}\pi_n(A)[n+1]\rightarrow A_{\leq n-1}$ is equivalent to the natural morphism $A_{\leq n}\rightarrow A_{\leq n-1}$ (recall the notation from Definition \ref{defi triv square zero ext}). 
\end{lem}
\begin{proof}
	\cite[Lem. 2.2.1.1]{TV2}.
\end{proof}

\subsection{Smooth and \'etale morphisms}
\label{sec:smooth and et}
In this section, we will follow \cite{TV2}. In the reference To\"en and Vezzosi deal with animated rings and derived algebraic geometry (in our sense) in the model categorical setting. Most definitions however are made in such a way, such that we can easily translate them to the $\infty$-categorical setting (for more explanation on how to go from model categories to $\infty$-categories, we recommend \cite{HTT} and \cite{HA}).

\begin{defi}
	A morphism $f\colon A\rightarrow B$ of animated $R$-algebras is called \textit{flat (resp. faithfully flat, smooth, \'etale)} if the following two conditions are satisfied
	\begin{enumerate}
		\item[(i)] the induced ring homomorphism $\pi_0f\colon\pi_0A\rightarrow\pi_0B$ is flat (resp. faithfully flat, smooth, \'etale), and 
		\item[(ii)] we have an isomorphism $\pi_*A\otimes_{\pi_0A}\pi_0B\rightarrow \pi_* B$ of graded rings.
	\end{enumerate}\par 
\end{defi}

Note that in the above definition $\pi_{0}f$ is a morphism of commutative rings, so asking whether they are flat, \'etale or smooth is natural. The condition (ii) on the homotopy groups is a natural compatibility condition that assures that homotopy groups and base change commute in the sense that $\pi_{n}M\otimes_{\pi_{0}A}\pi_{0}B\cong \pi_{n}(M\otimes_{A}B)$ for an $A$-module $M$ and a flat animated $A$-algebra $B$. (see \cite[Prop. 7.2.2.13]{HA}).
\newline\par

Let $f\colon A\rightarrow B$ be a homomorphism of rings. If $f$ is smooth, we know that the module of differentials $\Omega_{B/A}$ is finite projective. The other direction is in general not correct, i.e. there are ring homomorphisms locally of finite presentation with finite projective module of differentials that are not smooth. One example is the projection $k[X]\rightarrow k[X]/(X)\cong k$ for a field $k$. Its module of differentials $\Omega_{k/k[X]}$ vanishes, so is in particular a finite dimensional $k$-vector space. But the projection is not smooth. One way to see this, is that the cotangent complex $L_{k/k[X]}$ is quasi-isomorphic to $(X)/(X^{2})[-1]$\footnote{As the projection $k[X]\rightarrow k$ is a regular immersion, we can use \cite[08SJ]{stacks-project}.} (see \cite[07BU]{stacks-project}). This condition on the cotangent complex comes rather naturally in derived algebraic geometry, as we can see the cotangent complex as the derived version of the module of differentials. In fact, we will make this explicit in Proposition \ref{smooth cotangent} to see that a morphism of animated rings is smooth if and only if it is locally of finite presentation and its cotangent complex is finite projective.\par 
But for technical reasons, we need to understand the cotangent complex of the natural maps $A_{\leq k}\rightarrow A_{\leq k-1}$. As these are given by square zero extensions by $\pi_{k}A$ (see Lemma \ref{postnikov square zero}) and isomorphisms on $\pi_{i}$ for $i\leq k-1$, we will see that the cotangent complex is easier to understand.

\begin{lem}
	\label{cotangent of truncation}
	Let $A$ be an animated $R$-algebra and $k\geq 1$. Then there exists natural isomorphisms 
	$$
	\pi_{k+1}L_{A_{\leq k-1}/A_{\leq k}} \cong \pi_k A,
	$$
	and $\pi_i L_{A_{\leq k-1}/A_{\leq k}} \cong 0$ for $i\leq k$ (recall Notation \ref{notation truncation} for the $A_{\leq *}$).
\end{lem}
\begin{proof}
	This is \cite[Lem. 2.2.2.8]{TV2} translated to $\infty$-categories. \par
	We can also deduce this lemma using \cite{HA} and \cite{SAG}. Note that the fiber of $A_{\leq k}\rightarrow A_{\leq k-1}$ as an $A$-module is given by $\pi_kA[k]$ and thus the cofiber is given by $\pi_kA[k+1]$. Now the first assertion follows from \cite[Rem. 25.3.6.5]{SAG}. For the vanishing of the lower homotopy groups, note that for $i\leq k+2$, we have $\pi_i L^{\infty}_{A_{\leq k-1}/A_{\leq k}}\cong \pi_i L_{A_{\leq k-1}/A_{\leq k}}$ by \cite[Prop. 25.3.5.1]{SAG}, where $L^{\infty}_{A_{\leq k-1}/A_{\leq k}}$ is the cotangent complex associated to the underlying $E_{\infty}$-algebra of $A_{\leq k-1}\rightarrow A_{\leq k}$ (see \cite[\S 7.3]{HA} for more details). Thus, the vanishing follows with \cite[Lem. 7.4.3.17]{HA}. 
\end{proof}

In the following proof, we want to compute $$\Tor_*^{\pi_\ast A}(\pi_0A,\pi_\ast M)$$ for some $A\in\AniAlg{R}$ and an $A$-module $M$. The Tor-spectral sequence will be an important tool when calculating tensor products. But first lets talk about graded free resolutions, which compute the Tor-groups.

\begin{remark}[Graded free resolutions]
	\label{graded free res}
	Let $A\in\AniAlg{R}$ and fix an $A$-module $M$ such that there is an $n\geq 0$ such that $\pi_kM=0$ for all $k< n$. To compute the graded Tor group, we need a graded free resolution of $\pi_*M$, where a graded free resolution is an exact sequence $\dots\rightarrow P_1\rightarrow P_0\rightarrow \pi_*M\rightarrow 0$ of graded $\pi_*A$-modules with $P_i$ equal to the direct sum of shifts of $\pi_*A$ (see \cite[09KK]{stacks-project} for details on graded free resolutions). We can follow the existence of such a sequence inductively with \cite[09KN]{stacks-project}. Namely, set $M_0\coloneqq \pi_*M$, then for any $i$, we can find a short exact sequence of the form $0\rightarrow M_{i+1}\rightarrow P_i\rightarrow M_i\rightarrow 0$, such that $P_i$ is the direct sum of shifts of $\pi_*A$ and is concentrated in degrees $\geq n$.
	\par 
	The $P_i$ are constructed as follows, for any $m\in (M_i)_k$, we can find a $\pi_*A$-linear map $ \pi_*(A)[k]\rightarrow M_i$ sending $1\in  \pi_*(A)[k]_k$ to $m$ and $d(1)$ to $d(m)$. The module $P_i$ is now defined as the direct sum over all non-zero degrees and corresponding elements. Therefore by construction $P_i$ is a direct sum of shifts of $\pi_*A$ and further by induction, we see that if $M_0$ is concentrated in degrees $\geq n$, then for all $i\geq 0$ the module $P_i$ is concentrated in degrees $\geq n$.
\end{remark}

\begin{prop}
	\label{smooth cotangent}
	Let $f\colon A\rightarrow B$ be a morphism in $\AniAlg{R}$. 
	\begin{enumerate}
		\item The morphism $f$ is smooth if and only if the $B$-module $L_{B/A}$ is finite projective and $\pi_0B$ is of finite presentation over $\pi_0A$.
		\item The morphism $f$ is \'etale if and only if $L_{B/A}\simeq 0$ and $\pi_0B$ is of finite presentation over $\pi_0A$.
	\end{enumerate}	
\end{prop}
\begin{proof}
	This is \cite[Thm. 2.2.2.6]{TV2} in the $\infty$-categorical setting. For the convenience of the reader, we recall the proof of the first assertion. The proof of the second assertion is left out and can be reconstructed following the proof of \cite[Thm. 2.2.2.6]{TV2}.\par 
	Let $f$ be a smooth morphism, then it is flat by assumption, thus $\pi_0B= \pi_0A \otimes_A B$ (by \cite[Prop. 7.2.2.15]{HA} $B\otimes_A\pi_0A$ has to be discrete and since on $\pi_0$ it is an equivalence, we have the equivalence on the level of animated rings). In particular, we have  $$\pi_0 L_{B/A} = \pi_0(L_{B/A} \otimes_B \pi_0 B) = \pi_0 L_{\pi_0B/\pi_0A} = \Omega_{\pi_0B/\pi_0A}[0],$$ by compatibility of the cotangent complex with base change (see Remark \ref{cotangent bc}). Since $f$ is smooth, we see that $\pi_0(L_{B/A})=\Omega_{\pi_0B/\pi_0A}[0]$ is finite projective over $\pi_0B$. By Proposition \ref{projective lift}, there is a projective $B$-module $P$ with $P\otimes_B\pi_0B\simeq \pi_0L_{B/A}$ (thus $P$ is in fact finite projective). Using the projectivity of $P$, we lift the natural projection $P\rightarrow \pi_0L_{B/A}$ to a morphism $\phi\colon P\rightarrow L_{B/A}$ (see \cite[Prop. 7.2.2.6]{HA} and note that surjectivity of $\phi$ on $\pi_0$ implies that the fiber is connective). We want to show that $\phi$ is in fact an equivalence. For this it is enough to show $\cofib(\phi)\simeq 0$. By construction, it is clear that $\pi_0 \cofib(\phi)= 0$ and we will show by induction on $n$ that $\pi_n \cofib(\phi) =0$.\par
	To see this, let $n> 0$, assume $\pi_k\cofib(\phi)= 0$ for $k<n$ and consider the following Tor spectral sequence (see Remark \ref{Tor-ss})
	$$
	E_2^{p,q}=\Tor_p^{\pi_\ast B}(\pi_0B,\pi_\ast \cofib(\phi))_q\Rightarrow \pi_{p+q}(\pi_0B\otimes_B \cofib(\phi))= 0.
	$$
	To see that $\pi_0B\otimes_B \cofib(\phi)\simeq 0$, note that $\phi\otimes\id_{\pi_0B}$ is equivalent to the identity on $L_{\pi_0B/\pi_0A}$ and thus the cofiber of $\phi\otimes\id_{\pi_0B}$ which is $\pi_0B\otimes_B \cofib(\phi)$ vanishes.
	Let $P_\bullet$ be the graded free resolution of $\pi_*\cofib(\phi)$ constructed in Remark \ref{graded free res}. Then $E^{p,q}_2 = H^{-p}(\pi_0B\otimes_{\pi_*B}P_\bullet)_q = 0$ for $q< n$, since it is a subgroup of a quotient of  $(\pi_0B\otimes_{\pi_*B}P_p)_q$, which itself is a quotient of $(P_p)_q$.  Therefore, $\pi_n \cofib(\phi)\simeq H^0(\pi_0B\otimes_{\pi_*B}P_\bullet)_n = E^{0,n}_2\simeq \pi_n(\pi_0B\otimes_B  \cofib(\phi))\simeq 0.$\par 
	To see the first equivalence note that $H^0(\pi_0B\otimes_{\pi_*B}P_\bullet) \simeq  \pi_0B\otimes_{\pi_*B} H^0(P_\bullet)\simeq \pi_0B\otimes_{\pi_*B}\pi_*\cofib(\phi) \simeq \pi_*\cofib(\phi)$ (see \cite[09LL]{stacks-project} for the definition of the tensor product of dg-modules and note that taking  the $0$th-cohomology is just taking a cokernel which commutes with tensor products).
	\par 
	Now assume that $L_{B/A}$ is finite projective and $\pi_0A\rightarrow \pi_0B$ is of finite presentation. Consider the pushout square
	$$
	\begin{tikzcd}
	A\arrow[r,""]\arrow[d,""]&B\arrow[d,""]\\
	\pi_0A\arrow[r,""]&C.
	\end{tikzcd}
	$$
	We want to show that the natural morphism $C\rightarrow \pi_0C\simeq \pi_0B$ is an equivalence. For this assume there is a smallest integer $i>0$ such that $\pi_iC \not =0$. We get a fiber sequence 
	$$L_{C_{\leq i}/\pi_0 A}\otimes_{C_{\leq i}}\pi_0C\rightarrow L_{\pi_0C/\pi_0A}\rightarrow L_{\pi_0C/C_{\leq i}},$$
	using Lemma \ref{cotangent of truncation} (actually its proof), we see that
	$$\pi_i(L_{C/\pi_0 A}\otimes_{C}\pi_0C) \simeq \pi_i(L_{C_{\leq i}/\pi_0 A}\otimes_{C_{\leq i}}\pi_0C)\simeq \pi_{i+1}(L_{\pi_0C/C_{\leq i}})\simeq \pi_i(C).$$ But $L_{C/\pi_0 A}\otimes_{C}\pi_0C$ is projective over $\pi_0 C$ and thus discrete which is a contradiction. Therefore, we see that $\pi_0B\simeq \pi_0A\otimes_A B$ and thus $L_{\pi_0B/\pi_0A}\simeq L_{B/A}\otimes_{B}\pi_0B$ is discrete. Hence, with \cite[07BU]{stacks-project}, we see that $\pi_0A\rightarrow \pi_0B$ is smooth. \par 
	It remains to show that the natural map $\phi\colon\pi_nA\otimes_{\pi_0A}\pi_0B\rightarrow \pi_nB$ is an equivalence. But this follows from $B\otimes_A \pi_0A \simeq \pi_0B$ and \cite[Thm. 7.2.2.15]{HA} (we have to test that $B\otimes_A M$ is discrete for all discrete $A$-modules $M$, but since $\pi_0B$ is a flat $\pi_0A$-module, we see that $B\otimes_A M\simeq B\otimes_A \pi_0A\otimes_{\pi_0A}M\simeq \pi_0B\otimes_{\pi_0A}M$ is discrete).
%
\end{proof}

The following proposition is in a similar fashion to Proposition \ref{smooth cotangent}. Namely, we can characterize finitely presented morphisms by their cotangent complex and their behavior on the underlying discrete rings (i.e. on $\pi_{0}$).

\begin{prop}
	\label{lfp perfect cotangent}
	Let $f\colon A\rightarrow B$ be a morphism in $\AniAlg{R}$. 
	Then $f$ is locally of finite presentation if and only if the $B$-module $L_{B/A}$ is perfect and $\pi_0B$ is of finite presentation over $\pi_0A$\end{prop}
\begin{proof}
	We will not prove this and refer to \cite[Prop. 2.2.2.4]{TV2} or \cite[Prop. 3.2.18]{DAG}.
\end{proof}

\begin{cor}
\label{smooth imp lfp}
	Let $f\colon A\rightarrow B$ be a smooth morphism of animated $R$-algebras. Then $f$ is locally of finite presentation.
\end{cor}
\begin{proof}
	Combine Proposition \ref{lfp perfect cotangent} and \ref{smooth cotangent}.
\end{proof}

One other important fact is that for an animated ring $A$, we can lift truncated \'etale maps $B\rightarrow \pi_0A$ to \'etale maps $\widetilde{B}\rightarrow A$. The idea is to use Postnikov towers and the compatibility of cotangent complexes and truncations.

\begin{prop}
	\label{lift etale}
	Let $A$ be an animated $R$-algebra. Then the base change under the natural morphism $A\rightarrow \pi_0A$ induces a equivalence $\infty$-categories of \'etale $A$-algebras and \'etale $\pi_0A$-algebras.
\end{prop}
\begin{proof}
See \cite[Prop. 5.2.3]{CS}.
\end{proof}

Lastly, we can use Proposition \ref{lift etale} to show that any finite projective module $P$ over an animated ring $A$ is finite locally free. 

\begin{cor}
	\label{finite proj = locally free}
	Let $A$ be an animated ring and let $P$ be a finite projective module over $A$. Then there is a finite \'etale cover $(A\rightarrow A_{i})_{i\in I}$, i.e. the $A_i$ are \'etale $A$-algebras, where $I$ is finite and $A\rightarrow \prod_{i\in I}A_{i}$ is faithfully flat, such that $P\otimes_A A_i$ is free of finite rank, i.e. $P\otimes_A A_i\simeq A_i^r$ for some $r\in \NN$.
\end{cor}
\begin{proof}
	Let $\Proj(A)$ denote the full subcategory of $\MMod_A$ of projective modules. We have an equivalence of categories $\textup{h}\Proj(A)\rightarrow \textup{h}\Proj(\pi_0A)\simeq\Proj_{\pi_0A}$ given by the tensor product (see \cite[Cor 7.2.2.19]{HA}). By definition, this restricts to an equivalence of finite projective modules. Since classical finite projective modules on $\pi_0A$ are finite locally free, we know that there exists an open cover $(\widetilde{A}_i)$ of $\pi_0A$ such that $\widetilde{A}_i\otimes_{\pi_0A}\pi_0A\otimes_AP$ is equivalent to some finite free $\widetilde{A}_i$-module. By Proposition \ref{lift etale}, we can lift this open cover to an \'etale cover $A_i$ of $A$ (certainly $A\rightarrow \prod_{i\in I}A_{i}$ is \'etale and faithfully flatness can be checked on $\pi_{0}$). Since $A_i\otimes_A \pi_0A =\widetilde{A}_i$, the equivalence of the categories involved shows the claim.
\end{proof}


\section{Derived algebraic geometry}
\label{sec:der.alg.geo}

For this section, we will closely follow \cite[\S2]{TV2}, \cite{AG} and the lecture notes of Adeel Khan \cite{Khan}.\par  In \cite{TV2} To\"en and Vezzosi deal with derived algebraic geometry in the model categorical setting. This allows us to translate them easily to the $\infty$-categorical setting.\par 
In \cite{AG} Antieu and Gepner deal with spectral algebraic geometry (in the $\infty$-categorical setting). The ideas are more or less the same as we use them and in some parts, we can transport proofs one-to-one. But it is important to note that there is no higher principle which concludes derived algebraic geometry (theory of certain sheaves on animated algebras) as a corollary of spectral algebraic geometry (theory of certain sheaves on $E_{\infty}$-algebras). This is because there is no fully faithful embedding from animated algebras to $E_\infty$-algebras in general. In characteristic $0$ both $\infty$-categories are equivalent and thus the results from \cite{TV2} and \cite{AG} agree. In characteristic $p>0$ however, we have no such relation. Therefore the translation of \cite{AG} to our setting has to be treated with caution.

\subsection{Affine derived schemes}
\label{sec:affine derived schemes}
In the following $R$ will be a ring and $A$ an animated $R$-algebra.\par
Let us define the \'etale and fpqc topology.

\begin{propdef}
	Let $B$ an animated $A$-algebra. 
	\begin{enumerate}
		\item[(a)] There exists a Grothendieck topology on $\AniAlg{A}^{\op}$, called the \textup{fpqc-topology}, which can be described as follows:  A sieve (see \cite[Def. 6.2.2.1]{HTT}) $\Ccal\subseteq (\AniAlg{A}^{op})_{/B}\simeq\AniAlg{B}^{op}$ is a covering sieve if and only if it contains a finite family $(B\rightarrow B_{i})_{i\in I}$ for which the induced map $B\rightarrow \prod_{i\in I}B_{i}$ is faithfully flat.
		\item[(b)] There exists a Grothendieck topology on the full subcategory $(\AniAlg{A}^{\et})^{op}$ of \'etale $A$-algebras, called the \textup{\'etale-topology}, which can be described as follows:  A sieve $\Ccal\subseteq (\AniAlg{A}^{\et})^{\op}_{/B}\simeq(\AniAlg{B}^{\et})^{\op}$ is a covering sieve if and only if it contains a finite family $(B\rightarrow B_{i})_{i\in I}$ for which the induced map $A\rightarrow \prod_{i\in I}B_{i}$ is faithfully flat (and \'etale, which is automatic).
	\end{enumerate}
\end{propdef}
\begin{proof}
	Let $S$ be the collection of all faithfully flat (resp. faithfully flat and \'etale) morphisms in $\AniAlg{R}$. It is enough to check that $S$ satisfies the properties of \cite[Prop. A.3.2.1]{SAG}. For this, we note that a morphism of animated rings is faithfully flat (resp. faithfully flat and \'etale) if and only if it is after passage to connective $E_{\infty}$-rings. Since the functor $\theta\colon \AniAlg{A}\rightarrow \Einftycn_{\theta(A)}$ is conservative and commutes with limits and colimits (see Proposition \ref{fun E to SCR}), we see that $S$ satisfies the properties of \cite[Prop. A.3.2.1]{SAG} if and only if the collection of all faithfully flat (resp. faithfully flat and \'etale) morphisms in $\Einftycn_{\theta(A)}$ does so. But this follows from \cite[Prop. B.6.1.3]{SAG} - see also \cite[Var. B.6.1.7]{SAG} - (resp. \cite[Prop. B.6.2.1]{SAG}). 
\end{proof}

\begin{defrem}
	An \textit{affine derived scheme over $A$} is a functor from $\AniAlg{A}$ to spaces (i.e. a presheaf on $\AniAlg{A}^{\op}$), which is equivalent to $\Spec(B)\coloneqq \Hom_{\AniAlg{A}}(B,-)$ for some $B\in\AniAlg{A}$.\par 
	Note that $\Spec(B)$ is an fpqc sheaf by \cite[D. 6.3.5]{SAG} and Lemma \ref{fun E to SCR}
\end{defrem}

\begin{defi}
	Let $\Pbf$ be one of the following properties of a morphism animated rings: \textit{flat, smooth, \'etale, locally of finite presentation}. We say that a morphisms of affine derived schemes $\Spec(B)\rightarrow \Spec(C)$ has property $\Pbf$ if the underlying homomorphism $C\rightarrow B$ has $\Pbf$.
\end{defi}

\begin{rem}
	Let us remark that the above properties of morphisms of affine derived schemes are stable under equivalences, composition, pullbacks and are \'etale local on the source and target. This follows from classical theory (as found for example in \cite{stacks-project}) and Propositions \ref{smooth cotangent} and \ref{lfp perfect cotangent}, noting Proposition \ref{projective lift} and that (perfect) modules modules satisfy descent (see Remark \ref{perf fpqc local}).
\end{rem}

\begin{defi}
	For a discrete ring $A$, we set $$\Spec(A)_{\cl}\coloneqq \Hom_{\Ring}(A,-)\colon \Ring\rightarrow \Sets$$ to be its underlying classical scheme. We will abuse notation and denote the underlying locally ringed space of $\Spec(A)_{\cl}$ the same.
\end{defi}

\begin{remark}
The notation $(-)_{\cl}$ is introduced since even for a discrete ring $A$ the corresponding derived stack $\Spec(A)$ is a sheaf with values in spaces. Thus, for a (possibly non discrete) animated ring $B$ the space $\Hom_{\Ani}(A,B)$ need not to be discrete, e.g. $\Hom_{\Ani}(\ZZ[X],B)\simeq \Omega^\infty B$. But for example if we restrict ourself to discrete rings $C$, we have $\Hom_{\Ani}(A,C)\simeq \Hom_{\Ring}(\pi_0A,C)$ by adjunction, even when $A$ is not discrete.
\end{remark}

\subsection{Geometric stacks}\label{sec:geometric-stacks}

In this section we closely follow \cite{AG} and \cite{TV2}. In \cite{TV2} the notion of geometric stacks can be found in the context of model categories. Our notion agrees with the notions presented in \cite{TV2} (note that they speak of $n$-representable morphisms rather than $n$-geometric) and we want to remark that the definition of geometric stacks in \cite{AG} is different from ours. The main point are the $0$-geometric stacks. In our case, any scheme will be $1$-geometric, whereas in \cite{AG} a qcqs scheme with non-affine diagonal will not be $1$-geometric. Nevertheless, as the principal of the definition is analogous, the ideas presented in \cite{AG} often times agree with the ideas presented in \cite{TV2}.\par 

\begin{defi}
	Let $A$ be an animated ring. A \textit{derived stack over $A$} is a sheaf of spaces on $(\AniAlg{A}^{\et})^{\op}$.
	We denote the $\infty$-category of derived stacks over $A$ by $\dSt_A$. 
	If $A=\ZZ$, we simply say derived stack and denote the $\infty$-category of derived stacks by $\dSt$.
\end{defi}

\begin{rem}
\label{sheaf localization}
Let us remark, that this definition makes sense if we do not assume that $\AniAlg{A}$ for any animated ring $A$ is small. The reason for this is that a priori the full subcategory of derived stacks in $\Pcal((\AniAlg{A}^{\et})^{\op})$ is defined as the full subcategory of $S$-local objects, where $S$ consists of monomorphisms $U\hookrightarrow \Spec(A)$ that come from a covering sieve in $(\AniAlg{A}^{\et})^{\op}$. But if we assume smallness of $\AniAlg{A}$ one sees that $\dSt$ is a topological localization of $\Pcal((\AniAlg{A}^{\et})^{\op})$ (see \cite[Prop. 6.2.2.7]{HTT}) which can come quite handy for example in Remark \ref{rem.sheaf.kan}, where we use \cite[Prop. 5.5.4.2]{HTT}.
\end{rem}

\begin{defi}[\protect{\cite[\S6.2.3]{HTT}}]
Let $A$ be an animated ring. A morphism $f\colon X\rightarrow Y$ in $\dSt_{A}$ is an \textit{effective epimorphism} if the natural map $\colim_{\Delta}\Cv(X/Y)_{\bullet}\rightarrow Y$ in $\Pcal((\AniAlg{A}^{\et})^{\op})$ is an equivalence in $\dSt_{A}$ after sheafifcation\footnote{\label{foot.sheafi}We can describe $\dSt_{A}$ as a localization of $\Pcal(\AniAlg{A}^{\op})$ (see \cite[\S 6.2.2]{HTT}), so we get a functor $L\colon \Pcal(\AniAlg{A}^{\op})\rightarrow \dSt_{A}$ left adjoint to the inclusion, which we denote as sheafification.}.
\end{defi}

\begin{remark}[Effective epimorphisms]
	\label{effective epi}
	Let $f\colon \Spec(B)\rightarrow \Spec(A)$ be a morphism of affine derived schemes. By \cite[Prop. 7.2.1.14]{HTT} the map $f$ is an effective epimorphism in $\dSt$ if and only if its $0$-truncation, in the sense of \cite[\S 5.5.6]{HTT}, is an effective epimorphism. Up to homotopy $\tau_{\leq 0}\Spec(A)$ is given by the sheafification of $\pi_0\Hom_{\Ani}(A,-)$ (see \cite[Prop. 5.5.6.28]{HTT} together with Remark \ref{sheaf localization}). Thus $f$ is an effective epimorphism if and only if the sheafification of the induced map $\pi_0f\colon\pi_0\Hom(B,-)\rightarrow \pi_0\Hom(A,-)$ is an epimorphism (note that a priori \cite[Prop. 7.2.1.14]{HTT} tells us that $\pi_0f^{\sim}$ needs to be an effective epimorphism but in a $1$-topos every epimorphism is effective (see \cite[IV 7. Thm. 8]{MLM})).\par 
	We will see in Remark \ref{truncation} that the restriction of an affine derived scheme $\Spec(A)$ onto $\Ring$ preserves limits and colimits. In particular, any effective epimorphism of affine derived schemes $\Spec(A)\rightarrow \Spec(B)$ induces by adjunction an epimorphism of \'etale sheaves of sets $\Spec(\pi_{0}A)_{\cl}\rightarrow \Spec(\pi_{0}B)_{\cl}$, which in turn implies that the morphism of the underlying topological spaces of affine schemes, denoted by $|\Spec(\pi_{0}A)_{\cl}|\rightarrow|\Spec(\pi_{0}B)_{\cl}|$, is surjective. So for example if $B\rightarrow A$ is flat, then the effective surjectivity implies that it is in fact faithfully flat.
\end{remark}

\begin{defi}[\protect{\cite[Def. 1.3.3.1]{TV2}}]
	We will define a geometric morphism inductively.
	\begin{enumerate}
		\item[(1)] A derived stack is \textit{$(-1)$-geometric} or \textit{affine} if it is equivalent to an affine derived scheme.\par 
		A morphism of derived stacks $X\rightarrow Y$ is \textit{$(-1)$-geometric} or \textit{affine} if for all affine schemes $\Spec(A)$ and all $\Spec(A)\rightarrow Y$ the base change $X\times_{Y}\Spec(A)$ is affine.\par 
		A $(-1)$-geometric morphism of derived stacks $X\rightarrow Y$ is \textit{smooth} (resp. \textit{\'etale}) if for all affine derived schemes $\Spec(A)$ and all morphisms $\Spec(A)\rightarrow Y$ the base change morphism $\Spec(B)\simeq X\times_Y\Spec(A)\rightarrow \Spec(A)$ corresponds to a smooth (resp. \'etale) morphism of animated rings.\\ \par 
		Now let $n\geq 0$.
		\item[(2)] An \textit{$n$-atlas} of a derived stack $X$ is a family $(\Spec(A_i)\rightarrow X)_{i\in I}$ of morphisms of derived stacks, such that
		\begin{enumerate}
			\item[(a)] each $\Spec(A_i)\rightarrow X$ is $(n-1)$-geometric and smooth, and
			\item[(b)] the induced morphism $\coprod \Spec(A_i)\rightarrow X$ is an effective epimorphism.
		\end{enumerate}
		If each of the morphisms $\Spec(A_{i})\rightarrow X$ is \'etale, then we call the $n$-atlas \textit{\'etale}.
		\par
		A derived stack is called \textit{$n$-geometric} (resp. \textit{$n$-DM}), if 
		\begin{enumerate}
			\item[(a)] it has an $n$-atlas (resp. \'etale $n$-atlas), and
			\item[(b)] the diagonal $X\xrightarrow{\Delta}X\times X$ is $(n-1)$-geometric.
		\end{enumerate}
		\item[(3)] A morphism $X\rightarrow Y$ of derived stacks is called \textit{$n$-geometric} (resp. \textit{$n$-DM}) if for all affine derived scheme $\Spec(A)$ and all morphisms $\Spec(A)\rightarrow Y$ the base change $X\times_Y \Spec(A)$ is $n$-geometric (resp. \textit{$n$-DM}).\par 
		An $n$-geometric morphism $X\rightarrow Y$ of derived stacks is called \textit{smooth}, if for all affine derived scheme $\Spec(A)$ and all morphisms $\Spec(A)\rightarrow Y$ the base change $X\times_Y \Spec(A)$ has an $n$-atlas given by a family of affine derived schemes $(\Spec(A_i))_{i\in I}$, such that the induced morphism $A\rightarrow A_i$ is smooth.\par
		An $n$-DM morphism $X\rightarrow Y$ of derived stacks is called \textit{\'etale}, if for all affine derived scheme $\Spec(A)$ and all morphisms $\Spec(A)\rightarrow Y$ the base change $X\times_Y \Spec(A)$ has an \'etale $n$-atlas given by a family of affine derived schemes $(\Spec(A_i))_{i\in I}$, such that the induced morphism $A\rightarrow A_i$ is \'etale.\par 
	\end{enumerate}
	We call a morphism of derived stacks \textit{geometric} (resp. \textit{DM}) if it is $n$-geometric (resp. $n$-DM) for some $n\geq -1$.
\end{defi}

\begin{remark}
	From this definition one can see that an $n$-geometric morphism of derived stacks is automatically $(n+1)$-geometric. 
\end{remark}


\begin{defi}
	\label{def P-geometric}
	Let $\Pbf$ be a property of affine derived schemes that is stable under equivalences, pullbacks, compositions and is smooth-local on the source and target, then we say a morphism of derived stacks $X\rightarrow Y$ has $\Pbf$ if it is geometric and for an affine $(n-1)$-atlas $(U_i)_{i\in I}$ of the pullback along an affine derived scheme $\Spec(B)$ the corresponding morphism $U_i\rightarrow \Spec(B)$ of affine schemes has $\Pbf$.
\end{defi}

\begin{lem}
	The properties ``locally of finite presentation'', ``flat'' and ``smooth'' of morphisms of affine derived schemes satisfy the conditions of Definition \ref{def P-geometric}.
\end{lem}
\begin{proof}
	For flat and smooth this follows from the definition (note that on $\pi_{0}$ this follows from classical theory and since smooth covers are in particular flat and on $\pi_{0}$ faithfully flat, we see that the compatibility of the higher homotopy groups in the definition of smooth and flat morphisms is also clear).\par
	For locally of finite presentation, we use Proposition \ref{lfp perfect cotangent} and the fact that on $\pi_{0}$ this follows from classical theory and that perfectness of modules can be checked fpqc-locally (we will see this in Remark \ref{perf fpqc local}).
\end{proof}

Our definition above differs from the theory developed in \cite{TV2}, where they assume the property to be \'etale-locally on the source and base. This would include the property \textit{\'etale} but makes no sense for geometric stacks, since this would imply that a morphism of affine derived schemes is \'etale if and only if it is smooth locally \'etale but the next remark shows that this can not hold. We assume that there was a mixup in \cite{TV2}, since if one looks at later proofs where the condition \'etale is used one needs a stronger condition than geometric. This is not surprising, since this problem occurs even in the classical theory for Artin stacks. One can solve this for example by assuming that \'etale morphisms are always DM in the sense that after base change to an affine the resulting stack is actually a DM-stack, i.e. has an \'etale cover by schemes. We did this analogously but want to mention that this is only done for completion and is not used later on in any of the proofs.

\begin{rem}
	We want to remark that the property \textit{\'etale} is not smooth local on the base, since if it would be smooth local in our context, then it would be smooth local in classical theory of schemes which it is not.
\end{rem}

\begin{defi}
	A morphism of derived stacks is called \textit{\'etale} if it is DM and \'etale.
\end{defi}

\begin{defi}
	A morphism of derived stacks $X\rightarrow Y$ is
	\begin{enumerate}
		\item an \textit{open immersion} if it is a flat, locally of finitely presentation and a monomorphism, where flat is in the sense of Definition \ref{def P-geometric} and monomorphism means $(-1)$-truncated in the sense of \cite{HTT}, i.e. the homotopy fibers of $X\rightarrow Y$ are either empty or contractible.
		\item	a \textit{closed immersion} if it is affine and for any $\Spec(B)\rightarrow Y$ the corresponding morphism $X\times_Y\Spec(B)\simeq \Spec(C)\rightarrow \Spec(B)$ induces a surjection $\pi_0B\rightarrow\pi_0C$ of rings.
	\end{enumerate}
\end{defi}

\begin{defi}
	\label{defi immersion}
	A morphism $f\colon X\rightarrow Y$ of derived stacks is a \textit{locally closed immersion} if for all affine derived schemes $\Spec(A)$ and all morphisms $\Spec(A)\rightarrow Y$ the base change morphism $X\times_{Y}\Spec(A)\rightarrow \Spec(A)$ factors as a closed immersion follows by an open immersion.
\end{defi}

\begin{rem}
	The definition of a closed immersion does not impose any monomorphism condition. This makes sense, since we will see that a monomorphism automatically has vanishing cotangent complex (see Lemma \ref{cotangent monomorphism}) and in particular, any closed immersion which is on $t_{0}$ of finite presentation and a monomorphism will be \'etale (this can proven analogously to \cite[Cor. 2.2.5.6]{TV2}). But as many naturally arising closed immersions are not \'etale, e.g. any regular immersion with non-vanishing cotangent complex, the above definition seems to be the one suited for the world of derived algebraic geometry.
\end{rem}

Let us give an important example of an open immersion of derived stacks.

\begin{lem}
	Let $A$ be an animated $R$-algebra and let $f\in \pi_0A$ be an element. The inclusion $j\colon \Spec(A[f^{-1}])\hookrightarrow \Spec(A)$ is an open immersion.
\end{lem}
\begin{proof}
	The proof is given in \cite[Lec. 3 Lem. 4.2]{Khan}. But for the convenience of the reader, we recall the proof.\par 
	We have to check that $j$ is a monomorphism, which is flat and locally of finite presentation. Locally of finite presentation follows from Lemma \ref{localization lfp}. Flatness, i.e. $\pi_0A[f^{-1}]\otimes_{\pi_0A}\pi_iA\simeq \pi_iA[f^{-1}]$ follows from $\pi_i(A[f^{-1}])=\pi_i(A)[f^{-1}]$. To see that it is a monomorphism, we have to show that the homotopy fibers of $\Hom(A[f^{-1}],B)\rightarrow \Hom(A,B)$ for any $B\in \AniAlg{R}$ are either empty or contractible. But this follows from the general property of localization (see Lemma \ref{localization}). 
\end{proof}

\begin{defi}
	A derived stack is called \textit{separated} if the diagonal is a closed immersion.	
\end{defi}

\begin{defi}
	A derived stack $X$ is \textit{quasi-compact} if there exists an $n$-atlas consisting of a single affine. A morphism $f\colon X\rightarrow Y$ of derived stacks is \textit{quasi-compact} if for all affine derived schemes $\Spec(A)$ and all morphisms $\Spec(A)\rightarrow Y$ the base change $X\times_{Y}\Spec(A)$ is quasi-compact.
\end{defi}

\begin{remark}
	Since affine schemes are separated, we see that affine derived schemes are also separated (note that the diagonal of an affine scheme is representable and that we only have to check that the corresponding ring morphism on $\pi_0$ is surjective, which follows from classical theory).
\end{remark}

\begin{defi}
	A derived stack $X$ is \textit{locally geometric} if we can write $X$ as the filtered colimit of geometric derived stacks $X_i$, with open immersions $X_i\hookrightarrow X$.\par 
	We say that locally geometric stack $X\simeq \colim_{i\in I} X_{i}$ is \textit{locally of finite presentation} if each $X_{i}$ is locally of finite presentation. 
\end{defi}

\begin{defprop}
	For a morphism of derived stacks $f\colon X\rightarrow Y$, we define $\Im(f)$ as an epi-mono factorisation $X\twoheadrightarrow \Im(f)\hookrightarrow Y$ of $f$ (here ``epi'' means ``effective epimorphism''). This factorisation is unique up to homotopy.
\end{defprop}
\begin{proof}
	The existence of such a factorisation follows from \cite[Ex. 5.2.8.16]{HTT}. The uniqueness up to homotopy follows from \cite[5.2.8.17]{HTT}
\end{proof}

\begin{remark}
	In the reference used in the above proof one shows that the image of a morphism $f\colon X\rightarrow Y$ is equivalent to the $(-1)$-truncation of $f$, which in turn is equivalent to the colimit of the \v{C}ech nerve of $f$ (see \cite[Cor. 6.2.3.5]{HTT} and note that per definition $0$-connective morphisms are effective epimorphisms, see. \cite[Def. 6.5.1.10]{HTT}).
\end{remark}

The following lemmas are clear from the definitions and may seem unnecessary complicated but will enable us to give another definition of open immersion, which shows that we won't have to deal with geometricity of open immersions.

\begin{lem}
\label{diag of mono}
	Let $\iota\colon U\hookrightarrow \Spec(A)$ be a monomorphism of derived stacks. Then $U$ has an affine diagonal.
\end{lem}
\begin{proof}
	Since $\iota$ is a monomorphism, we see that the diagonal of $\iota$ is an equivalence and hence we conclude.
\end{proof}

\begin{lem}
	\label{smooth affine geometric}
	Let $(A_i)_{i\in I}$ be a family of animated $B$-algebras having the property \textbf{P}, where \textbf{P} is as in Definition \ref{def P-geometric}. Let $U$ denote the image of the natural map $\coprod_{i\in I}\Spec(A_i)\rightarrow \Spec(B)$ and assume that the base change of $\Spec(A_{i})\rightarrow U$ with any derived affine scheme and any morphism $\Spec(C)\rightarrow U$ is smooth (note that this makes sense by Lemma \ref{diag of mono}) and has property \textbf{P}. Then the $\coprod_{i\in I} \Spec(A_i)$ is a $0$-atlas for $U$, via the natural map and in particular $U\hookrightarrow \Spec(B)$ is $0$-geometric and has property \textbf{P}.
\end{lem}
\begin{proof}
	This follows from the definitions.
\end{proof}

\begin{lem}
	\label{mono P-geometric}
	Let $U\hookrightarrow X$ be a monomorphism of derived stacks and let \textbf{P} be a property as in Definition \ref{def P-geometric}. Assume the base change with any $\Spec(A)\rightarrow X$ has a cover by a disjoint union of affine derived schemes over $A$ such that the conditions of Lemma \ref{smooth affine geometric} are satisfied. Then $U\hookrightarrow X$ is $0$-geometric and has property \textit{P}.
\end{lem}
\begin{proof}
	This follows from Lemma \ref{smooth affine geometric}.
\end{proof}

\begin{lem}
\label{mono geometric}
	Let $\iota\colon U\hookrightarrow X$ be a geometric monomorphism of derived stacks. Then $\iota$ is $0$-geometric.
\end{lem}
\begin{proof}
	This follows from Lemma \ref{mono P-geometric}.
\end{proof}

\begin{rem}
	Lemma \ref{mono geometric} implies for example that open immersions and locally closed immersions are automatically $0$-geometric.
\end{rem}

\begin{lem}
	\label{open immersion geometric}
	A morphism $U\rightarrow X$ of derived stacks is an open immersion if and only if for any $\Spec(A)\rightarrow X$  the base change is a monomorphism and there is an effective epimorphism $\coprod_{i\in I}\Spec(A_i)\rightarrow \Spec(A)\times_X U$ such that each $\Spec(A_i)\rightarrow\Spec(A)$ is an open immersion.
\end{lem}
\begin{proof}
	This follows from Lemma \ref{mono P-geometric}.
\end{proof}

We can also define derived versions of schemes with the notion of open immersions.
\begin{defi}
	Let $X$ be a derived stack. Then $X$ is a \textit{derived scheme} if it admits a cover $(\Spec(A_i)\hookrightarrow X)_{i\in I}$ such that each $\Spec(A_i)\hookrightarrow X$ is an open immersion (in particular $X$ is $1$-geometric).
\end{defi}

\begin{remark}
	If we have a morphism $\coprod_{i\in I} U_i\rightarrow X$ of derived stacks where each $U_i$ is an open immersion, we sometimes write $\bigcup_{i\in I} U_i$ for its image. \par 
	If $X\rightarrow Y$ is a morphism of derived stacks, where the diagonal of $Y$ is representable and $X$ is a derived scheme, then the image of $X\rightarrow Y$ is a derived scheme.\par 
	If $\coprod_{i\in I} \Spec(A_i)\rightarrow X$ is a morphism of derived stacks, where $X$ has representable diagonal and $\Spec(A_i)$ are affine open in $X$, then $\bigcup_{i\in I}\Spec(A_i)$ is an open substack of $X$.
\end{remark}

\begin{remdef}[Truncation]
	\label{truncation}
	We give a quick summary of \cite[\S 2.2.4]{TV2}.
	\par
	Let $\Abf$ be the model category of simplicial commutative $R$-algebras, as explained in section \ref{sec:simplicial commutative algebras}.
	The inclusion $\Alg{R}\hookrightarrow \Abf$ has a left adjoint $\pi_0\colon \Abf\rightarrow \Alg{R}$. This is a Quillen adjunction for the trivial model structure on $\Alg{R}$. This induces an adjunction $\begin{tikzcd} \pi_0\colon \AniAlg{R}\arrow[r,"",shift left = 0.8]&\arrow[l,"",shift left = 0.8]\Alg{R}\colon i \end{tikzcd}$. We therefore get adjunctions
	$$
	\begin{tikzcd} i_{!}\colon\Pcal(\Alg{R}^{\op}) \arrow[r,"",shift left = 0.8]&\arrow[l,"",shift left = 0.8]\Pcal(\AniAlg{R}^{\op})\colon i^\ast\colon \Pcal(\AniAlg{R}^{\op})\arrow[r,"",shift left = 0.8]&\arrow[l,"",shift left = 0.8]\Pcal(\Alg{R}^{\op}) \colon \pi_0^\ast \end{tikzcd}
	$$  
	(here $i^\ast$ (resp. $\pi_0^\ast$) is defined as the restriction of a presheaf and $\Pcal(\Ccal)$ denotes the $\infty$-category of presheaves of spaces on $\Ccal$ and $i_{!}$ is given by left Kan extension).
	The inclusion $\Alg{R}\hookrightarrow \Abf$ induces an equivalence to the discrete animated $R$-algebras and thus the restriction preserves $i^{*}$ preserves sheaves and thus composing $i_{!}$ with the sheafification, we get the adjunction
	$$
	\begin{tikzcd} i_{!}\colon  \Shv_{\et}(\Alg{R})\arrow[r,"",shift left = 0.8]&\arrow[l,"",shift left = 0.8]\Shv_{\et}(\AniAlg{R})\colon i^\ast\colon\Shv_{\et}(\AniAlg{R}) \arrow[r,"",shift left = 0.8]&\arrow[l,"",shift left = 0.8] \Shv_{\et}(\Alg{R})\colon \pi_0^\ast. \end{tikzcd}
	$$  
	For convenience later on, we define $t_0\coloneqq i^\ast$ and $\iota\coloneqq i_{!}$. Note, that by general theory of Kan extensions, the functor $\iota$ is indeed fully faithful (see \cite[\S 4.3.2]{HTT}).\par  
	For a derived scheme $X\in\Shv(\AniAlg{R})^{\et}$, we denote its image under $t_0$ with $X_{\cl}$ and call it the \textit{underlying classical scheme}. Note that $t_0(\Spec(A)) \simeq \Spec(\pi_0A)_{\cl}$ and $\iota(\Spec(\pi_0A)_{\cl})\simeq \Spec(\pi_0A)$ (by adjunction and the fact that any morphism from an animated ring to a discrete ring is characterized by the corresponding morphisms on discrete rings). Also by adjunction being an effective epimorphism is preserved under $t_{0}$. So if $\coprod_{i\in I}\Spec(A_i)$ is a Zariski atlas of $X$, we see that $\coprod_{i\in I}\Spec(\pi_0A_i)_\cl$ is a cover of $X_{\cl}$ and thus $X_{\cl}$ has values in discrete spaces, i.e. sets (note that \'etale locally any morphism $\Spec(B)\rightarrow X_{\cl}$ factors through $\coprod_{i\in I}\Spec(\pi_0A_i)_\cl$, in particular the points of $X_{\cl}$ can be computed by the points of its atlas, which are discrete). Hence, $X_{\cl}$ recovers the classical notion of a scheme.\par 
	Let us state a few interesting properties of $t_0$ and $i$.
	\begin{enumerate}
		\item	The functor $t_0$ has a right and left adjoint (see above),
		\item	the functor $t_0$ preserves geometricity (here geometricity of sheaves in $\Alg{R}$ is defined similarly to derived stacks, see \cite[\S 2.1.1]{TV2} for further information) and the properties flat, smooth and \'etale along geometric morphisms,
		\item	the functor $\iota$ preserves geometricity, homotopy pullbacks of $n$-geometric stacks along flat morphisms and sends flat (resp. smooth, \'etale) morphisms of $n$-geometric stacks to flat (resp. smooth, \'etale) morphisms of $n$-geometric stacks,
		\item if $X\in\Shv(\Alg{R})^{\et}$ is $n$-geometric and $X'\rightarrow \iota(X)$ is a flat morphism, then $X'$ is the image of an $n$-geometric stack under $\iota$.
	\end{enumerate}
	A proof for these statements is given in \cite[Prop. 2.2.4.4]{TV2}.
\end{remdef}

We list some properties of geometric morphisms of derived stacks.

\begin{lem}
	\label{geometric base change}
	Let $X\rightarrow Z$ and $Y\rightarrow Z$ be morphisms of derived stacks. If $X\rightarrow Z$ is $n$-geometric, then so is $X\times_Z Y\rightarrow Y$.
\end{lem}
\begin{proof}
	This follows immediately from the definition.
\end{proof}

\begin{lem}
	\label{geometric local}
	A morphism of derived stacks $X\rightarrow Y$ is $n$-geometric if and only if the base change under $\Spec(A)\rightarrow Y$ for any $A\in \AniAlg{R}$ is $n$-geometric.
\end{lem}
\begin{proof}
	This follows immediately from the definitions.
\end{proof}

\begin{lem}
	\label{geometric comp}
	Let $f\colon X\rightarrow Y$ and $g\colon Y\rightarrow Z$ be morphisms of derived stacks. If $f$ and $g$ are $n$-geometric, then so is $g\circ f$.
\end{lem}
\begin{proof}
	The proof is straightforward using induction on $n$ (see \cite[Prop. 1.3.3.3 (3)]{TV2}).
\end{proof}

\begin{prop}
	\label{diag geometric}
	Let $f\colon X\rightarrow Y$ be a morphism of derived stacks. Assume $X$ is $n$-geometric and the diagonal $Y\rightarrow Y\times Y$ is $n$-geometric. Then $f$ is $n$-geometric.
\end{prop}
\begin{proof}
	This is analogous to \cite[Lem. 4.30]{AG}.\par
	Let $\coprod_{i\in I} \Spec(A_i)\twoheadrightarrow X$ be an $n$-atlas. Consider a morphism $\Spec(A)\rightarrow Y$, where $A$ is an animated ring. Then we have a morphism $\coprod_{i\in I} \Spec(A_i)\times_Y \Spec(A)\rightarrow \coprod_{i\in I} \Spec(A_i\otimes A)$, which is $n$-geometric, since it is the base change of the diagonal under $\coprod_{i\in I} \Spec(A_i\otimes A)$. Therefore $\coprod_{i\in I}\Spec(A_i)\times_Y \Spec(A)$ has an $n$-atlas, say given by $(\Spec(B_j)\rightarrow\coprod_{i\in I}\Spec(A_i)\times_Y \Spec(A))_{j\in J}$ and by Lemma \ref{geometric comp}, we see that $(\Spec(B_j)\rightarrow X\times_Y\Spec(A))_{j\in J}$ is an $n$-atlas. This finishes the proof, but to make things clear, we finally get following diagram with pullback squares
	$$
	\begin{tikzcd}
	  	\coprod_{j\in J} \Spec(B_j)\arrow[d,"n\textup{-atlas}",two heads,swap]&\coprod \Spec(A_i\otimes A) &\\
		\coprod_{i\in I} \Spec(A_i)\times_Y \Spec(A)\arrow[ru]\arrow[r,"n\textup{-geom.}",two heads]\arrow[d,""]&X\times_Y\Spec(A)\arrow[r,""]\arrow[d,""]&\Spec(A)\arrow[d,""]\\
		\coprod A_i\arrow[r,"n\textup{-atlas}",two heads]&X\arrow[r,""]&Y.
	\end{tikzcd}
	$$
\end{proof}

\begin{cor}
	Let $X$ and $Y$ be $n$-geometric stacks. Then any morphism $X\rightarrow Y$ is $n$-geometric.
\end{cor}
\begin{proof}
	This follows immediately from the definitions and Proposition \ref{diag geometric}
\end{proof}

\begin{prop}
	\label{diag + proj geometric}
	Let $X\rightarrow Y$ be an effective epimorphism of derived stacks and suppose that $X$ and $X\times_YX$ are $n$-geometric. Further, assume that the projections $X\times_Y X\rightarrow X$ are $n$-geometric and smooth. Then $Y$ is an $(n+1)$-geometric stack. If in addition $X$ is quasi-compact and $X\rightarrow Y$ is a quasi-compact morphism, then $Y$ is quasi-compact. Finally if $X$ is locally of finite presentation, then so is $Y$.
\end{prop}
\begin{proof}
This is analogous to \cite[Lem. 4.29]{AG}.\par
	Let $\coprod_{i\in I} \Spec(A_i)\twoheadrightarrow X$ be an $n$-atlas. Consider the following diagram with pullback squares
	$$
	\begin{tikzcd}
	\Spec(A_i)\times_Y \Spec(A_j)\arrow[d,""]\arrow[r,""]&X\times_Y \Spec(A_j)\arrow[d,""]\arrow[r,""]& \Spec(A_j)\arrow[d,""]\\
	\Spec(A_i)\times_Y X\arrow[d,""]\arrow[r,""]&X\times_YX\arrow[r,""]\arrow[d,""]&X\arrow[d,""]\\
	\Spec(A_i)\arrow[r,""]&X\arrow[r,"f"]&Y.
	\end{tikzcd}
	$$
	It suffices to show that $\coprod_{i\in I}\Spec(A_i)\rightarrow X\rightarrow Y$ is an $(n+1)$-atlas. Since the projections $X\times_Y X\rightarrow X$ and $\Spec(A_i)\rightarrow X$ are $n$-geometric smooth, we will see that $\Spec(A_i)\rightarrow X\rightarrow Y$ is $n$-geometric smooth, proving our claim.\par 
	Indeed, let $\Spec(C)\rightarrow Y$ be a morphism from an affine derived scheme. Consider the base change $\coprod_{j\in I} \Spec(A_j)\times_Y \Spec(C)\twoheadrightarrow \Spec(C)$, which is an effective epimorphism. In particular, we can find an \'etale covering $\Spec(\widetilde{C})\rightarrow \Spec(C)$, which factors through $\coprod_{j\in I} \Spec(A_j)\times_Y \Spec(C)$. To show that $\Spec(A_i)\times_Y \Spec(C)$ has an $n$-atlas, it suffices to check that the base change with $\widetilde{C}$ has an $n$-atlas (see \cite[Prop. 1.3.3.4]{TV2}). Now let us look at the following diagram with pullback squares 
	$$
	\begin{tikzcd}
	Z\arrow[r,""]\arrow[d," f'"]& W\arrow[r,""]\arrow[d,""]& \Spec(\widetilde{C})\arrow[d,"f"]\\
	\coprod_{j\in I} \Spec(A_i)\times_Y \Spec(A_j)\times_Y C\arrow[d,"g'", two heads]\arrow[r,"h"]&\coprod_{j\in I} X\times_Y \Spec(A_j)\times_Y C\arrow[d,"", two heads]\arrow[r,"l"]& \coprod_{j\in I}\Spec(A_j)\times_Y C\arrow[d,"g", two heads]\\
	C\times_Y \Spec(A_i)\arrow[d,""]\arrow[r,""]&C\times_YX\arrow[r,""]\arrow[d,""]&C\arrow[d,""]\\
	\Spec(A_i)\arrow[r,""]&X\arrow[r,""]&Y.
	\end{tikzcd}
	$$
	Since $ g\circ f$ is affine \'etale effective epimorphism, we know that $g'\circ f'$ is affine \'etale effective epimorphism. 
	Since the projections are $n$-geometric smooth and by the above the projection $\Spec(A_i)\times_Y \Spec(A_j)\rightarrow \Spec(A_i)$ is $n$-geometric smooth, we see that $l\circ h$ is $n$-geometric smooth. Therefore, $Z$ has an $n$-atlas and since $g\circ f'$ is affine \'etale, we see that the $n$-atlas of $Z$ gives an $n$-atlas of $\Spec(A_i)\times_Y \Spec(C)$.\par
	The rest of the statement follows immediately by the definitions.
\end{proof}

We conclude this section with an important remark. This remark shows, for an animated ring $A$, how open subschemes of $\Spec(\pi_{0}A)_{\cl}$ can be lifted to derived open subschemes of $\Spec(A)$. In particular, when we want to show that an inclusion of derived stacks is an open immersion, it suffices to show that it is an open immersion after applying $t_{0}$.

\begin{remark}[Lifting opens along affines]
	\label{lift along affine}
	Let $A$ be an animated ring.
	Assume we have an open subscheme $U\hookrightarrow \Spec(\pi_0A)_{\cl}$ of an affine scheme. Let $( \Spec(\pi_0A_{f_i})_{\cl}\rightarrow U)_{i\in I} $ be an affine open cover by basis elements. Certainly, we can lift this open cover to an open subscheme $V\coloneqq \Im(\coprod_{i\in I}\Spec(A[f_i^{-1}]))$, where the image is taken in $\Spec(A)$. Let $B$ be a animated $A$-algebra with structure morphism $w\colon \Spec(B)\rightarrow \Spec(A)$. Then $w$ factors through an $u\colon \Spec(B)\rightarrow V$, i.e. $u\in V(B)$, if and only if there is an \'etale cover $(B\rightarrow B_j)_{j\in J}$ such that for every $j$ there is an $i$ with $\pi_0w_j(f_i)$ invertible, where $w_j$ is the composition of $w$ with the natural map $B\rightarrow B_j$. \par 
	To see this, assume we have a map $u\colon \Spec(B)\rightarrow V$ of derived $A$-schemes. Then base change with the affine open cover of $V$ gives an affine open $\coprod_{i\in I} \Spec(B_i)$ cover of $\Spec(B)$ that maps to $\coprod_{i\in I} \Spec(A[f_i^{-1}])$ via projection (note that $V$ is an open subscheme of an affine scheme and thus separated, so in particular the diagonal of $V$ is affine). The projection $\coprod_{i\in I}\Spec(B_i)\rightarrow\coprod_{i\in I} \Spec(A[f_i^{-1}])$ is induced by the termwise projections (note that coproducts in $\infty$-topoi are universal). Thus, by the universal property of localization, we see that $\pi_0w_i(f_i)$ is invertible in $\pi_0B_i$.
	\par
	Now assume there is an \'etale cover $( \Spec(B_j)\rightarrow \Spec(B))_{j\in J}$ such that for every $j$ there is an $i$ with $\Spec(B_j)\rightarrow \Spec(A[f_i^{-1}])$. In particular, we get a map $\Spec(B_j)\rightarrow \coprod_{i\in I}\Spec(A[f_i^{-1}])$ and by taking coproducts and fiber products, we get a map $$\Spec(B)\simeq\colim_{\Delta}(\Cv(\coprod_{j\in J}\Spec(B_j)/B)_{\bullet})\rightarrow \coprod_{i\in I}\Spec(A[f_i^{-1}])\rightarrow V.$$ 
\end{remark}

\subsection{Quasi-coherent modules over derived stacks}

In this section we will shortly look at quasi-coherent modules over derived stacks.  We show that they behave as ``expected''. Namely, quasi-coherent modules over derived stacks still satisfy descent\footnote{We will make this explicit later on, as we did not define Grothendieck topologies on derived stacks.}. We also have pullback and pushforward functors that are adjoint to another. Further, we show that the $\infty$-category of quasi-coherent modules (in the derived sense) over a (classical) scheme $X$ is equivalent to $\Dcal_{\textup{qc}}(X)$.\par 
We will closely follow \cite[\S I.3]{GR}, \cite{DAG} and \cite{Khan} and generalize some results following their ideas.

\begin{defi}
	Let $X$ be a presheaf on $\AniAlg{R}^{\op}$, we define the \textit{$\infty$-category of quasi-coherent modules over $X$} to be 	$$
		\QQCoh(X)\coloneqq \lim_{\Spec(A)\rightarrow X}\MMod_A.
	$$
	An element $\Fcal\in\QQCoh(X)$ is called \textit{quasi-coherent module over $X$} or \textit{$\Ocal_{X}$-module}. For any affine derived scheme $\Spec(A)$ and any morphism $f\colon\Spec(A)\rightarrow X$, we denote the image of a quasi-coherent module $\Fcal$ under the projection $\QQCoh(X)\rightarrow\MMod_{A}$, with $f^{*}\Fcal$.\par 
	We define the $\infty$-category of \textit{perfect quasi-coherent modules over $X$} to be
	$$
		\QQCoh_{\perf}(X)\coloneqq \lim_{\Spec(A)\rightarrow X}\MMod^{\perf}_A.
	$$
	We say that a perfect quasi-coherent module $\Fcal$ over $X$ has \textit{Tor-amplitude in $[a,b]$} if for every derived affine scheme $\Spec(A)$ and any morphism $f\colon \Spec(A)\rightarrow X$ the $A$-module $f^{*}\Fcal$ has Tor-amplitude in $[a,b]$. 
\end{defi}

\begin{remark}
	\label{qqcoh right kan}
	We see that that by definition $\QQCoh(-)$ (resp. $\QQCoh_{\perf}(-)$) is a right Kan extension of $\MMod_{-}\colon \AniAlg{R}\rightarrow \ICat$ (resp. $\MMod_{-}^{\perf}\colon \AniAlg{R}\rightarrow\ICat$) onto $\Pcal(\AniAlg{R}^{\op})^{\op}$ along the Yoneda emebdding.\par 
	A priori $\MMod_{-}$ is a functor from animated rings to the $\infty$-category of not necessarily small $\infty$-categories. But for the purpose of this article if we talk about the right Kan extension along the Yoneda embedding to presheaves on $\AniAlg{\ZZ}$, we assume smallness of the module categories.
\end{remark}

\begin{rem}
\label{right kan of mono}
	Note that limits preserve monomorphisms in $\ICat$ (as this $\infty$-category has limits). Therefore, if $F,G\colon \AniAlg{R}\rightarrow\ICat$ are functors and $\alpha\colon F\rightarrow G$ is a natural transformation such that $\alpha(A)$ is a monomorphism for all $A\in\AniAlg{R}$, we see that for the induced morphism $R\alpha$ on the right Kan extensions $RF$ resp. $RG$ of $F$ resp. $G$ under the Yoneda embedding $\AniAlg{R}\hookrightarrow\Pcal(\AniAlg{R}^{\op})^{\op}$ the evaluation on some $X\in\Pcal(\AniAlg{R}^{\op})^{\op}$ yields a monomorphism $R\alpha(X)\colon RF(X)\hookrightarrow RG(X)$.\par 
	This can be applied to see that for example for all $\in\Pcal(\AniAlg{R}^{\op})^{\op}$, we have that the natural morphism $\QQCoh_{\perf}(X)\rightarrow \QQCoh(X)$ is fully faithful and we can see $\QQCoh_{\perf}(X)$ as a full subcategory of $\QQCoh(X)$.
\end{rem}

\begin{lem}
\label{limit of stable}
	Let $\Ccal$ be the limit of stable $\infty$-categories $\Ccal_{k}$ indexed by some simplicial set $K$ with finite limit preserving transition maps, then $\Ccal$ is stable.
\end{lem}
\begin{proof}
	Since $\Ccal_{k}$ have finite limits, we know that $\Ccal$ has finite limits. Then the spectrum of $\Ccal$ is stable by \cite[Cor.  1.4.2.17]{HA}, but by \cite[Rem. 1.4.2.25]{HA}, we know that $\Sp(\Ccal)$ itself is a limit of the tower $\dots\rightarrow \Ccal_{*}\xrightarrow{\Omega}\Ccal_{*}$. In particular, we have $$\Sp(\Ccal)\simeq\Sp(\lim_{K}\Ccal_{\bullet})\simeq\lim_{K}\Sp(\Ccal_{\bullet})\simeq\lim_{K}\Ccal_{\bullet}\simeq \Ccal,$$ where we use \cite[Prop. 1.4.2.21]{HA} for the second to last equivalence.	
\end{proof}

\begin{rem}
\label{qcoh stable}
	By Lemma \ref{limit of stable}, we know that for any $X\in\Pcal(\AniAlg{R}^{\op})$ the $\infty$-category $\QQCoh(X)$ is stable, since it is the limit of stable $\infty$-categories and the transition maps are given by base change (the base change functor preserves fiber sequences, as they are equivalently cofiber sequences, and finite products, that are equivalent to finite coproducts).
\end{rem}

The following proposition is a generalization of \cite[\S I.3 Cor. 1.3.11]{GR} but we can follow the idea of the proof.

\begin{prop}
	\label{right kan of sheaf}
	Let $\Ccal$ be a presentable $\infty$-category and let $F\colon \AniAlg{R}\rightarrow \Ccal$ be a (hypercomplete) sheaf with respect to the Grothendieck topology $\tau\in\lbrace \textup{fpqc, \'etale}\rbrace$ on $\AniAlg{R}$. Let $RF$ denote the right Kan extension of $F$ along the Yoneda embedding $\AniAlg{R}\hookrightarrow\Pcal(\AniAlg{R}^{\op})^{\op}$. Further let us denote the corresponding $\infty$-topos of (hypercomplete) $\tau$-sheaves on $\AniAlg{R}$ with $\Shv_{\tau}$.  Then for any diagram $p\colon K\rightarrow \Pcal(\AniAlg{R}^{\op})$, where $K$ is a simplicial set, and morphism $\colim_{K} X_{k}\rightarrow Y$ that becomes an equivalence in $\Shv_{\tau}$ after sheafification\footnote{Recall${}^{\ref{foot.sheafi}}$ that we can describe $\Shv_{\tau}$ as a localization of $\Pcal(\AniAlg{R}^{\op})$ (as seen in the proof), so we get a functor $L\colon \Pcal(\AniAlg{R}^{\op})\rightarrow \Shv_{\tau}$ left adjoint to the inclusion, which we call \textit{sheafification}.}, we have that the natural map $RF(Y)\rightarrow \lim_{K} RF(X_{k})$ is an equivalence.
\end{prop}
\begin{proof}
	First, let us note that since $\Ccal$ is presentable, we can find a small subcategory $\Ccal'\subseteq \Ccal$ such that $\Ccal$ is a localization of $\Pcal(\Ccal')$ (see \cite[Thm. 5.5.1.1]{HTT}). In particular, the elements $RF(Y)$ and $\lim_{K} RF(X_{k})$ may be regarded as functors from $\Ccal'$ to $\SS$ and the natural morphism is an equivalence if and only if it is an equivalence after composing with the evaluation for every $c\in\Ccal'$ (see \cite[01DK]{kerodon}). We note that the inclusion of $\Ccal$ into $\Pcal(\Ccal')$ preserves limits and since the evaluation of a functor $G\in \Pcal(\Ccal')$ at $c$ is equivalent to $\Hom_{\Pcal(\Ccal')}(j(c),G)$, where $j\colon \Ccal'\hookrightarrow \Pcal(\Ccal')$ denotes the Yoneda embedding (see \cite[Lem. 5.5.2.1]{HTT}), we see that also the evaluation preserves limits. So it is enough to check that for every $c\in\Ccal'$, the morphism $RF(Y)(c)\rightarrow \lim_{\Delta}(RF(X_{k})(c))$ is an equivalence. In particular, we may replace $F$ by $\Hom_{\Pcal(\Ccal')}(j(c), -) \circ F$ for any $c\in \Ccal'$ and so without loss of generality, we may assume that $\Ccal\simeq \SS$.\par
	We will first discuss the case of $\tau$-sheaves.	By definition of the $\infty$-category $\Shv_{\tau}$, we know that all $\tau$-sheaves are $S$-local, where $S$ is the collection of those monomorphism $U\hookrightarrow \Spec(A)$, where $A\in\AniAlg{R}$, such that it defines a $\tau$-covering sieve (see \cite[\S 6.2.2]{HTT} for details). This in particular defines a localization functor $L\colon \Pcal(\AniAlg{R})\rightarrow \Pcal(\AniAlg{R})$ with essential image given by $\Shv_{\tau}$. Using \cite[Prop. 5.5.4.2]{HTT}, we see that any equivalence in $\Shv_{\tau}$ is local, i.e. any morphism $f\colon U\rightarrow V$ in $\Pcal(\AniAlg{R})$ such that $Lf$ is an equivalence and any $Q\in\Shv_{\tau}$ we have that the natural morphism $$\Hom_{\Pcal(\AniAlg{R})}(V,Q)\rightarrow\Hom_{\Pcal(\AniAlg{R})}(U,Q)$$ is an equivalence. In particular, in our situation, we have that 
	$$
	\Hom_{\Pcal(\AniAlg{R})}(Y,Q)\rightarrow\Hom_{\Pcal(\AniAlg{R})}(\colim_{K}X_{k},Q)\simeq\lim_{K} \Hom_{\Pcal(\AniAlg{R})}(X_{k},Q)
	$$ 
	is an equivalence (note that the colimit in the second $\Hom$ is taken in the $\infty$-category $\Pcal(\AniAlg{R})$, whereas for $Y\simeq \colim_{K} X_{k}$ colimit is taken in $\Shv_{\tau}$ which do not agree in general since the inclusion $\Shv_{\tau}\hookrightarrow \Pcal(\AniAlg{R})$ does not preserve colimits in general).\par  Since $F$ is a sheaf with respect to the topology $\tau$, we therefore have 	$$
	\Hom_{\Pcal(\AniAlg{R})}(Y,F)\simeq \lim_{K}\Hom_{\Pcal(\AniAlg{R})}(X_{k},F).
	$$ Since we can write any presheaf on $\AniAlg{R}$ as a colimit of representable ones (see \cite[Lem. 5.1.5.3]{HTT}) and the Yoneda lemma \cite[Lem. 5.5.2.1]{HTT}, we finally have the equivalence
	 $$
		 RF(Y)\xrightarrow{\sim} \lim_{K} RF(X_{k}).
	 $$
	 \par 
	 The case of hypercomplete $\tau$-sheaves is completely analogous, noting that the $\infty$-topos of hypercomplete $\tau$-sheaves can be realized as a localization of $\Pcal(\AniAlg{R})$ with respect to hypercovers (see \cite[Cor. 6.5.3.13]{HTT}).
\end{proof}

\begin{defi}
Let $\Ccal$ be a presentable $\infty$-category and let  $\tau$ be the fpqc or \'etale topology on $\AniAlg{R}$.
	A functor $F\colon \Pcal(\AniAlg{R}^{\op})^{\op}\rightarrow \Ccal$ is a \textit{(hypercomplete) sheaf or satisfies $\tau$-descent} if for any effective epimorphism $X\rightarrow Y$ (resp. a hypercover $X^{\bullet}\rightarrow Y$), we have $$RF(Y)\simeq \lim_{\Delta}RF(\Cv(X/Y)_{\bullet})\textup{ (resp. }RF(Y)\simeq \lim_{\Delta_{s}}RF(X^{\bullet})).$$
\end{defi}

\begin{rem}
\label{rem.sheaf.kan}
	In the setting of Proposition \ref{right kan of sheaf}, we see that if $F$ is a (hypercomplete) sheaf, then so is its right Kan extension $RF$.
\end{rem}

\begin{rem}
\label{infty cat presentable}
	An important example of a presentable $\infty$-category is the $\infty$-category $\ICat$ of small $\infty$-categories. Presentability of this $\infty$-category follows from the fact that it is the $\infty$-category of a combinatorial simplicial model category (marked simplicial sets with the model structure of \cite[Prop. 3.1.5.2]{HTT}), which are precisely the presentable $\infty$-categories ( \cite[Prop. A.3.7.6]{HTT}).
\end{rem}

\begin{prop}
	\label{right kan ex derived scheme}
	Let $\Ccal$ be a presentable $\infty$-category and let $F\colon \AniAlg{R}\rightarrow \Ccal$ be an \'etale sheaf. Let $RF$ denote a right Kan extension of $F$ along the Yoneda embedding $\AniAlg{R}\hookrightarrow \Pcal(\AniAlg{R})^{\op}$. Then for any derived scheme $X$ over $R$, the natural morphism
	$$
	RF(X)\rightarrow \lim_{\substack{U\hookrightarrow X \\ \textup{affine open}}} F(U)
	$$ 
	is an equivalence.
\end{prop}
\begin{proof}
	This lemma is a generalization of \cite[Lec. 1 Prop. 3.5]{Khan} but can be proven the same. For the convenience of the reader, we give a proof.\par 
	Let $X$ be derived scheme and $Y\coloneqq \coprod_{i\in I}\Spec(A_i)\rightarrow X$ be a Zariski atlas. By Remark \ref{rem.sheaf.kan} we have
	$$
	RF(X) \simeq \lim_\Delta RF(\Cv(Y/X)_\bullet).
	$$ 
	For any affine open $U\hookrightarrow X$ let $Y_U\coloneqq \coprod_{i\in I} U\times_X \Spec(A_i)\rightarrow U$ be the induced cover on $U$. Thus the question reduces to showing the equivalence
	$$
	RF(\Cv(Y/X)_n)\rightarrow \lim_{\substack{U\hookrightarrow X\\ \textup{ affine open}}} RF(\Cv(Y_U/U)_n) ,
	$$
	for all $[n]\in\Delta$. By cofinality, we may replace $X$ in the limit argument by $\Cv(Y/X)_n$ for any $n$.\par
	To see this, note that for every affine open $U\hookrightarrow \Cv(X/Y)_n$, we get a morphism $U\rightarrow U\times_X U\simeq \Cv(Y/X)_n\times_X U \times_{\Cv(Y/X)_n} U \simeq \Cv(Y_U/U)_n\times_{\Cv(Y/X)_n} U\rightarrow \Cv(Y_U/U)$. Thus $$\lim_{\substack{U\hookrightarrow \Cv(Y/X)_n\\ \textup{affine open}}} RF(U) \simeq \lim_{\substack{U\hookrightarrow X\\ \textup{affine open}}} RF(\Cv(Y_U/U)_n).$$	
	\par 
	Assume the pairwise intersection of the $\Spec(A_i)$ is affines, then $X$ is affine and the question is trivial. Now assume the pairwise intersection is not affine then it is open in an affine scheme and thus separated. Thus they admit Zariski covers, where each of the pairwise intersection is affine. Repeating the whole process concludes the proof.
\end{proof}

\begin{remark}
	\label{perf fpqc local}
	Let us remark that the functors $A\mapsto \MMod_{A}$ and $A\mapsto \MMod_A^{\textup{perf}}$ are hypercomplete sheaves for the fpqc topology. The first assertion follows by \cite[Cor. D.6.3.3]{SAG}. The second assertion is clear since modules satisfy flat hyperdescent and since perfect modules are precisely the dualizable ones (see \cite[Prop. 7.2.2.4]{HA}), we can construct a dual fpqc locally (see \cite[Prop. 4.6.1.11]{HA}) - see the proof \cite[Lem. 5.4]{AG} for a more detailed explanation.
\end{remark}

\begin{rem}
	\label{descent qcoh}
	Using the definition of the functors $\QQCoh$ and $\QQCoh_{\perf}$, we see with Remark \ref{perf fpqc local} and Remark \ref{rem.sheaf.kan} that these functors satisfy descent in the sense that for any effective epimorphism of derived stacks $X\twoheadrightarrow Y$, we have $$\QQCoh(Y)\simeq \lim_{\Delta}\QQCoh(\Cv(X/Y)_{\bullet})\textup{ resp. }\QQCoh_{\perf}(Y)\simeq \lim_{\Delta}\QQCoh_{\perf}(\Cv(X/Y)_{\bullet}).$$
\end{rem}

\begin{remark}
	Using Remark \ref{qqcoh right kan} and Proposition \ref{right kan ex derived scheme}, we see that a quasi-coherent module over a derived scheme is given by a compatible family of modules $(\Fcal_A)_A$ for every affine open $\Spec(A)\hookrightarrow X$.
\end{remark}

\begin{remark}
	Let us recall the derived direct and inverse image. For a morphism of animated rings $\Spec(B)\rightarrow \Spec(A)$, we get a forgetful functor $\MMod_B\rightarrow \MMod_A$ (this follows from \cite[4.6.2.17]{HA}). This functor is right adjoint to the tensor product $B\otimes_A -$. We can globalize this to the case where we replace the domain by an arbitrary derived stack $X$. Namely, any quasi-coherent module $\Fcal$ over $X$ is determined by its underlying $C$-module $\iota^{*}\Fcal$, for $\iota\colon\Spec(C)\rightarrow X$. Since $C$ is naturally an $A$-algebra, we can forget the $C$-structure and view $\Fcal$ as a limit in $\MMod_A$. The tensor product with each such $C$ also induces a functor from $A$-modules to quasi-coherent $X$-modules. We can also globalize this construction on the base for a geometric morphism of derived stacks, i.e. if $f\colon X\rightarrow S$ is a geometric morphism of derived stacks, we get an adjunction
	$$
	\begin{tikzcd} 
	f^{\ast}\colon \QQCoh(S) \arrow[r,"",shift left = 0.8]&\arrow[l,"",shift left = 0.8] \QQCoh(X)\colon f_\ast,
	\end{tikzcd}
	$$
	note here that the right adjoint comes formally from the fact that the pullback by construction commutes with colimits. If one adds assumptions to $f$, then we can say more about the pushforward but we will not do this here and refer to \cite[\S 5.5]{DAG} since it is not of interest for us.\par
	If we work with classical schemes, we will write $Lf^*$ and $Rf_*$ to differentiate between the classical notions.
\end{remark}

\begin{prop}
\label{derived cat is kan ext}
	Let $X$ be a scheme. Then we have an equivalence of $\infty$-categories $\Dcal_{\textup{qc}}(X)\simeq \QQCoh(X)$, where $\Dcal_{\textup{qc}}(X)$ denoted the derived $\infty$-category of $\Ocal_{X}$-modules with quasi-coherent cohomologies.
\end{prop}
\begin{proof}
This is shown in the spectral setting in \cite{SAG} and can be followed in our setting from Lurie's PhD thesis \cite{DAG}. For convenience of the reader we will show how to conclude this proposition as a consequence of both references.\par
	A scheme $X$ is per definition a locally ringed space $(X,\Ocal_{X})$. We let $X'\coloneqq\Shv_{\Sets}(X)$ denote the Grothendieck topos associated to the small Zariski-site of $X$. The sheaf of rings $\Ocal_{X}$ can be viewed as a ring object of $X'$. So in particular, the tuple $(X',\Ocal_{X})$ defines a locally ringed topos. We write $\Xcal$ for the $1$-localic $\infty$-topos assocaited to $X'$ (which exists by \cite[Prop. 6.4.5.7]{HTT}). As explained in \cite[Rem. 1.4.1.5]{SAG}, we can view $\Ocal_{X}$ as a sheaf of connective $0$-truncated $E_{\infty}$-rings (which are just commutative rings) on $\Xcal$, which we denote by $\Ocal_{\Xcal}$ and hence get a spectrally ringed space $(\Xcal,\Ocal_{\Xcal})$, which is local (we refer to \cite[\S I.1.1]{SAG} for the definitions).\par 
	Since $\Ocal_{\Xcal}$ takes values in commutative rings, we can also view it as a sheaf on $\Xcal$, with values in animated rings. Therefore $(\Xcal,\Ocal_{\Xcal})$ defines a spectral scheme resp. a derived scheme in the sense of \cite[\S I.1]{SAG} resp. \cite{DAG}. Note that also the definition of an $\Ocal_{\Xcal}$-module agrees in both references, i.e. in both references, we see an $\Ocal_{\Xcal}$-module as an $\Ocal_{\Xcal}$-module object in $\Shv_{\Sp}(\Xcal)$, where $\Ocal_{\Xcal}$ is naturally seen as a sheaf with values in spectra.\par 
	By \cite[Thm. 4.6.5]{DAG}, we have an equivalence of $\QQCoh(X)$ and $\infty$-categories of sheaf $\Ocal_{\Xcal}$-modules $M$ on $\Xcal$ such that 
	\begin{enumerate}
		\item $\pi_{i}M$ (which is defined as the sheafification of the presheaf $V\mapsto \pi_{i}M(V)$) is a quasi-coherent sheaf on the underlying Deligne-Mumford stack of $(\Xcal,\Ocal_{\Xcal})$, and
		\item the underlying sheaf of spaces of $M$ is a hypersheaf (as explained above $M$ can be seen as a sheaf on $\Xcal$ with values in $\Sp$, i.e. a limit preserving functor $M\colon \Xcal^{\op}\rightarrow \Sp$, and composing with $\Omega^{\infty}\colon \Sp\rightarrow \SS$ defines the underlying sheaf of spaces of $M$).
	\end{enumerate}
	Using \cite[Prop. 2.2.6.1]{SAG}, we see that $\QQCoh(X)$ is equivalent to the $\infty$-category of quasi-coherent sheaves on the spectral Deligne-Mumford stack $(\Xcal,\Ocal_{\Xcal})$.
	Now \cite[Cor. 2.2.6.2]{SAG} shows that indeed the derived $\infty$-category of $\Ocal_{X}$-modules with quasi-coherent cohomology is equivalent to $\QQCoh(X)$.
\end{proof}

\subsection{Deformation theory of derived stacks}

This section is derived from \cite[\S 1.4]{TV2},\cite[\S 4.2]{AG}, \cite[Lecture 5]{Khan}.\par 
In this section, we will globalize the results of Section \ref{sec:cotangent affine} and \ref{sec:smooth and et} to geometric morphisms of derived stacks. For this, we will define a global version of the cotangent complex and list properties. Most importantly, we will show that any $n$-geometric morphism has a cotangent complex and that smooth morphisms are characterized by the cotangent complex. Further, we will use the results to show that geometric stacks are automatically hypercomplete sheaves for the \'etale topology.\par
We let $R$ be a ring and assume every derived stack is a derived stack over $R$.
\par
Let $f\colon X\rightarrow Y$ be morphism of derived stacks. Let $x\colon \Spec(A)\rightarrow X$ be an $A$-point, where $A$ is an animated $R$-algebra. Let $M\in\MMod^\cn_A$ and let us look at the commutative square
$$
\begin{tikzcd}
X(A\oplus M)\arrow[r,""]\arrow[d,""]&X(A)\arrow[d,"f"]\\
Y(A\oplus M)\arrow[r,""]&Y(A),
\end{tikzcd}
$$ 
where the vertical arrow are given by the canonical projection $A\oplus M\rightarrow A$. We set the \textit{dervations at the point $x$} as
$$
\Der_x(X/Y,M)\coloneqq \fib_x(X(A\oplus M)\rightarrow X(A)\times_{Y(A)}Y(A\oplus M)),
$$
where we see $x$ as a point in the target via the natural map induced by $
\Spec(A\oplus M)\rightarrow\Spec(A) \xrightarrow{x}X\xrightarrow{f}Y.
$
\begin{defi}[\protect{\cite[Def. 1.4.1.5]{TV2}}]
\label{defi.cotangent.global}
	Let $f\colon X\rightarrow Y$ be a morphism of derived stacks. We say $L_{f,x}\in\MMod_A$ is a \textit{cotangent complex for $f$ at the point $x\colon \Spec(A)\rightarrow X$}, if it is $(-n)$-connective, for some $n\geq 0$ and for all $M\in\MMod_A^\cn$ there is a functorial equivalence $$\Hom_{\MMod_A}(L_{f,x},M)\simeq\Der_x(X/Y,M).$$\par 
	When such $L_{f,x}$ exists, we say \textit{$f$ admits a cotangent complex at the point $x$}. If there is no possibility of confusion, we also write $L_{X/Y,x}$ for $L_{f,x}$. We also write $L_X$ if $Y\simeq \Spec(R)$.
\end{defi}

\begin{defi}
	Let $f\colon X\rightarrow Y$ be a morphism of derived stacks. We say that $L_f\in \QQCoh(X)$ \textit{is a cotangent complex for $f$} if for all points $x\colon \Spec(A)\rightarrow X$ the $A$-module $x^*L_f$ is a cotangent complex for $f$ at the point $x$.\par 
	If $L_f$ exists, we say that \textit{$f$ admits a cotangent complex}. We will write $L_{f,x}$ instead of $x^*L_f$ if $f$ admits a cotangent complex.\par 
	If $Y\simeq\Spec(R)$ and $L_{f}$ exists, we say that $X$ admits an \textit{absolute cotangent complex}.
\end{defi}

\begin{remark}
	By Lemma \ref{connective yoneda ff}, the  cotangent complex, for any morphism of derived stacks, is unique up to homotopy.
\end{remark}

\begin{rem}
	\label{affines have cotangent}
	Note that any morphism of affine derived schemes $f\colon \Spec(B)\rightarrow \Spec(A)$ admits a cotangent complex. For any point $x\colon\Spec(C)\rightarrow \Spec(B)$, we have $L_{\Spec(B)/\Spec(A),x}\coloneqq L_{B/A}\otimes_B C$.
\end{rem}

\begin{lem}
	\label{properties of cotangent}
	Let $f\colon X\rightarrow Y$ be a morphism of derived stacks.
	\begin{enumerate}
		\item If $X$ and $Y$ admit absolute cotangent complexes, then $f$ admits a cotangent complex and we have the following cofiber sequence for any point $x\colon \Spec(A)\rightarrow X$
		$$
		L_{Y,f\circ x}\rightarrow L_{X,x}\rightarrow L_{f,x}.
		$$
		\item If $f$ admits a cotangent complex, then for any morphism of derived stacks $Z\rightarrow Y$ and any point $x\colon \Spec(A)\rightarrow X\times_YZ$, we have
		$$
		L_{f,x}\simeq L_{X\times_YZ/Z,x}.
		$$ 
		\item If for any morphism $x\colon \Spec(A)\rightarrow X$ the projection $\pr\colon X\times_{Y,f\circ x} \Spec(A)\rightarrow \Spec(A)$ admits a cotangent complex, then $f$ admits a cotangent complex and further we have
		$$
		L_{f,x}\simeq L_{\pr,x}.
		$$
		\item If for any point $x\colon \Spec(A)\rightarrow X$ the stack $X\times_{Y,f\circ x} \Spec(A)$ admits a cotangent complex, then $f$ has a cotangent complex and we have
		$$
		L_{\Spec(A),\id_{\Spec(A)}}\rightarrow L_{X\times_Y A,x}\rightarrow L_{f,x}.
		$$
	\end{enumerate}
\end{lem}
\begin{proof}
	The proof in the model categorical case is given in \cite[Lem. 1.4.1.16]{TV2}. But these properties are straightforward to check.\par 
	Part \textit{1} and \textit{2} follow from the definitions. Part \textit{3} follows from \textit{2} and part \textit{4} follows from \textit{1}, \textit{3} and Remark \ref{affines have cotangent}
\end{proof}

\begin{lem}
	\label{relative global cotangent sequence}
	Let $X\xrightarrow{f} Y\xrightarrow{g} Z$ be a morphism of derived stacks. Assume $Y/Z$ admits a cotangent complex, then $X/Y$ admits a cotangent complex if and only if $X/Z$ admits a cotangent complex. Further, we obtain a cofiber sequence
	$$
	f^*L_{Y/Z}\rightarrow L_{X/Z}\rightarrow L_{X/Y}
	$$
	of quasi-coherent modules over $X$ if the cotangent complexes exist. 
\end{lem}
\begin{proof}
	This is stated in \cite[Lec. 5 Prop. 5.7]{Khan}. But anyway we will give a proof. \par Let us take a point $x\colon \Spec(A)\rightarrow X$. Then what we impose is that we have a cofiber sequence
	$$
	L_{Y/Z,f\circ x}\rightarrow L_{X/Z,x}\rightarrow L_{X/Y,x}.
	$$
	By Lemma \ref{connective yoneda ff} it is enough to show that for any connective $A$-module $M$ the fiber of $\Der_{x}(X/Z,M)\rightarrow \Der_{f\circ x}(Y/Z,M)$ at the trivial derivation (given by $A\oplus M\rightarrow A\rightarrow X\rightarrow Y$) is given by $\Der_{x}(X/Y,M)$. But this is clear.
\end{proof}

\begin{lem}
	\label{cotangent monomorphism}
	Let $j\colon X\hookrightarrow Y$ be a monomorphism of derived stacks over $A$, then $j$ admits a cotangent complex and $L_j\simeq 0$.
\end{lem}
\begin{proof}
	The proof is the same as in \cite[Lec. 5 Prop. 5.9]{Khan}, but for the convenience of the reader, we recall the proof.\par 
	It suffices to show that at any point $x\colon\Spec(B)\rightarrow X$ the cotangent complex is zero, i.e. the space of derivations at $x$ is contractible. Since $j$ is a monomorphism, i.e. its fibers are $(-1)$-truncated (so either contractible or empty (see \cite[Def. 5.5.6.8]{HTT})), we know that the canonical map
	$$
	X(B\oplus M)\rightarrow X(B)\times_{Y(B)} Y(B\oplus M)
	$$
	is also a monomorphism (note that any point in the fiber of the above morphism defines a point in the fiber of $X(B\oplus M)\rightarrow Y(B\oplus M)$, which per definition has either contractible or empty fibers). But the canonical map $u\colon \Spec(B\oplus M)\rightarrow \Spec(B)\rightarrow X$ defines a derivation, so the space of derivations is nonempty and thus contractible.
\end{proof}

Now we can easily see that the homotopy groups of the localization with respect to one element is given by the localization of the homotopy groups.

\begin{lem}
\label{homotopy of localization}
	Let $A$ be an animated ring and $f\in\pi_{0}A=\pi_{0}\Hom_{\MMod_{A}}(A,A)$. Then we have $\pi_{i}(A[f^{-1}])\cong (\pi_{i}A)_{f}$ as $\pi_{0}A$-modules.
\end{lem}
\begin{proof}
From Proposition \ref{localization} and Lemma \ref{localization lfp} it follows that the map $\Spec(A[f^{-1}])\hookrightarrow \Spec(A)$ is locally of finite presentation and a monomorphism. Since monomorphisms have a vanishing relative cotangent complex (see Lemma \ref{cotangent monomorphism}), we conclude with Proposition \ref{smooth cotangent} that $A\rightarrow A[f^{-1}]$ is \'etale. Hence, we conclude using the definition of \'etale morphisms.
\end{proof}

\begin{lem}
	\label{etale cover of triv sq zero ext}
	Let $(A\rightarrow A_i)_{i\in I}$ be an \'etale covering in $\AniAlg{R}$ and let $M$ be an $A$-module. Then the family induced by base change $(A\oplus M\rightarrow A_i\oplus (M\otimes_A A_i))_{i\in I}$ is an \'etale cover.
\end{lem}
\begin{proof}
	We only need to show that $A_i\otimes_A (A\oplus M)\simeq A_i\oplus (M\otimes_A A_i)$, since \'etale covers are stable under base change. But by construction the functors $A_i\otimes_A (-\oplus M)$ and $(-\otimes_A A_i) \oplus (M\otimes_A A_i)$ from $\AniM_R$ to $\AniAlg{A_i}$ commute with sifted colimits and thus we are reduced to classical commutative algebra, where it is clear.
\end{proof}

\begin{lem}
	\label{cotangent exists for schemes}
	Let $f\colon X\rightarrow Y$ be a morphism of derived schemes. Then $f$ admits a cotangent complex.
\end{lem}
\begin{proof}
	A proof sketch is given in \cite[Lec. 5 Thm. 5.12]{Khan}, but since these are lecture notes, we recall the proof.\par
	We may assume that $Y=\Spec(R)$ using Lemma \ref{relative global cotangent sequence}.\par 
	By Proposition \ref{right kan ex derived scheme}, we know that $$\QQCoh(X)\simeq \lim_{\substack{\Spec(B)\hookrightarrow X\\\textup{ open immersion }}} \MMod_B.$$ For each of the open affines $\Spec(B)$ in $X$, we take $L_{\Spec(B)}\coloneqq L_{B/\ZZ}$, viewed as a quasi-coherent sheaf on $\Spec(B)$. Since the cotangent complex is compatible with taking pullbacks, i.e. $L_{B}\otimes_B A\simeq L_A$ for a triangle 
	$$
	\begin{tikzcd}
	\Spec(A)\arrow[rd,""]\arrow[d,""] \\
	\Spec(B)\arrow[r,""]&X,
	\end{tikzcd}
	$$
	where $\Spec(A),\Spec(B)$ are open in $X$, we see that this indeed defines an object in the limit, which we denote with $L_X$ (for the compatibility, note that the relative cotangent complexes of monomorphisms vanish by Lemma \ref{cotangent monomorphism}). This defines a cotangent complex on $X$.\par 
	We have to show that $L_X$ represents the space of derivations. Since we have glued the cotangent complex for affine opens, we will use a descent argument for arbitrary points. For this, we use that modules satisfy fpqc descent and that for any point $x\colon \Spec(B)\rightarrow  X$, the base change with an affine open cover of $X$ gives an affine open cover $(B_i)$ of $B$. Hence, for a connective $B$-module $M$, it suffices to show that
	$$
	\lim\Hom_{B_i}(L_{X,B_i},M_i)\simeq \lim \Der_{B_i}(X/\ZZ,M_i),
	$$ 
	where $M_i\coloneqq M\otimes_B B_i$, which is clear termwise, since each $B_i$ factors through some affine open $A_i$ of $X$ by construction (to write the derivations as a limit use the sheaf property of $X$, $\Spec(\ZZ)$ and Lemma \ref{etale cover of triv sq zero ext}). Note that for each affine open the cotangent complex exists and we claim that $L_{X,B_i}\simeq L_{X,A_i}\otimes_{A_i} B_i\simeq L_{\Spec(A_i),A_i}\otimes_{A_i}B_i\simeq L_{\Spec(A_i),B_i}$, which concludes the lemma.\par 
	To see this, note that $\Der_{B_i}(A_i/X,M_i)\simeq 0$, since  $\Spec(A_i)\hookrightarrow X$ is a monomorphism (see Lemma \ref{cotangent monomorphism}). Therefore $\Der_{B_i}(A_i/\ZZ,M_i)\simeq \Der_{B_i}(X/\ZZ,M_i)$ (since its fiber is $\Der_{B_i}(A_i/X,M_i)$), the same holds if we replace the point by $A_i$. Since the cotangent complexes at $A_i$ exist, we get $L_{\Spec(A_i),A_i}\simeq L_{X,A_i}$ and after tensoring with $B_i$, we see that $L_{X,B_i}$ is a cotangent complex for $X/\ZZ$ at $B_i$ if and only if $\Der_{B_i}(A_i/\ZZ,M_i)\simeq \Der_{B_i}(X/\ZZ,M_i)$ but this we have seen above, i.e. we have equivalences
	\begin{align*}
		\Hom_{B_i}(L_{X,B_i},M_i)\simeq \Hom_{B_i}(L_{\Spec(A_i),B_i},M_i)\simeq \Der_{B_i}(A_i/\ZZ,M_i)\simeq\Der_{B_i}(X/\ZZ,M_i).
	\end{align*}\par
	We remark that by this construction and commutativity of $\tau_{\geq 0}$ with limits, we have that $L_X$ is connective. In particular, we have shown that $L_X$ is a cotangent complex for $X/\ZZ$.
\end{proof}

\begin{remark}
	Let us give another construction of a cotangent complex. Consider the functor $L_{-}\colon \AniAlg{\ZZ}\rightarrow \Dcal(R)$ given by the usual cotangent complex seen as complex of abelian groups. We denote its right Kan extension along the inclusion $\AniAlg{R}\hookrightarrow  \dSch^{\op}_{/\Spec(R)}$ with $\RR L_{-}$. By the above proof, we see that for a derived scheme $X$, we have $L_{X}\simeq \RR L_{X}$ in $\Dcal(R)$. In particular, by stability of the derived $\infty$-category, the same holds for the relative cotangent complex.
\end{remark}

\begin{defi}
	We recall the notion of an obstruction theory for derived stacks respectively morphism of derived stacks (see \cite[1.4.2.1, 1.4.2.2]{TV2}).
	\begin{enumerate}
		\item[(i)] A derived stack $X$ is called \textit{infinitesimally cartesian} or \textit{inf-cartesian} if and only if for every animated $R$-algebra $A$, connective $A$-module $M$ with $\pi_0M=0$ and derivation $d\in \Der(A,M)$ the pullback square
		$$
		\begin{tikzcd}
			A\oplus_d M\arrow[r,""]\arrow[d,""]&A\arrow[d,"d"]\\
			A\arrow[r,"s"]&A\oplus M,
		\end{tikzcd}
		$$
		where $s$ denotes the trivial derivation, induces a pullback square 
		$$
		\begin{tikzcd}
			X(A\oplus_d M)\arrow[r,""]\arrow[d,""]&X(A)\arrow[d,"d"]\\
			X(A)\arrow[r,"s"]&X(A\oplus M).
		\end{tikzcd}
		$$
	\par 
	A morphism $f\colon X\rightarrow Y$ of derived stacks is called \textit{infinitesimally cartesian} or \textit{inf-cartesian} if and only if for every animated $R$-algebra $A$, connective $A$-module $M$ with $\pi_0M=0$ and derivation $d\in \Der(A,M)$ we have a pullback square 
	$$
	\begin{tikzcd}
		X(A\oplus_d M)\arrow[r,""]\arrow[d,""]&Y(A\oplus_d M)\arrow[d,""]\\
		X(A)\times_{X(A\oplus M)}X(A)\arrow[r,""]&Y(A)\times_{Y(A\oplus M)}Y(A).
	\end{tikzcd}
	$$

	\item[(ii)] A derived stack $X$ has an \textit{obstruction theory} if and only if it has a cotangent complex and is  infinitesimally cartesian.\par A morphism of derived stacks $f\colon X\rightarrow Y$ has an \textit{obstruction theory} if and only if it has a cotangent complex and is  infinitesimally cartesian.
	\end{enumerate}
\end{defi}

\begin{defi}[\protect{\cite[Def.  1.2.8.1]{TV2}}]
	Let $f\colon X\rightarrow Y$ be a morphism of derived stacks, we say $f$ is \textit{formally smooth} if for any $A\in \AniAlg{R}$, a connective $A$-module $M$ with $\pi_0M=0$, and derivation $d\in \Der_{R}(A,M)$ the natural map
	$$
	\pi_{0}X(A\oplus_d M)\rightarrow \pi_{0}(X(A)\times_{Y(A)}Y(A\oplus_d M))
	$$
	is surjective.
\end{defi}

\begin{lem}
	\label{properties of obstruction}
	Let $f\colon X\rightarrow Y$ be a morphism of derived stacks.
	\begin{enumerate}
		\item If $X$ and $Y$ have an obstruction theory, then $f$ has an obstruction theory.
		\item If $f$ has an obstruction theory, then for any morphism of derived stacks $Z\rightarrow Y$ the base change $Z\times_Y X\rightarrow Z$ has an obstruction theory.
		\item If for any $A\in \AniAlg{R}$ and any morphism $\Spec(A)\rightarrow Y$ the base change $X\times_Y \Spec(A)\rightarrow \Spec(A)$ has an obstruction theory, then $f$ has an obstruction theory.
	\end{enumerate}
\end{lem}
\begin{proof}
	This is \cite[Lem. 1.4.2.3]{TV2} but nevertheless we recall the proof.\par 
	The existence of the cotangent complex follows from Lemma \ref{properties of cotangent}.\par 
	Part \textit{1} and \textit{2} are clear by definition. For part \textit{3} let $B$ be an animated ring, $M$ a connective $B$-module with $\pi_0M=0$ and $d\in \Der_{R}(B,M)$. We need to show that the diagram
	$$
	\begin{tikzcd}
	X(A\oplus_d M)\arrow[r,""]\arrow[d,""]&Y(A\oplus_d M)\arrow[d,""]\\
	X(A)\times_{X(A\oplus M)}X(A)\arrow[r,""]&Y(A)\times_{Y(A\oplus M)}Y(A)
	\end{tikzcd}
	$$
	is a pullback diagram. Let $x\in Y(A\oplus_{d} M)$, we claim that it suffices to show that the induced morphism of the fibers of the two horizontal arrows at $x$ is an equivalence.\par 
	Indeed, assume that we have a commutative diagram
	\begin{equation}
	\label{eq.fiber.seq.lem}
	\begin{tikzcd}
	F\arrow[r,""]\arrow[d,""]& \ast\arrow[d,"x"]\\
	X(A\oplus_d M)\arrow[r,""]\arrow[d,""]&Y(A\oplus_d M)\arrow[d,""]\\
	X(A)\times_{X(A\oplus M)}X(A)\arrow[r,""]&Y(A)\times_{Y(A\oplus M)}Y(A)
	\end{tikzcd}
	\end{equation}
	where the upper square and the outer square is a pullback. Let us also consider the pullback diagram
	$$
	\begin{tikzcd}
	Z\arrow[r,"\alpha"]\arrow[d,""]&Y(A\oplus_d M)\arrow[d,""]\\
	X(A)\times_{X(A\oplus M)}X(A)\arrow[r,""]&Y(A)\times_{Y(A\oplus M)}Y(A)
	\end{tikzcd}
	$$
	(note that naturally $\fib_{x}(\alpha)\simeq F$ as the outer square of (\ref{eq.fiber.seq.lem}) is a pullback diagram).
	We have a naturally induced morphism of fiber sequences (i.e. a commutative diagram of the form)
	$$
	\begin{tikzcd}
		F\arrow[r,""]\arrow[d,"\simeq"]& X(A\oplus_d M)\arrow[d,""]\arrow[r,""]&Y(A\oplus_d M)\arrow[d,"\simeq"]\\
		\fib_{x}(\alpha)\arrow[r,""]&Z\arrow[r,""]&Y(A\oplus_d M).
	\end{tikzcd}
	$$
	The long exact homotopy sequence for fiber sequences in $\SS$ now implies the claim.\par
	That the upper square and the outer square of (\ref{eq.fiber.seq.lem}) are pullback diagrams follows from the fact that the pullback of $f$ under the morphism corresponding to $x$ has an obstruction theory.
\end{proof}

The following technical lemma shows, how liftings along square zero extensions are linked to loops in the space of derivations. This is crucial, when dealing with formal smoothness of morphisms. Lifts of morphisms along square zero extensions are controlled by the cotangent complex which is, in some cases, easier to handle.

\begin{lem}
	\label{obstruction morph imp formal smooth}
	Let $f\colon X\rightarrow Y$ be a morphism of derived stacks, and assume $f$ has an obstruction theory. Let $A\in\AniAlg{R}$, $M$ be a connective $A$-module with $\pi_0M=0$ and $d\in \Der(A,M)$ a derivation. Let $x\in X(A)\times_{Y(A)}Y(A\oplus_d M)$ be a point and $L(x)$ the fiber of $X(A\oplus_d M)\rightarrow X(A)\times_{Y(A)}Y(A\oplus_d M)$ at $x$. There exists an element $\alpha(x)\in \pi_0\Hom_{\MMod_A}(L_{f,x},M)$ such that $L(x)\simeq\Omega_{\alpha(x),0}\Hom_{\MMod_A}(L_{f,x},M)$, where we consider the pullback diagram
	$$
	\begin{tikzcd}
	\Omega_{\alpha(x),0}\Hom_{\MMod_A}(L_{f,x},M)\arrow[r,""]\arrow[d,""]&\ast\arrow[d,"\alpha(x)"]\\
	\ast\arrow[r,"0"]&\Hom_{\MMod_A}(L_{f,x},M).
	\end{tikzcd}
	$$
\end{lem}
\begin{proof}
	This is \cite[Prop. 1.4.2.6]{TV2}, but for the convenience of the reader we give a proof.\par 
	First, note that $x$ corresponds to a diagram of the form
	$$
	\begin{tikzcd}
	\Spec(A)\arrow[r,""]\arrow[d,""]&X\arrow[d,""]\\
	\Spec(A\oplus_d M)\arrow[r,""]&Y.
	\end{tikzcd}
	$$
	After composition with the natural maps, we get
	$$
	\begin{tikzcd}
	\Spec(A\oplus M)\arrow[r,"d"]\arrow[d,"s"]&\Spec(A)\arrow[r,""]\arrow[d,""]&X\arrow[d,""]\\
	\Spec(A)\arrow[r,""]&\Spec(A\oplus_d M)\arrow[r,""]&Y,
	\end{tikzcd}
	$$
	which gives a point $\alpha(x)\in \Hom_{\dSt_{A//Y}}(\Spec(A\oplus M),X)\simeq\Der_A(X/Y,M)$.
	Using that $f$ is inf-cartesian, we get a pullback diagram
	$$
	\begin{tikzcd}
	L(x)\arrow[r,""]\arrow[d,""]&\ast\arrow[d,"\alpha(x)"]\\
	\ast\arrow[r,"0"]&\Der_{A}(X/Y,M).
	\end{tikzcd}
	$$
	\par
	To see this, note the following commutative diagram with pullback squares.
	$$
	\begin{tikzcd}[column sep=1ex]
	L(x)\arrow[r,""]\arrow[d,""]&\ast\arrow[d,"x"]\\
	X(A\oplus_d M)\arrow[r,""]\arrow[d,""]&X(A)\times_{Y(A)}Y(A\oplus_d M)\arrow[d,"\alpha"]\\
	\ast\arrow[r,"0"]\arrow[d,""]&\Der_{A}(X/Y,M)\arrow[r,""]\arrow[d]&\ast\arrow[d,""]\\
	X(A)\times_{Y(A)}Y(A\oplus M)\arrow[r,""]\arrow[d,"\simeq"]&X(A\oplus M)\arrow[r,""]&X(A)\times_{Y(A)}Y(A\oplus M)\\
	X(A)\times_{Y(A)\times_{Y(A\oplus M)}Y(A)} Y(A).
	\end{tikzcd}
	$$ 
\end{proof}

\begin{lem}
	\label{affine obstruction}
	Any affine derived scheme $X\simeq \Spec(B)$ has an obstruction theory.
\end{lem}
\begin{proof}
	Certainly, $X$ has a cotangent complex by $L_{X,x}\simeq L_B\otimes_B A$, for $x\colon\Spec(A)\rightarrow X.$ So we are left to show that $X$ is infinitesimally cartesian. But this follows from compatability of the $\Hom$ functor with limits.
\end{proof}

\begin{lem}
	\label{affine formal smooth}
	Let $f\colon X\rightarrow Y$ be a morphism of affine derived schemes. If $f$ is smooth, then it is formally smooth.
\end{lem}
\begin{proof}
	This follows from \cite[Prop. 2.2.5.1]{TV2}, but for convenience of the reader, we give a proof in our setting.\par 
	Let $A$ be an animated $R$-algebra, $M$ a connected $C$-module and $d\in \Der(A,M)$. Let $x\in \pi_0 X(A)\times_{Y(A)}Y(A\oplus_d M)$ be a point. We have to show that the fiber of 
	$$
	X(A\oplus_d M)\rightarrow X(A)\times_{Y(A)}Y(A\oplus_d M)
	$$ 
	along $x$ is nonempty. By Lemma \ref{affine obstruction} and \ref{obstruction morph imp formal smooth}, it suffices to show that $\pi_0\Hom(L_{B/C}\otimes_B A, M)$ is contractible, where $X\simeq \Spec(B)\rightarrow \Spec(C)\simeq Y$. By Proposition \ref{smooth cotangent} the $A$-module $L_{B/C}\otimes_B A$ is finite projective, so especially a retract of a free module (see \cite[Cor. 7.2.2.9]{HA}) and therefore $\pi_0\Hom(L_{B/C}\otimes_B A, M)$ is a retract of a product of $\pi_0M$, which is zero by hypothesis on $M$.
\end{proof}

\begin{rem}
	We want to remark that Lemma \ref{affine formal smooth} holds more generally. An $n$-geometric morphism of derived stacks is smooth if any only if after restriction to $\Ring$ via $t_{0}$ is locally of finite presentation and it is formally smooth. This is a bit technical but a proof of this is given for example in \cite[Prop. 2.2.5.1]{TV2}.
\end{rem}

The next proposition and corollary show, how smoothness of a geometric morphism is linked to its cotangent complex. This can be seen as a globalization of Proposition \ref{smooth cotangent}.

\begin{prop}
	\label{cotangent of smooth}
	Let $f\colon X\rightarrow Y$ be an $n$-geometric morphism of derived stacks. Then $f$ has an obstruction theory. Further if $f$ is smooth, then $f$ is formally smooth and $L_f$ is perfect with Tor-amplitude in $[-n-1,0]$.
\end{prop}
\begin{proof}
	The proof of this lemma in the spectral setting is given in \cite[Prop. 4.45]{AG}. The proof in the derived setting is analogous. But for the convenience of the reader, we give a proof.\par 
	We prove this lemma by induction over $n$. For the formal smoothness part we will first reduce to the case where $Y$ is affine.\par 
	Indeed, let $B$ be an animated $R$-algebra. Note that we have to show that for any point $x\in \pi_0(X(B)\times_{Y(B)}Y(B\oplus M))$ its fiber under $X(B\oplus_d M)\rightarrow X(B)\times_{Y(B)}Y(B\oplus_d M)$ is nonempty. The point $x$ corresponds to a commutative diagram of the form 
	$$
	\begin{tikzcd}
	\Spec(B)\arrow[r,""]\arrow[d,""]&X\arrow[d,""]\\
	\Spec(B\oplus_d M)\arrow[r,""]&Y.
	\end{tikzcd}
	$$
	After base change, we get a diagram of the form
	$$
	\begin{tikzcd}
	\Spec(B)\arrow[r,""]\arrow[d,""]&X\times_Y \Spec(B\oplus_d M)\arrow[d,""]\\
	\Spec(B\oplus_d M)\arrow[r,"\id"]&\Spec(B\oplus_d M)
	\end{tikzcd}
	$$
	showing that, we can replace $f$ by the projection $X\times_Y \Spec(B\oplus_d M)\rightarrow \Spec(B\oplus_d M)$, in particular, we can assume $Y$ to be affine (this reduction is part of \cite[Prop. 2.2.5.1]{TV2}).\par 	
	Further for the  existence of an obstruction theory, we may assume without loss of generality that $Y\simeq \Spec(R)$ (see Lemma \ref{properties of obstruction} and use that affine schemes have an obstruction theory by Lemma \ref{affine formal smooth}).\par
	Let $n=-1$, then $X\simeq\Spec(A)$ and each $A$ is a smooth $R$-algebra. In particular, we see with Lemma \ref{affine obstruction} and  Proposition \ref{smooth cotangent} that $L_{A/R}$ exists and is finite projective. The formal smoothness follows from Lemma \ref{affine formal smooth}.\par 
	Now assume $n\geq 0$ and let $p\colon U\simeq \coprod_{i\in I}\Spec(A_i)\rightarrow X$ be an $n$-atlas, where $A_i$ are smooth $R$-algebras. Let $B$ be a animated $R$-algebra, $M$ be a connective $B$-module with $\pi_0M=0$ and $d\in \Der(B,M)$ a derivation.\par 
	\textit{Inf-cartesian.} For this we will follow \cite[Lem. 1.4.3.10]{TV2}.\par 
	By Lemma \ref{etale cover of triv sq zero ext} any \'etale cover $B\rightarrow B'$ gives a cartesian square of the form 
	$$
	\begin{tikzcd}
	B'\oplus_d M\arrow[r,""]\arrow[d,""]&B'\arrow[d,"d"]\\
	B'\arrow[r,"s"]&B'\oplus M,
	\end{tikzcd}
	$$
	which covers the square induced by the derivation. So to check if $X(B\oplus_d M)\simeq X(B)\times_{X(B\oplus M)}X(B)$, we can pass to an \'etale cover of $B$. Therefore, we may assume that any image $x_1\in X(B)$ of $x\in X(B)\times_{X(B\oplus M)}X(B)$ under the projection, lifts to a point in $u\in \Spec(A_i)(B)$, for some $i$. Next, we claim that the point $x$ lifts to a point $y\in  \Spec(A_i)(B)\times_{ \Spec(A_i)(B\oplus M)} \Spec(A_i)(B)$.\par
	To see this, consider the following commutative diagram
	$$
	\begin{tikzcd}
	\Spec(A_i)(B)\times_{ \Spec(A_i)(B\oplus M)} \Spec(A_i)(B)\arrow[r,"f"]\arrow[d,"p"]&X(B)\times_{X(B\oplus M)}X(B)\arrow[d,"q"]\\
	\Spec(A_i)(B)\arrow[r,""]&X(B).
	\end{tikzcd}
	$$
	Let $F(p)$ (resp. $F(q)$) denote the fiber of $u$ (resp. $x_1$) under $p$ (resp. $q$). We get a natural morphism $g\colon F(p)\rightarrow F(q)$. Moreover the fiber of $f$ along $x$ receives a natural morphism from $\fib_x(g)$. Therefore, to see that $\fib_x(f)$ is nonempty it is enough to show that $\fib_x(g)$ is nonempty. But now $g$ is naturally identified, per definition, with the morphism $\Omega_{d',0}\Der_B(A_i,M)\rightarrow \Omega_{d',0}\Der_B(X,M)$, where $d'$ is the derivation that is given by the image of $u$ (note that $X(B)\rightarrow X(B\oplus M)\rightarrow X(B)$ is equivalent to the identity). Thus the fiber of $g$ is given by $\Omega_{d',0}\Der_B(A_i/X,M)$, which is equivalent to $\Omega_{d',0}\Hom(L_{A_i/X,B},M)$ by induction hypothesis. But now, again by induction hypothesis, we can find an \'etale cover of $B$ such that $\pi_0\Hom(L_{A_i/X,B},M)=0$, since $M$ is assumed to be connected, and therefore $\Omega_{d',0}\Hom(L_{A_i/X,B},M)$ is nonempty.\par 
	Now consider the commutative digram
	$$
	\begin{tikzcd}
	\Spec(A_i)(B\oplus_d M)\arrow[r,"a"]\arrow[d,""]& \Spec(A_i)(B)\times_{ \Spec(A_i)(B\oplus M)} \Spec(A_i)(B)\arrow[d,""]\\
	X(B\oplus_d M)\arrow[r,""]&X(B)\times_{X(B\oplus M)}X(B).
	\end{tikzcd}
	$$
	By induction hypothesis this square is a pullback sqaure, further $a$ is an equivalence by affineness. Since $x$ lifts to a point in $y\in \Spec(A_i)(B)\times_{ \Spec(A_i)(B\oplus M)}\Spec(A_i)(B)$, we see that the fiber at $x$ is given by the fiber of $a$ at $y$, which is nonempty and contractible (since affine schemes are inf-cartesian by Lemma \ref{affine obstruction}).
	\par
	\textit{Existence.}	Let us look at the fiber $L_f$ of $L_{f\circ p}\rightarrow L_p$. By induction hypothesis $L_p$ and $L_{f\circ p}$ exist and if $f$ is smooth both are perfect with Tor-amplitude in $[-n,0]$ and $[0,0]$ respectively. In particular if $f$ is smooth then $L_f$ is perfect and has Tor-amplitude in $[-n-1,0]$. We have to show that $L_f$ satisfies the universal property of the cotangent complex for $f$ at any point. Let $x\colon \Spec(B)\rightarrow X$ be a morphism. We may assume, that $x$ factors through $p$, since $p$ is an effective epimorphism, so we can pass to an \'etale cover $B$ which factors through $U$. Let $y\colon \Spec(B)\rightarrow U$ be such a factorisation. We get a map
	$$
	F\colon \Der_{y}(U,N)\rightarrow \Der_{x}(X,N)
	$$
	for any connective $B$-module $N$, which is surjective.\par 
	To see the surjectivity of $F$, note that by induction hypothesis $p$ is formally smooth. Further any element in $d\in\Der_{x}(X,N)$ corresponds to a diagram of the form 
	$$
	\begin{tikzcd}
	\Spec(B\oplus N) \arrow[r,"d"] & X\\
	\Spec(B)\arrow[ur,"x", swap]\arrow[u].
	\end{tikzcd}
	$$
	Using the factorization by $U$, we get an element in $U(B)\times_{X(B)} X(B\oplus N)$. By formal smoothness of $p$, we can lift this element to an element in $U(B\oplus N)$ (note that $B\times_{0,B\oplus N[1],0}B\simeq B\oplus N$). But by construction this element has to be a derivation of $U$ at $y$.	
	\par 
	The fiber of $F$ along the derivation, which is induced by $\Spec(B\oplus N)\rightarrow \Spec(B)\xrightarrow{x}X$ is given by $\Der_{y}(U/X,N)$. Thus, we get a fiber sequence
	$$
	\Hom_B(L_p,N)\rightarrow \Hom_B(L_{f\circ p},N)\rightarrow \Der_{x}(X,N).
	$$ After delooping\footnote{For a pointed $\infty$-category $\Ccal$ a deloop of an object $c\in\Ccal$ is an object $c'\in\Ccal$, such that $c\simeq \Omega c'$. For $\Ccal = \SS$ the $\infty$-category of spaces, there is a deloop for every object. This follows from the effectivity of groupoid objects in $\SS$ (see \cite[Cor.  6.1.3.20]{HTT}) (the map $x\rightarrow \ast$, where $x\in \SS$, defines a simplicial object, which extends via the colimit to a \v{C}ech nerve).} and surjectivity of $F$, we see that $\Der_{x}(X,N)$ is the fiber of $$B\Hom_B(L_p,N)\rightarrow B\Hom_B(L_{f\circ p},N),$$ where the prefix ``$B$'' denotes the deloop, and therefore $\Der_x(X,N)\simeq\Hom_B(L_f,N)$.\par 
	To see this, note that the map from $\Der_x(X,N)$ to the fiber of $B\Hom_B(L_p,N)\rightarrow B\Hom_B(L_{f\circ p},N)$ is by the five-lemma an equivalence on the homotopy groups.\par 
	\textit{Formal smoothness.} This is part of \cite[Prop. 2.2.5.1]{TV2}.\par 
	Assume $f$ is smooth and $x\colon \Spec(B)\rightarrow X$ is a point. By the above $f$ has an obstruction theory. Therefore by Lemma \ref{obstruction morph imp formal smooth} it suffices to show that $\pi_0\Hom(L_{f,x},M)$ is contractible. But this follows from 
	$$
	\pi_0\Hom(L_{f,x},M)\simeq \pi_0 (L_{f,x}^\vee \otimes_B M)\simeq \pi_0L_{f,x}^\vee \otimes_{\pi_0B} \pi_0M\simeq 0,
	$$
	by connectedness of $M$ (note that the above construction of the relative cotangent complex implies that $L_{f}$ is perfect in the smooth case and therefore dualizable and the dual is connective).
\end{proof}

\begin{cor}
	\label{cotangent implies smooth}
	Let $f\colon X\rightarrow Y$ be an $n$-geometric morphism of derived stacks. Then $f$ is smooth if and only if $t_{0}f$ is locally of finite presentation and $L_f$ is perfect and has Tor-amplitude in $[-n-1,0]$.
\end{cor}
\begin{proof}
	The proof is the same as in the spectral setting presented in \cite[Prop. 4.46]{AG} using Proposition \ref{cotangent of smooth}. But for the convenience of the reader, we recall the proof.\par 
	We may assume that $Y\simeq\Spec(A)$ is affine (use Lemma \ref{properties of cotangent}). Let us fix an $n$-atlas $p\colon U\coloneqq \coprod_{i\in I} \Spec(T_i)\rightarrow X$. The "only if" part is Proposition \ref{cotangent of smooth} and the fact that the $T_i$ are smooth $A$-algebras by construction and thus are locally of finite presentation (see Proposition \ref{smooth cotangent}).\par
	For the "if" part assume the $\pi_{0}T_i$ are locally of finite presentation over $\pi_{0}B$, $L_f$ exists, is perfect and has Tor-amplitude in $[-n-1,0]$. We have a cofiber sequence
	$$
	p^*L_f\rightarrow L_{U/A}\rightarrow L_p,
	$$ 
	where by construction $L_f$ and $L_p$ are perfect with Tor-amplitude in $[-n-1,0]$ and $[-n,0]$ respectively (for the existence and Tor-amplitude of $L_p$, we use Proposition \ref{cotangent of smooth}). Therefore, $L_{U/A}$ is also perfect with Tor-amplitude in $[-n-1,0]$. But since $U$ is the disjoint union of affines the Tor-amplitude of $L_{U/A}$ is concentrated in $[0,0]$ and thus $L_{U/A}$ is perfect and finite projective, which implies that $\Spec(T_i)\rightarrow \Spec(A)$ is smooth (see Lemma \ref{smooth cotangent}).
\end{proof}

\begin{cor}
	\label{global lfp cotangent}
	Let $f\colon X\rightarrow Y$ be an $n$-geometric morphism of derived stacks locally of finite presentation. Then $L_{f}$ is perfect. 
\end{cor}
\begin{proof}
	Per definition of perfect quasi-coherent modules over a derived stack, we have to check that for any point $x\colon \Spec(A)\rightarrow X$ the cotangent complex $L_{f,x}$ is a perfect $A$-module. By Lemma \ref{properties of cotangent}, we know that the cotangent complex of the projection $\pr\colon X\times_{Y,f\circ x}\Spec(A)\Spec(A)$ at the point induced by $x$ is equivalent to $L_{f,x}$. So without loss of generality, we may assume that $Y\simeq \Spec(B)$ is affine.\par 
	Since $f$ is $n$-geometric and locally of finite presentation, we know that there exists an $n$-atlas $(p_{i}\colon\Spec(A_{i})\rightarrow X)_{i\in I}$ such that $A_{i}$ are locally of finite presentation over $A$. Since perfect quasi-coherent modules satisfy fpqc descent (see Remark \ref{descent qcoh}), we have that $L_{f}\in\QQCoh_{\perf}(X)$ if and only if each $p_{i}^{*}L_{f}$ is perfect. But by Lemma \ref{relative global cotangent sequence}, we have the following cofiber sequence 	
	$$
	p_{i}^{*}L_{f}\rightarrow L_{\Spec(A_{i})/\Spec(A)}\rightarrow L_{p_{i}}.
	$$
Since $A_{i}$ is locally of finite presentation over $A$, we know by Proposition \ref{lfp perfect cotangent}, that $L_{\Spec(A_{i})/\Spec(A)}$ is perfect and since by definition $p_{i}$ is smooth, we have with Proposition \ref{cotangent of smooth} that indeed $p_{i}^{*}L_{f}$ is perfect.
\end{proof}

The last part of this section is dedicated to show that a geometric derived stack $X$ is automatically hypercomplete for the \'etale topology. The idea is to show that for any animated $R$-algebra $A$, we have $X(A)\simeq \lim_{n}X(A_{\leq n})$. Then we reduce to the case, where we look at $n$-truncated sheaves, which are always hypercomplete.

\begin{lem}
\label{nilcomplete}
	Let $X$ be an $n$-geometric derived stack for some $n\geq -1$ and $A$ an animated ring. Then the natural morphism $X\rightarrow \lim_{n}X\circ\tau_{\leq n}$ is an equivalence.
\end{lem}
\begin{proof}
This is analogous to \cite[Prop. 5.3.7]{DAG}.\par  
	We will do this by induction over $n$. This is certainly true for if $X$ is affine. So assume that $n\geq 0$ and let $p\colon U\coloneqq \coprod_{i\in I}\Spec(A_{i})\rightarrow X$ be an $n$-atlas.\par 
	By definition $U\times_{X} U$ is $(n-1)$-geometric, this also holds for every successive fiber product, i.e. every element of the \v{C}ech nerve $\Cv(U/X)_{\bullet}$ is $(n-1)$-geometric. Since $p$ is an effective epimorphism, we have that the natural map $\colim_{\Delta}\Cv(U/X)_{\bullet}\rightarrow X$ is an equivalence. By induction hypothesis, we have for every $[n]\in \Delta$ that the natural map $\Cv(U/X)_{[n]}\rightarrow \lim\Cv(U/X)_{[n]}\circ\tau_{\leq n}$ is an equivalence. Thus, also its colimit under $\Delta$ is an equivalence, so we get an induced commutative diagram
	$$
	\begin{tikzcd}
		\colim_{\Delta}\Cv(U/X)_{[n]}\arrow[r,""]\arrow[d,""]&\colim_{\Delta}\lim_{n}\Cv(U/X)_{[n]}\circ\tau_{\leq n}\arrow[d,""]\\
		X\arrow[r,""]& \lim_{n}X\circ\tau_{\leq n},
	\end{tikzcd}
	$$
where the top arrow and left vertical arrow are equivalences and the right vertical arrow is a monomorphism. Thus the bottom vertical arrow is an equivalence if $\lim_{n} U\circ\tau_{\leq n}\rightarrow \lim_{n}X\circ \tau_{\leq n}$ is an effective epimorphism. Let $x\in \lim_{n}X(A_{\leq n})$ and consider the projection onto $X(A_{\leq 0})$, denoted by $x_{0}$. Then we can find an \'etale cover $\widetilde{\pi_{0}A}$ of $A_{\leq 0}\simeq \pi_{0}A$ such that $x_{0}$ has a lift in $U(\widetilde{\pi_{0}A})$. By Proposition \ref{lift etale} there is an \'etale cover $\widetilde{A}$ of $A$ such that $\pi_{0}\widetilde{A}\simeq \widetilde{\pi_{0}A}$. In particular, we see that we can lift the image of $x_{0}$ in $X(\widetilde{A}_{\leq 0})$ under $U(\widetilde{A}_{\leq 0})\rightarrow X(\widetilde{A}_{\leq 0})$. Now let $x_{n}$ be the image of $x$ in $X(A_{\leq n})$. We will show the result by induction. Assume the argument holds for $n-1$. In particular, let $u_{n-1}$ be the lift of $x_{n-1}$ under $U(\widetilde{A}_{\leq n-1})\rightarrow X(\widetilde{A}_{\leq n-1})$. It is enough to prove that we can find a lift of $(u_{n-1},x_{n})$ under $U(\widetilde{A}_{\leq n})\rightarrow U(\widetilde{A}_{\leq n-1})\times_{X(\widetilde{A}_{\leq n-1})}X(\widetilde{A}_{\leq n})$, since then for all $n\in \NN_{0}$ there is a lift $u_{n}$ of $x_{n}$ compatible with the maps in the limit, i.e. we get an element $u\in\lim_{n}U(\widetilde{A}_{\leq n})$ that maps to the image of $x$ in $\lim_{n}X(\widetilde{A}_{\leq n})$. But this follows from formal smoothness of $U\rightarrow X$ (see Proposition \ref{cotangent of smooth}) and the fact that the map $A_{\leq n}\rightarrow A_{\leq n-1}$ is a square zero extension (see Lemma \ref{postnikov square zero}).
\end{proof}

\begin{lem}
\label{geometric truncated}
	Let $X\rightarrow Y$ be an $n$-geometric morphism and $A$ be a $k$-truncated animated $R$-algebra. Then $X(A)\rightarrow Y(A)$ is $(n+k+1)$-truncated.
\end{lem}
\begin{proof}
This is a consequence of Lemma \ref{nilcomplete} and analogous to \cite[Cor. 5.3.8]{DAG}.\par
	We have to show that the fiber of $X(A)\rightarrow Y(A)$ is $(n+k+1)$-truncated. By Lemma \ref{nilcomplete} it suffices to show that for all $n\in \NN_{0}$ the map $X(A_{\leq j})\rightarrow X(A_{\leq j-1})\times_{Y(A_{\leq n-1})}Y(A_{\leq j})$ is $(n+j+1)$-truncated, whenever $j\leq k$. But from Lemma \ref{postnikov square zero}  and Lemma \ref{obstruction morph imp formal smooth} the fiber of the previous map is given by the loop of $\Hom_{\MMod_{A}}(L_{X/Y,A}, \pi_{j}A[j])$ which by adjunction is $(n+j+1)$-truncated, since by definition $L_{X/Y}[n]$ is connective\footnote{Note that $\pi_{n+j+1+k}\Hom_{\MMod_{A}}(L_{X/Y,A}, \pi_{j}A[j-1])\cong\pi_{0}\Hom_{\MMod_{A}}(L_{X/Y,A}[n+j+1+k], \pi_{j}A[j+1])$ and since $L_{X/Y,A}[n+j+1+k]\in(\MMod_{A})_{\geq j+1+k}$ and $ \pi_{j}A[j+1]\in(\MMod_{A})_{\leq -j+1}$, we see by definition of the $t$-structures that $\pi_{n+j+1+k}\Hom_{\MMod_{A}}(L_{X/Y,A}, \pi_{j}A[j+1])\cong0$ for $k\geq 1$.}. 
\end{proof}

\begin{lem}
\label{geometric hypercomplete}
	Let $X$ be an $n$-geometric stack, then $X$ is hypercomplete.
\end{lem}
\begin{proof}
This is a direct consequence of Lemma \ref{nilcomplete} and \ref{geometric truncated} and the fact that truncated $\infty$-topoi are automatically hypercomplete. This is analogous to \cite[Cor. 5.3.9]{DAG} but anyway we will explain this.\par
	By Lemma \ref{nilcomplete}, we have $X\simeq X\circ \tau_{\leq n}$ and for any $k$-truncated animated ring $A$, we have that $X(A)$ is $(n+k+1)$-truncated (see Lemma \ref{geometric truncated}), in particular $X\circ\tau_{\leq n}$ is hypercomplete (see \cite[Lem. 6.5.2.9]{HTT}) and since the $\infty$-topos of hypercomplete sheaves has limits, we have that $X$ is hypercomplete.
\end{proof}


\section{The stack of perfect modules}
\label{sec:perf}

In this section, we want to prove that the derived stack of perfect modules is locally geometric. This was already proven in \cite{TVaq} in the model categorical setting and in \cite{AG} in the spectral setting, but we recall the proof in its entirety in our setting.\par
We recall some lemmas needed for the proof, as they will become important later on when analyzing the substacks of derived $F$-zips.
	
\begin{lem}
\label{upper semi-cont qc}
	Let $A$ be a commutative ring and $P$ be a perfect complex of $A$-modules and let $n\in \NN_{0}$. Further, for $k\in \ZZ$ let $\beta_{k}\colon \Spec(A)_{\cl}\rightarrow \NN_{0}$ be the function given by $s\mapsto \dim_{\kappa(s)}\pi_{k}(P\otimes_{A}\kappa(s))$. Then $\beta_{k}^{-1}([0,n])$ is quasi-compact open.
\end{lem}
\begin{proof}
	By \cite[0BDI]{stacks-project} $\beta_{k}$ is upper semi-continuous and locally constructible. As $\Spec(A)_{\cl}$ is affine it is quasi-compact quasi-separated and so we see that $\beta_{k}^{-1}([0,n])$ is quasi-compact open.
\end{proof}

\begin{rem}
\label{tor and type}
	Let $A$ be a commutative ring and $P$ be a perfect complex of $A$-modules. Let $I\subseteq \ZZ$ be a finite subset and for $k\in \ZZ$ let $\beta_{k}$ be as in Lemma \ref{upper semi-cont qc} . Assume that $\beta_i \not =0$ for $i\in I$ and zero everywhere else. Then using \cite[0BCD,066N]{stacks-project}, we see that $P$ has Tor-amplitude in $[\min(I),\max(I)]$.
\end{rem}

\begin{lem}
\label{lem perfect zero}
	Let $A$ be a commutative ring and $P$ be a perfect complex of $A$-modules. Then there exists a quasi-compact open subscheme $U\subseteq\Spec(A)_{\cl}$ with the following property,
	\begin{enumerate}
		\item[$\bullet$] an affine scheme morphism $\Spec(B)_{\cl}\rightarrow \Spec(A)_{\cl}$ factors through $U$ if and only if $P\otimes_{A} B\simeq 0$.
	\end{enumerate}
\end{lem}
\begin{proof}
	Let $\beta_{k}$ be as in Lemma \ref{upper semi-cont qc}. Then we set
	$$
		U\coloneqq\bigcap_{k\in\ZZ} \beta_{k}^{-1}(\lbrace 0 \rbrace ).
	$$
	As $P$ is perfect, so in particular has finite Tor-amplitude, this intersection has only finitely many pieces that are non equal to $\Spec(A)_{\cl}$. Therefore, $U$ is a finite intersection of quasi-compact opens in an affine scheme (see Lemma \ref{upper semi-cont qc}) and thus quasi-compact open.\par 
	Now assume we have a morphism $\Spec(B)_{\cl}\rightarrow \Spec(A)_{\cl}$ such that $P\otimes_{A}B = 0$. Then certainly for any $b\in B$ and all $i\in\ZZ$ we have $\dim_{\kappa(b)}\pi_{i}(P\otimes_{A}B\otimes_{B}\kappa(b))=0$. Let $a\in\Spec(A)$ be the image of $b$. Then for all $i\in\ZZ$ we have the following equalities 
	\begin{align*}
		\dim_{\kappa(b)}\pi_{i}(P\otimes_{A}B\otimes_{B}\kappa(b)) &=\dim_{\kappa(b)}\pi_{i}(P\otimes_{A}\kappa(a)\otimes_{\kappa(a)}\kappa(b))\\
		&=\dim_{\kappa(b)}\pi_{i}(P\otimes_{A}\kappa(a))\otimes_{\kappa(a)}\kappa(b)\\
		&=\dim_{\kappa(a)}\pi_{i}(P\otimes_{A}\kappa(a)),
	\end{align*}
	where we use flatness of field extensions in the second equality. Therefore, we see that $\Spec(B)\rightarrow \Spec(A)$ factors through $U$.\par
	For the other direction assume that $\Spec(B)_{\cl}\rightarrow \Spec(A)_{\cl}$ factors through $U$. Then for any $b\in \Spec(B)$ and any $i\in \ZZ$, we have that $\pi_{i}(P\otimes_{A}B\otimes_{B}\kappa(b)) = 0$. By Remark \ref{tor and type}, we see that $P\otimes_{A}B$ is given by a finite projective module $M$ concentrated in one degree. The fiberwise dimension of $M$ is equal to $0$ by assumption and thus by Nakayama $M=0$.
\end{proof}

The next lemma shows that the vanishing locus of perfect complexes is quasi-compact open. This will be applied to the cofiber of morphisms of perfect complexes. In particular, the locus classifying equivalences between fixed perfect modules is therefore quasi-compact open.

\begin{lem}
\label{equiv zariski open}
Let $A\in\AniAlg{R}$ and $P$ be a perfect $A$-module. Define the derived stack $V_P$ via $V_P(B) =$ full sub-$\infty$-category of $\Hom_{\AniAlg{R}}(A,B)$ consisting of morphisms $u\colon A\rightarrow B$ such that $P\otimes_{A,u}B\simeq 0$. This is a quasi-compact open substack of $\Spec(A)$. 
\end{lem}

\begin{proof}
	This is \cite[Prop. 2.23]{TVaq} translated to our setting. But for the convenience of the reader, we give a proof.
	\par
	Consider $Q\coloneqq P\otimes_A \pi_0A$. Then $Q$ is a perfect complex of $\pi_0A$-modules. Lemma \ref{lem perfect zero} shows that there is a quasi-compact open subscheme $U\subseteq\Spec(\pi_0 A)_{\cl}$, such that for any point $u\colon \Spec(R')_{\cl}\rightarrow \Spec(\pi_0A)_{\cl}$, where $R'$ is a commutative ring, the module $Q\otimes_{\pi_0 A} R'$ is isomorphic to $0$ if and only if $u$ factors through $U$.\par 
	Let $f_1,\dots,f_n\in\pi_0A$, such that the $\pi_0A_{f_i}$ covers $U$. Then $V\coloneqq \Im(\coprod_{i=1}^{n} \Spec(A[f_i^{-1}]))$, the image of $\coprod_{i=1}^{n}\Spec(A[f_{i}^{-1}])\rightarrow \Spec(A)$, is equivalent to $V_P$.\par 
	Indeed, take a morphism $u\colon A\rightarrow B$ in $\AniAlg{R}$. Then $u\in V(B)$ if and only if there exists an $i$ such that $\pi_0u(f_i)$ is \'etale locally invertible in $\pi_0B$ (see Remark \ref{lift along affine}). This is equivalent to $P\otimes_A \pi_0 B\simeq0$ by the choice of the $f_i$. This again is equivalent to $P\otimes_A B\simeq 0$.\par
	To see this, assume the $P\otimes_A B\not\simeq 0$ and take the minimal $i\in\ZZ$ with $\pi_i (P\otimes_A B)\not\simeq 0$. Consider the Tor spectral sequence 
	$$
	E_2^{p,q}=\Tor_p^{\pi_\ast B}(\pi_\ast(P\otimes_A B),\pi_0B)_q\Rightarrow \pi_{p+q}(P\otimes_A \pi_0B).
	$$
	By the definition of the graded tensor product, we see that $E^{0,q}_2 = \pi_q(P\otimes_A B)$ for $q= i$ and $0$ for $q<i$. As explained in Remark \ref{graded free res}, we can choose a graded free resolution of $P\otimes_{A}B$ such that each term of the resolution concentrated in degrees $\geq i$. Therefore, the spectral sequence is concentrated in the quadrant where the left lower corner is at $p=0$ and $q=i$. Thus, $\pi_{i}(P\otimes_{A}B)\cong E_{2}^{0,i}\cong \pi_{i}(P\otimes_{A}\pi_{0}B)\cong 0$ contradicting the assumption.\par 
	The geometricity follows from Lemma \ref{smooth affine geometric}.
\end{proof}

Next, we show that the stack classifying morphisms between perfect modules\footnote{Note that for two perfect $A$-modules $P,Q$ over some animated ring $A$, we have $\Hom_{\MMod_{B}}(P\otimes_{A}B,Q\otimes_{A}B)\simeq \Hom_{\MMod_{A}}(P\otimes Q^{\vee},B)$ as shown in the proof of Lemma \ref{diagonal geometric}.} is actually geometric and in good cases smooth. Since derived $F$-zips will come with two bounded perfect filtrations (i.e. finite chains of morphisms of perfect modules), this lemma is crucial for the geometricity of derived $F$-zips.

\begin{lem}
	\label{spec sym geometric}
	Let $A$ be an animated $R$-algebra. Let $P$ be a perfect $A$-module with Tor-amplitude concentrated in $[a,b]$ with $a\leq 0$.
	Then the derived stack 
	\begin{align*}
		F^{A}_P\colon \AniAlg{A}&\rightarrow \SS\\
		B&\mapsto \Hom_{\MMod_A}(P,B)
	\end{align*}
	is $(-a-1)$-geometric and locally of finite presentation over $\Spec(A)$\footnote{Certainly, we can view $F_{P}^{A}$ as a derived stack over $R$ with a morphism to $\Spec(A)$. So for any animated $R$-algebra $C$ that does not come with a morphism $A\rightarrow C$ the value of $F_{P}^{A}$ is empty.}. Further, the cotangent complex of $F_P$ at a point $x\colon \Spec(B)\rightarrow F^{A}_P$ is given by 
	$$
	L_{F_P,x}\simeq P\otimes_A B.
	$$
	\par 
	In particular, if $b\leq0$, then $F_P$ is smooth.
\end{lem}
\begin{proof}
	Before showing the geometricity, let us calculate the space of derivations of $F_{P}$ and hence the cotangent complex.\par
	Let $x\colon \Spec(B)\rightarrow F^{A}_P$ be a morphism of derived stacks corresponding to a morphism $f\colon P\rightarrow B$ in $\MMod_{A}$ and $M$ be a connective $B$-module. We have that $\Der_{x}(F^{A}_{P}/A,M)$ is given by the fiber of $\Hom_{A}(P,B\oplus M)\rightarrow \Hom_{A}(P,B)$ at $f$. The underlying $R$-module of $B\oplus M$ is per construction the direct sum of the underlying $R$-module of $B$ and of $M$. Therefore, any morphism $P\rightarrow B\oplus M$ is uniquely up to homotopy characterized by a morphism $P\rightarrow B$ and $P\rightarrow M$ and thus, we see that $\Der_{x}(F^{A}_{P}/A,M)\simeq \Hom_{A}(P,M)\simeq \Hom_{A}(P\otimes_{A}B,M)$. Hence, we have $L_{F_{P},x}\simeq P\otimes_{A}B$.\par 
	Let us conclude the rest of the proof, which we will prove by induction on $a$.\par
	If $a= 0$, then $P$ is connective and $F^{A}_P\simeq \Hom_{\AniAlg{A}}(-,\Sym_A P)$ which has the desired properties.\par 
	Now assume $a<0$. By Lemma \ref{general props of Tor} we have $P\simeq \fib(Q\rightarrow M[a+1])$, where $Q$ has Tor-amplitude in $[a+1,b]$ and $M$ is a finite projective $A$-module. Thus we get a fiber sequence 
	$$F^{A}_{M[a+1]}\rightarrow F^{A}_Q\rightarrow F^{A}_P\rightarrow F_{M[a]}.$$
	By induction hypothesis $F_Q$ is $(-a-2)$-geometric and locally of finite presentation. We will see that the map $p\colon F^{A}_Q\rightarrow F^{A}_P$ is an effective epimorphism.\par
	Indeed, note that the above fiber sequence and projectivity of $M$ imply that $\pi_0p$ is surjective and thus $p$ is an effective epimorphism (see Remark \ref{effective epi}).
	\par 
	The diagonal $F^{A}_Q\times_{F^{A}_P}F^{A}_Q$ is given by $F^{A}_{Q\oplus_P Q}$\footnote{Here $Q\oplus_{P}Q$ is defines as the pushout of the morphism $P\rightarrow Q$ with itself.}, which will be $(-a-2)$-geometric with smooth projections to $F^{A}_Q$.\par 
	To see this, note that we have a fiber sequence $Q\rightarrow Q\oplus_PQ\rightarrow M[a+1]$ which has a retract. Thus the natural map $Q\oplus_PQ\rightarrow Q\oplus M[a+1]$ is an equivalence on the level of homotopy groups by the splitting lemma (the induced exact sequences are short exact, using the retract) and therefore $Q\oplus_PQ\simeq Q\oplus M[a+1]$. Hence, $F^{A}_{Q\oplus_P Q}\simeq F^{A}_Q\times F^{A}_{M[a+1]}$, which is the pullback of $(-a-2)$-geometric stack and thus itself geometric. \par
	Also the projection to $F^{A}_Q$ is smooth, because $F^{A}_{M[a+1]}$ is smooth (the smoothness of $F^{A}_{M[a+1]}$ follows since $L_{F^{A}_{M[a+1]},x}\simeq M[a+1]\otimes_A B$ at a point $x\colon \Spec(B)\rightarrow F^{A}_{M[a+1]}$ and thus has Tor-amplitude in $[a+1,0]$, which concludes (see Corollary \ref{cotangent implies smooth})).
	\par
	By Proposition \ref{diag + proj geometric}, we see that $F^{A}_P\rightarrow\Spec(A)$ is a quasi-compact $(-a-1)$-geometric stack locally of finite presentation.\par 
	\par 
	If $b\leq 0$, then $L_{F^{A}_P}$ is perfect with Tor-amplitude concentrated in degree $[a,0]$ and therefore $F^{A}_P$ is smooth by Corollary \ref{cotangent implies smooth}.	
\end{proof}

\begin{rem}
	A variant of Lemma \ref{spec sym geometric} in the spectral setting can be found in \cite[Thm. 5.2]{AG}. Alternatively, one can look at the proofs given in \cite[Lem. 3.9]{TVaq} and \cite[3.12]{TVaq} to construct a proof in the model categorical setting.
\end{rem}

\begin{lem}
	\label{diagonal geometric}
	The diagonal map $\PPerf^{[a,b]}\rightarrow \PPerf^{[a,b]}\times_{R}\PPerf^{[a,b]}$ is $(b-a)$-geometric and locally of finite presentation.
\end{lem}
\begin{proof}
	This is part of the proof of \cite[Thm. 5.6]{AG} translated to our setting. For the convenience of the reader, we give a proof.\par
	A morphism $\Spec(A)\rightarrow  \PPerf^{[a,b]}\times_{R}\PPerf^{[a,b]}$ corresponds to two perfect modules $P,Q$ with Tor-amplitude concentrated in $[a,b]$. The pullback under the diagonal classifies equivalences  between $P$ and $Q$. This is an open, $0$-geometric substack of 
	\begin{align*}
	\Hom_{\MMod_A}(P\otimes_A Q^\vee,-) \simeq \Hom_{\MMod_A}(P,Q\otimes_A-)
	\simeq \Hom_{\MMod_-}(P\otimes_A -,Q\otimes_A-),
	\end{align*} (note that perfect modules are dualizable).\par 
	To see this, note that for any morphism $\Spec(B)\rightarrow \Hom_{\MMod_A}(P\otimes_A Q^\vee,-)$, given by a morphism $\varphi\colon P\otimes_A B\rightarrow Q\otimes_A B$, the stack $\Equiv(P,Q)\times_{\Hom_{\MMod_A}(P\otimes_A Q^\vee,-)}\Spec(B)$ classifies morphisms $u\colon B\rightarrow C$, where $\cofib\varphi\otimes_{B,u} C\simeq 0$, which is an open, $0$-geometric substack of $\Spec(B)$ by Lemma \ref{equiv zariski open}.
	\par 
	Now $P\otimes Q^\vee$ is a perfect module of Tor-dimension $[a-b,b-a]$ (see Lemma \ref{general props of Tor}) and thus Lemma \ref{spec sym geometric} concludes the proof.
\end{proof}

\begin{defi}
	Let $n\in \NN$ and $A\in\AniAlg{R}$. We denote the $\infty$-category of finite projective $A$-modules of rank $n$ with $\BGL_{n}(A)$. 
\end{defi}

\begin{lem}
	Let $n\in \NN$. The functor $A\mapsto \BGL_{n}(A)$ from $\AniAlg{R}$ to $\ICat$ satisfies fpqc descent.
\end{lem}
\begin{proof}
	We already know that modules satisfy decent so it is enough to check that an $A$-module $M$ is finite projective of rank $n$ if it is after base change to an fpqc-cover $(A\rightarrow A_{i})_{i\in I}$. Note that $\pi_{0}A\rightarrow \pi_{0}A_{i}$ is faithfully flat for every $i\in I$. Now assume that $M\otimes_{A}A_{i}$ is finite projective of rank $n$. Then it is in particular flat and we will show first that $M$ is flat over $A$. By flatness, the natural map $$\pi_{j}A_{i}\otimes_{\pi_{0}A}\pi_{0}M
\cong \pi_{j}A_{i}\otimes_{\pi_{0}A_{i}}\pi_{0}M\otimes_{\pi_{0}A}\pi_{0}A_{i}\cong\pi_{j}A_{i}\otimes_{\pi_{0}A_{i}}\pi_{0}(M\otimes_{A}A_{i})\rightarrow \pi_{j}(M\otimes_{A}A_{i})$$
is an equivalence. By flatness of $A\rightarrow A_{i}$, we have $\pi_{j}A_{i}\cong \pi_{j}A\otimes_{\pi_{0}A}\pi_{0}A_{i}$. Hence, we have $\pi_{j}A_{i}\otimes_{\pi_{0}A}\pi_{0}M\cong \pi_{j}A\otimes_{\pi_{0}A}\pi_{0}A_{i}\otimes_{\pi_{0}A}\pi_{0}M$. By faithfully flatness the map $\pi_{j}A\otimes_{\pi_{0}A}\pi_{0}M\rightarrow \pi_{j}M$ is an equivalence if and only if it so after base change to $\pi_{0}A_{i}$ for all $i\in I$. But the above shows that this base change gives the map $$\pi_{j}A\otimes_{\pi_{0}A}\pi_{0}A_{i}\otimes_{\pi_{0}A}\pi_{0}M\rightarrow \pi_{j}(M\otimes_{A}A_{i})$$ and flatness of $A_{i}$ over $A$ shows that $\pi_{j}(M\otimes_{A}A_{i})\cong \pi_{j}M\otimes_{\pi_{0}A}A_{i}$ (see \cite[Prop. 7.2.2.13]{HA}) so indeed, $\pi_{j}A\otimes_{\pi_{0}A}\pi_{0}M\rightarrow \pi_{j}M$ is an equivalence. Therefore, $M$ is flat over $A$. \par 
Now a flat module is finite projective of rank $n$ if it is so on $\pi_{0}$ (see Lemma \ref{flat proj} and use the definition of finite projectiveness) but this follows from classical faithfully flat descent.
\end{proof}

\begin{remark}
	The inclusion of $\SS\hookrightarrow \ICat$ is left adjoint to the functor $(-)^{\simeq}$ that passes to the largest Kan complex contained in an $\infty$-category (see \cite[Prop. 1.2.5.3]{HTT}). Therefore if $F\colon \Ccal\rightarrow \ICat$ is a (hypercomplete) sheaf for some Grothendieck topology on $\Ccal$ then also $F^{\simeq}\coloneqq(-)^{\simeq}\circ F$ is one.
\end{remark}

\begin{defi}
	We define the derived stack classifying vector bundles as the stack
	\begin{align*}
		\BGL_{n,R}\colon \AniAlg{R}&\rightarrow \SS\\
		A&\mapsto \BGL_{n}(A)^{\simeq}.
	\end{align*}
	Further, we denote by $\GL_{n,R}$ the loop under the map $\Spec(R)\rightarrow \BGL_{n,R}$, which is given for an animated $R$-algebra $A$ by $\ast\mapsto A^{n}$, i.e. we have the following pullback diagram in $\dSt_{R}$
	$$
	\begin{tikzcd}
		\GL_{n,R}\arrow[r,""]\arrow[d,""]& \Spec(R)\arrow[d,""]\\
		\Spec(R)\arrow[r,""]&\BGL_{n,R}
	\end{tikzcd}
	$$
	(note that for a commutative ring $A$, we have that $\BGL_{n,R}(A)$ is the groupoid of rank $n$ vector bundles on $A$ and thus $\GL_{n,R}(A)$ is indeed given by the points of the general linear group scheme of rank $n$).
\end{defi}

\begin{lem}
	\label{proj geometric}
	Let $\Proj_R$ denote the derived stack classifying finite projective modules. Then $\Proj_R\simeq \coprod_{n\in\NN} \BGL_{n,R}$, in particular $\Proj_R$ is $1$-geometric and smooth.\par 
	Further, $\GL_{n,R}$ is an affine derived scheme and $\Proj_R$ has an affine diagonal.
\end{lem}
\begin{proof}
	The proof is the same as \cite[Cor. 1.3.7.12]{TV2} but for the convenience of the reader, we give a sketch.\par
	That $\Proj_R\simeq\coprod_{n\in \NN} \BGL_{n,R}$, where $\BGL_{n,R}$ denotes the stack of finite projective modules of rank $n$, is clear. So it suffices to show that $\BGL_{R,n}$ is a $1$-geometric smooth stack. It is enough to show that $\GL_{n,R}\rightarrow \Spec(R)$ is $0$-geometric smooth. Then $\GL_{n,R}\rightarrow \Spec(R)$ defines a $0$-Segal groupoid (see \cite[Def. 1.3.4.1]{TV2}) and $\BGL_{R,n}$ is $1$-geometric (see \cite[Prop. 1.3.4.2]{TV2}). That $\BGL_{n,R}$ is smooth follows from the fact that the natural morphism $\Spec(R)\rightarrow \BGL_{n,R}$ gives a $1$-atlas.\par 
	The claim about $\GL_{n,R}$ follows in the following way. The stack $\GL_{n,R}$ is equivalent to the stack classifying automorphisms of $R^n$, i.e. $\GL_{n,R}(A)\simeq \Equiv_{A}(A^n)$, for $A\in \AniAlg{R}$. But by Lemma \ref{equiv zariski open} this is a $0$-geometric open substack of $F_{R^{n^2}}\simeq \Spec(\Sym_R(R^{n^2}))$, which is a $(-1)$-geometric smooth stack.\par 
	Alternatively, one could follow \cite[Prop. 1.3.7.10]{TV2} and show directly that the inclusion $\iota\colon \GL_{n,R}\hookrightarrow F_{R^{n^2}}$ is representable and \'etale by showing that for any point $x\in F_{R^{n^2}}(A)$, we have $\iota^{-1}(x)\simeq \Spec(A[\det(x)^{-1}])$. In particular, we see in this way that $\GL_{n,R}$ is representable by an affine derived scheme. This also shows that $\BGL_{n,R}$ has an affine diagonal, since $\Spec(R)\times_{\Spec(R)}\Spec(R)\rightarrow \BGL_{n,R}\times_{\Spec(R)}\BGL_{n,R}$ is a $1$-atlas by the above and so, we have a pullback diagram of the form
	$$
	\begin{tikzcd}
		\GL_{n,R}\arrow[r,""]\arrow[d,""]& \Spec(R)\simeq\Spec(R)\times_{\Spec(R)}\Spec(R)\arrow[d,""]\\
		\BGL_{n,R}\arrow[r,"\Delta"]&\BGL_{n,R}\times_{\Spec(R)}\BGL_{n,R},
	\end{tikzcd}
	$$
	which shows that the diagonal is affine, since this can be tested after passing to a cover by affines (see \cite[Lem. 1.3.2.8]{TV2}).
\end{proof}

\begin{rem}
\label{kan fib pull}
	For the proof of the next theorem, we want to remark some generalities about pullbacks of Kan complexes.\par 
	Let $X,Y$ be Kan complexes, i.e. elements in $\SS$, and assume we have a morphism $X\rightarrow\Fun(\del\Delta^{1},Y)$ in $\SS$. We want to compute the following pullback in $\SS$
	$$
	\begin{tikzcd}
		W\arrow[r,""]\arrow[d,""]& \Fun(\Delta^{1},Y)\arrow[d,"i"]\\
		X\arrow[r,""]&\Fun(\del\Delta^{1},Y),
	\end{tikzcd}
	$$ 
	where $i$ is given by the restriction.
	In general this is a pullback in the $\infty$-categorical sense, which can also be computed on the level of model categories via the homotopy pullback (recall that $\SS$ is the $\infty$-category associated to the model category of simplicial sets with the usual model structure (weak equivalences are given by weak equivalences of the underlying Kan complexes and fibrations are Kan fibrations)). We claim that the homotopy pullback of the underlying Kan complexes is equivalent in $\SS$ to the ordinary pullback of simplicial sets if $i$ is a Kan fibration (i.e. $i$ is a fibration in the model category of simplicial sets and since each simplicial set involved is already a Kan complex they are per definition fibrant).\par
	Indeed, this is a classical result and remarked in \cite[Rem. A.2.4.5]{HTT} but we will shortly sketch the idea behind it. Let $\Abf$ be a combinatorial model category (e.g. the category of simplicial sets with the model structure explained above) and let $I$ be the diagram category given by three objects $\lbrace 0,1,2\rbrace$ together with morphisms $0\rightarrow 2$, $1\rightarrow 2$ and identities. One can attach the injective model structure onto the functor category $\Fun(I,\Abf)$ by defining weak equivalences to be pointwise weak equivalences and defining cofibrations also pointwise (the fibration are then given by certain lifting properties, which we will not discuss on detail). The homotopy limit is defined as the right Quillen adjoint of the constant functor $\Abf\rightarrow  \Fun(I,\Abf)$, which is weakly equivalent to the the limit of a \textit{fibrant} diagram in $\Fun(I,\Abf)$, i.e. for a diagram $a\rightarrow c\leftarrow b$, the homotopy limit is defined as an object that is weakly equivalent to ordinary limit of $a'\rightarrow c'\leftarrow b'$, where $a'\rightarrow c'$ and $b'\rightarrow c'$ are fibrations, $c'$ is fibrant and we have a commutative diagram of the form
	$$
	\begin{tikzcd}
		a\arrow[r,""]\arrow[d,""]& c\arrow[d,""]&b\arrow[d,""]\arrow[l,""]\\
		a'\arrow[r,""]&c'&b'\arrow[l,""],
	\end{tikzcd}
	$$
	where the vertical arrows are weak equivalences (this is analogous to the theory of homotopy pushouts, which can be found in \cite[\S A.2.4]{HTT}). Now let $I_{R}$ be the full subcategory of $I$ generated by the elements $0$ and $2$ and define $I_{L}$ as the full subcategory of $I$ generated by the elements $1$ and $2$. The tuple $(I_{L},I_{R})$ makes $I$ into a Reedy category on $I$ (see \cite[\S A.2.9]{HTT} for more on Reedy categories). As explained in \cite[Proposition A.2.9.19]{HTT} there is a model structure on $\Fun(I,\Abf)$ corresponding to the Reedy structure on $I$ called the Reedy model structure. Important for us is that a diagram $a\rightarrow c\leftarrow b$ is fibrant for the Reedy model structure if $a,b$ are fibrant and either $a\rightarrow c$ or $b\rightarrow c$ is a fibration. Further, as remarked in \cite[Rem. A.2.9.23]{HTT} the Reedy model structure and the injective model structure are Quillen equivalent via the identity functor. Therefore, if we have a fibrant diagram with respect to the Reedy model structure, then the homotopy limit is per definition weak equivalent to the ordinary limit in $\Abf$. Since the homotopy limit in this case is the homotopy pullback, we are done.\par

But that $i$ is a Kan fibration follows from \cite[Cor. 3.1.3.3]{kerodon} and hence we can compute the above pullback in the $\infty$-category $\SS$ via the limit of the underlying simplicial sets. \par 
	Now assume that $Y$ is an arbitrary $\infty$-category. We want to compute the following pullback in $\SS$
	$$
	\begin{tikzcd}
		W\arrow[r,""]\arrow[d,""]& \Fun(\Delta^{1},Y)^{\simeq}\arrow[d,"i"]\\
		X\arrow[r,""]&\Fun(\del\Delta^{1},Y)^{\simeq},
	\end{tikzcd}
	$$ 
	where $i$ is naturally given by applying the functor $(-)^{\simeq}$ to the restriction.
	If $i$ is a Kan fibration, then we can apply the argument above. In general it may not be clear if $i$ is a Kan fibration. But in this case, we have that the natural morphism $F\colon \Fun(\Delta^{1},Y)\rightarrow\Fun(\del\Delta^{1},Y)$ is an isofibration of $\infty$-categories (see \cite[01F3]{kerodon}), meaning that it is an inner fibration on the level of homotopy categories, we have the following property: if $x\in h\Fun(\Delta^{1},Y)$ and we have an isomorphism $u'\colon y\xrightarrow{\sim} F(x)$ then, there exists a $x'\in h\Fun(\Delta^{1},Y)$ with an isomorphism $u\colon x'\rightarrow  x$ such that $F(u) = u'$. In particular, \cite[Prop. 4.4.3.7]{kerodon} implies that $i$ is a Kan fibration.
\end{rem}

\begin{theorem}
	\label{main thm}
	The derived stack
	\begin{align*}
	\PPerf_R\colon \AniAlg{R} &\rightarrow \SS\\
	A&\mapsto (\MMod_A^{\textup{perf}})^{\simeq}
	\end{align*}
	is locally geometric and locally of finite presentation.\par
	To be more specific, we can write $\PPerf_R = \colim_{a\leq b} \PPerf_R^{[a,b]}$, where $\PPerf_R^{[a,b]}$ is the moduli space consisting of perfect modules which have Tor-amplitude concentrated in degree $[a,b]$ and each $\PPerf_R^{[a,b]}$ is $(b-a+1)$-geometric and locally of finite presentation and the inclusion $\PPerf_{R}^{[a,b]}\hookrightarrow \PPerf_{R}$ is a quasi-compact open immersion. If $b-a\leq 1$ then $\PPerf_R^{[a,b]}$ is in fact smooth. 
\end{theorem}
\begin{proof}
	The proof in the model categorical setting can be found in \cite[Prop. 3.7]{TVaq} and in the spectral setting in \cite[Thm. 5.6]{AG}. The latter follows the former with few changes for readability. We will follow the proof presented in the latter using our setting.\par
	We show that $\PPerf^{[a,b]}$ is $n+1$-geometric, where $n=b-a$, by induction over $n$.\par 
	For $n=0$, we are done, since then we have $\Proj\simeq \PPerf^{[a,a]}$ (see Lemma \ref{general props of Tor}), which is $1$-geometric and locally of finite presentation by Lemma \ref{proj geometric}.\par 
	Now let $n>0$ and assume $\PPerf^{[a+1,b]}$ is $n$-geometric and locally of finite presentation. Let $U$ be defined via the pullback diagram of derived stacks
	$$
	\begin{tikzcd}
	U\arrow[r,""]\arrow[d,"p", swap]&\Fun(\Delta^1,\MMod^{\perf})^{\simeq}\arrow[d,""]\\
	\PPerf^{[a+1,b]}\times_R \PPerf^{[a+1,a+1]}\arrow[r,""]&\Fun(\del\Delta^1,\MMod^{\perf})^{\simeq}.
	\end{tikzcd}
	$$
	Let $\Spec(A)\rightarrow \PPerf^{[a+1,b]}\times_R \PPerf^{[a+1,a+1]}$ be given by $(P,Q)$, where $P$ is a perfect $A$-module of Tor-amplitude $[a+1,b]$ and $Q$ is the $a+1$ shift of a finite projective $A$-module, then $p^*(P,Q)$ classifies morphisms between those, i.e. $p^*(P,Q)\simeq \Spec(\Sym(P\otimes_AQ^\vee))$ (note that $P\otimes_AQ^\vee$ is perfect and has Tor-amplitude in $[0,b-(a+1)]$ and thus is connective (see Lemma \ref{general props of Tor}). Therefore $p$ is $(-1)$-geometric and locally of finite presentation and with Lemma \ref{diagonal geometric} and \ref{geometric comp}, we see that $U$ is $n$-geometric and locally of finite presentation. Note that if $b-a\leq 1$ then $p$ is even smooth and using that $\PPerf^{[a+1,a+1]}$ is smooth, we see that $U$ is smooth.\par
	By sending a morphism to its fiber, we get a morphism of derived stacks $q\colon U\rightarrow \PPerf^{[a,b]}$.
	Using Proposition \ref{diag geometric} with Lemma \ref{diagonal geometric}, we see that $q$ is $n$-geometric, so it suffices to show that $q$ is also smooth and an effective epimorphism.\par 
	That it is an effective epimorphism follows from Lemma \ref{general props of Tor}. To check smoothness let $\Spec(A)\rightarrow \PPerf^{[a,b]}$ be a morphism classified by a perfect $A$-module $P$ with Tor-amplitude in $[a,b]$. Then $q^{-1}(P)(B)$, for some animated $A$-algebra $B$, consists of morphisms of perfect $B$-modules $f\colon Q\rightarrow M[a+1]$ such that $\fib(f)\simeq P\otimes_A B$, where $Q$ has Tor-amplitude in $[a+1,b]$ and $M$ is finite projective. Since locally every finite projective module is free of finite rank, we can decompose 
	$$
	q^{-1}(P)\simeq \coprod_m q^{-1}(P)_m,
	$$
	where $q^{-1}(P)_m$  is the substack of $q^{-1}(P)$, where the classified morphisms have codomain given by the $a+1$ shift of free modules of rank $m$. The stack $q^{-1}(P)_m$ is equivalent to the stack classifying morphisms $ A^m[a]\rightarrow P$, where the cofiber has Tor-amplitude in $[a+1,b]$, which is equivalent to $A^m[a]\rightarrow P$ beeing a surjection on $\pi_a$.
	\par 
	To see the equivalence of stacks let us look at $q^{-1}(P)_m(B)$. These are all morphisms $f\colon Q\rightarrow B^m[a+1]$ such that $\fib(f)\simeq P\otimes_A B$, again $Q$ is a perfect $B$-module with Tor-amplitude in $[a+1,b]$. Since $\MMod_A$ is stable, we see that $P\otimes_AB\rightarrow Q\rightarrow B^{m}[a+1]$ is a fiber diagram if and only if it is a cofiber diagram and thus after shift we see that $q^{-1}(P)_m(B)$ consists of morphisms $g\colon B^{m}[a]\rightarrow P\otimes_AB$ such that its cofiber $\cofib(g)$ has Tor-amplitude in $[a+1,b]$. By Lemma \ref{general props of Tor} $\cofib(g)$ has Tor-amplitude in $[a,b]$. Since after tensoring $M^m[a]\rightarrow P\otimes_A M\rightarrow \cofib(g)\otimes_B M$, where $M$ is a discrete $\pi_0B$-module, we still have a cofiber sequence it is enough to check that $\pi_a(g\otimes\id_M)$ is a surjection. But since $\pi_a(P\otimes_AM) = \pi_a(P) \otimes_{\pi_0A} M$ (use the degeneracy of the Tor-spectral sequence at $(0,a)$) and the ordinary tensor product of $\pi_0A$-modules preserves surjections it is enough to check that $\pi_ag$ is surjective. And therefore $q^{-1}(P)_m(B)$ consists of morphisms $B^m[a]\rightarrow P\otimes_A B$, which are surjective on $\pi_a$ (obviously any morphism $B^m[a]\rightarrow P\otimes_A B$ with cofiber having Tor-amplitude in $[a+1,b]$ is surjective on $\pi_a$).\par 	
	By this characterization the stack $q^{-1}(P)_m$ is an open substack of $F^{A}_{(P^{\vee})^{m}[a]}$ (see Lemma \ref{spec sym geometric} for notation).\par
	To see this, let $\Spec(B)\rightarrow F_{(P^{\vee})^{m}[a]}$ be given by a morphism $\xi\colon B^m[a]\rightarrow P\otimes_A B$. Let $Z$ be the pullback of $\Spec(B)$ along the inclusion $q^{-1}(P)_m\hookrightarrow F_{(P^{\vee})^{m}[a]}$. In particular, for any animated $A$-algebra $C$, we have that $Z(C)$ consists of those morphisms $f\colon B\rightarrow C$, such that $\pi_af^\ast\xi$ is surjective.
	Since $P\otimes_A B$ is perfect and has Tor-amplitude in $[a,b]$ its homotopy group $\pi_a(P\otimes_A B)$ is finitely presented (see \cite[Cor. 7.2.4.5]{HA}). Therefore, being surjective is an open condition on $\pi_0B$ (see \cite[Prop. 8.4]{WED}). Further refining by principal affine opens $D(f_i)\subseteq\Spec(\pi_0B)$, we get an open substack $\bigcup \Spec(B[f_i^{-1}])$ of $\Spec(B)$. Now a morphism $u\colon B\rightarrow C$ is in $Z(C)$ if and only if \'etale locally there is an $i$ such that $\pi_0u(f_i)$ is invertible. Therefore, $Z\simeq \bigcup\Spec(B[f_i^{-1}])$.\par 
	Since $q^{-1}(P)_m$ is open in $F^{A}_{(P^{\vee})^{m}[a]}$, which itself is smooth by Lemma \ref{spec sym geometric}, we see that $q^{-1}(P)$ is smooth over $A$, which concludes the proof.
\par 
Indeed, let $P$ be a perfect $A$-module with Tor-amplitude in $[a,b]$. Then, by Lemma \ref{general props of Tor}, we can find a cofiber sequence $M[a]\rightarrow P\rightarrow Q$, where $Q$ is perfect of Tor-amplitude $[a+1,b]$ and $M$ is finite projective. Analogous to the above, we see that $P$ has Tor-amplitude in $[a+1,b]$.\par 
For the open immersion part it suffices by induction to show that for all $a< b\in \ZZ$ the inclusion $\PPerf_{R}^{[a+1,b]}\hookrightarrow \PPerf_{R}^{[a,b]}$ is an open immersion. Let $A\in\AniAlg{R}$ and $\Spec(A)\rightarrow\PPerf_{R}^{[a,b]}$ be a morphism classified by a perfect $A$-module $P$ of Tor-amplitude $[a,b]$. By Lemma \ref{general props of Tor}, we have a fiber sequence of $A$-modules $P\rightarrow M[a+1]\rightarrow Q$, where $Q$ is perfect of Tor-amplitude in $[a+1,b]$ and $M$ is finite projective. Now $P$ has Tor-amplitude in $[a+1,b]$ if and only if $M\simeq 0$.  But by Lemma \ref{equiv zariski open}, we see that the vanishing locus of $M$ is a quasi-compact open in $\Spec(A)$, which concludes the proof.
\end{proof}

\begin{cor}
	\label{cotangent of perf}
	Let $A$ be an animated $R$-algebra and let $\Spec(A)\rightarrow \PPerf_R$ be a morphism given by a perfect $A$-module $P$. The cotangent complex $L_{\PPerf_R,A}$ is perfect and if $A/R$ is \'etale then the cotangent complex a that point is given by
	$$
		L_{\PPerf_R,A}\simeq (P\otimes_A P^{\vee})^{\vee}[-1].
	$$
\end{cor}
\begin{proof}
	This is analogous to \cite[Cor. 5.9]{AG}, but for the convenience of the reader we recall the proof.\par 
	The first assertion follows from Theorem \ref{main thm} with Corollary \ref{global lfp cotangent}.\par
	For the second assertion let $\Omega_P\PPerf_R$ denote the loop of $\PPerf_R$ along $\Spec(A)\rightarrow \PPerf_R$ classified by a perfect $A$-module $P$, i.e. the $\Omega_P\PPerf_R\coloneqq \Spec(A)\times_{\PPerf_{R}}\Spec(A)$. We know per definition of the cotangent complex that we have a pullback diagram
	$$
	\begin{tikzcd}
	L_{\PPerf,A} \arrow[r]\arrow[d] & L_{A/R}\arrow[d]\\
	L_{A/R} \arrow[r]&L_{\Omega_P\PPerf_R,*},
	\end{tikzcd}
	$$ 
	where $*$ is the point corresponding to the canonical map $\Spec(A)\rightarrow \Omega_P\PPerf_R$.\par
	Let $T\in \AniAlg{A}$, then $\Omega_P\PPerf_R(T)\simeq \Equiv_T(P\otimes_A T)$, where $   \Equiv_T(P\otimes_A T)$ denotes the $T$-automorphisms of $P\otimes_A T$. In particular $\Omega_P\PPerf$ is an open substack of $$\Hom_A((P\otimes_A P^{\vee})^\vee,-)$$ (after using adjunctions). By Lemma \ref{spec sym geometric}, we now have $$ L_{\Omega_P\PPerf_R,\ast}\simeq (P\otimes_A P^{\vee})^{\vee}.$$\par 
	Therefore, if $A/R$ is \'etale, we have $\Sigma L_{\PPerf_R,A}\simeq L_{\Omega_P\PPerf_R,*}$, whe finishing the proof.
\end{proof}

\bibliographystyle{alpha}
\bibliography{NotesDAG}

\begin{thebibliography}{MLM94}

\bibitem[AG12]{ArinkinGaits}
Dima Arinkin and Dennis Gaitsgory.
\newblock Singular support of coherent sheaves, and the geometric langlands
  conjecture, 2012.

\bibitem[AG14]{AG}
Benjamin Antieau and David Gepner.
\newblock Brauer groups and \'{e}tale cohomology in derived algebraic geometry.
\newblock {\em Geom. Topol.}, 18(2):1149--1244, 2014.

\bibitem[Ann18]{Annala}
Toni Annala.
\newblock Bivariant derived algebraic cobordism, 2018.

\bibitem[Avr99]{Avramov}
Luchezar~L. Avramov.
\newblock Locally complete intersection homomorphisms and a conjecture of
  {Q}uillen on the vanishing of cotangent homology.
\newblock {\em Ann. of Math. (2)}, 150(2):455--487, 1999.

\bibitem[CS21]{CS}
Kestutis Cesnavicius and Peter Scholze.
\newblock Purity for flat cohomology, 2021.

\bibitem[GR17]{GR}
Dennis Gaitsgory and Nick Rozenblyum.
\newblock {\em A study in derived algebraic geometry. {V}ol. {I}.
  {C}orrespondences and duality}, volume 221 of {\em Mathematical Surveys and
  Monographs}.
\newblock American Mathematical Society, Providence, RI, 2017.

\bibitem[GW10]{WED}
Ulrich G\"{o}rtz and Torsten Wedhorn.
\newblock {\em Algebraic geometry {I}}.
\newblock Advanced Lectures in Mathematics. Vieweg + Teubner, Wiesbaden, 2010.
\newblock Schemes with examples and exercises.

\bibitem[HSS00]{SymSpec}
Mark Hovey, Brooke Shipley, and Jeff Smith.
\newblock Symmetric spectra.
\newblock {\em J. Amer. Math. Soc.}, 13(1):149--208, 2000.

\bibitem[Kha18]{Khan}
Adeel Khan.
\newblock Lecture notes: Descent in algebraic k-theory.
\newblock \url{https://www.preschema.com/lecture-notes/kdescent/}, WS
  2017/2018.

\bibitem[Lur04]{DAG}
Jacob Lurie.
\newblock {\em Derived algebraic geometry}.
\newblock ProQuest LLC, Ann Arbor, MI, 2004.
\newblock Thesis (Ph.D.)--Massachusetts Institute of Technology.

\bibitem[Lur09]{HTT}
Jacob Lurie.
\newblock {\em Higher topos theory}, volume 170 of {\em Annals of Mathematics
  Studies}.
\newblock Princeton University Press, Princeton, NJ, 2009.

\bibitem[Lur17]{HA}
Jacob Lurie.
\newblock Higher algebra.
\newblock \url{https://www.math.ias.edu/~lurie/papers/HA.pdf}, 2017.

\bibitem[Lur18]{SAG}
Jacob Lurie.
\newblock Spectral algebraic geometry.
\newblock \url{https://www.math.ias.edu/~lurie/papers/SAG-rootfile.pdf}, 2018.

\bibitem[Lur21]{kerodon}
Jacob Lurie.
\newblock Kerodon.
\newblock \url{https://kerodon.net}, 2021.

\bibitem[MLM94]{MLM}
Saunders Mac~Lane and Ieke Moerdijk.
\newblock {\em Sheaves in geometry and logic}.
\newblock Universitext. Springer-Verlag, New York, 1994.
\newblock A first introduction to topos theory, Corrected reprint of the 1992
  edition.

\bibitem[PV15]{PV}
Mauro Porta and Gabriele Vezzosi.
\newblock Infinitesimal and square-zero extensions of simplicial algebras,
  2015.
\newblock arXiv:1310.3573.

\bibitem[Qui67]{Quillen}
Daniel~G. Quillen.
\newblock {\em Homotopical algebra}.
\newblock Lecture Notes in Mathematics, No. 43. Springer-Verlag, Berlin-New
  York, 1967.

\bibitem[Sim96]{Simpson}
Carlos Simpson.
\newblock Algebraic (geometric) $n$-stacks, 1996.
\newblock arXiv:alg-geom/9609014.

\bibitem[SS03]{SS}
Stefan Schwede and Brooke Shipley.
\newblock Stable model categories are categories of modules.
\newblock {\em Topology}, 42(1):103--153, 2003.

\bibitem[{Sta}19]{stacks-project}
The {Stacks Project Authors}.
\newblock \textit{Stacks Project}.
\newblock \url{https://stacks.math.columbia.edu}, 2019.

\bibitem[TV07]{TVaq}
Bertrand To\"en and Michel Vaqui\'e.
\newblock Moduli of objects in dg-categories.
\newblock {\em Annales scientifiques de l'\'Ecole Normale Sup\'erieure}, Ser.
  4, 40(3):387--444, 2007.

\bibitem[TV08]{TV2}
Bertrand To\"{e}n and Gabriele Vezzosi.
\newblock Homotopical algebraic geometry. {II}. {G}eometric stacks and
  applications.
\newblock {\em Mem. Amer. Math. Soc.}, 193(902):x+224, 2008.

\bibitem[Wil69]{universe}
N.~H. Williams.
\newblock On grothendieck universes.
\newblock {\em Compositio Mathematica}, 21(1):1--3, 1969.

\bibitem[Yay22a]{yaylali}
Can Yaylali.
\newblock Derived {$F$}-zips, 2022.

\bibitem[Yay22b]{thesis}
Can Yaylali.
\newblock {\em Derived $F$-zips}.
\newblock PhD thesis, Technische Universit{\"a}t Darmstadt, Darmstadt, 2022.

\end{thebibliography}

\end{document}